\documentclass[matann2]{svjour}

\usepackage{amsmath}
\usepackage{amsfonts}
\usepackage{latexsym}
\usepackage{mathrsfs}

\usepackage{mathptmx} 

\usepackage{array}
\usepackage{amscd}
\usepackage{a4}
\usepackage[mathcal]{euscript}
\usepackage{amssymb}
\usepackage{pstricks}
\usepackage{amsrefs} 

\input{boxedeps} 
\SetEPSFDirectory{} 
\HideDisplacementBoxes
\SetRokickiEPSFSpecial  
\DeclareMathAlphabet{\ams}{U}{msb}{m}{n}
\DeclareMathAlphabet{\goth}{U}{euf}{m}{n}

\def\so{\mathbf{SO}}

\def\sl{\mathbf{SL}}

\def\gl{\mathbf{GL}}

\def\m{\mathbf{M}}
\def\mso{\mathbf{MSO}}
\def\msp{\mathbf{MSp}}

\def\t{\mathbf{T}}

\def\sp{\mathbf{Sp}}

\def\id{\text{id}}

\def\rk{\text{rk}}
\def\diag{\text{diag}}

\def\ov{\overline}
\def\ul{\underline}

\def\wh{\widehat}

\def\dom{\text{dom}}
\def\codim{\text{codim}\,}
\def\opp{\text{opp}}

\def\FF{\mathscr F}
\def\BB{\mathscr B}
\def\AA{\mathscr A}
\def\cA{\mathscr A}

\def\II{\mathscr I}

\def\HH{\mathcal H}
\def\cH{\mathcal H}

\def\TT{\mathcal T}

\def\cS{\mathcal S}

\def\SS{\goth{S}}

\def\XXX{\goth{X}}

\def\ve{\varepsilon}

\def\aa{\alpha}

\def\bb{\beta}
\def\ss{\sigma}

\def\ve{\varepsilon}

\def\wh{\widehat}

\def\Z{\ams{Z}}
\def\R{\ams{R}}
\def\C{\ams{C}}
\def\F{\ams{F}}

\def\0{\mathbf{0}}
\def\1{\mathbf{1}}

\def\quo{/\kern -.45em\sim}

\newpsobject{showgrid}{psgrid}{subgriddiv=1,griddots=10,gridlabels=6pt,gridcolor=red}

\def\Langle{\langle\kern -2pt\langle}
\def\Rangle{\rangle\kern -1.9pt\rangle}

\def\lf{\lfloor}
\def\rf{\rfloor}
\spnewtheorem*{remark*}{Remark}{\it}{\rm}
\spnewtheorem*{remarks*}{Remarks}{\it}{\rm}
\title{Partial mirror symmetry, lattice presentations and algebraic monoids}

\author{Brent Everitt and
John Fountain
\thanks{Some of the results of this paper were 
obtained while the first author was visiting
the Institute for Geometry and its Applications, 
University of Adelaide, Australia. 
He is grateful for their hospitality. He also apologises to Tony
Armstrong for stealing the quote of the first sentence. A grant from the Royal Society made
it possible for the second author to visit the University of Adelaide to continue the
work reported here. 
Both authors benefited from the
hospitality of Stuart Margolis and Bar-Ilan University, Israel, for
which they are very grateful.}
}

\institute{
Department of Mathematics, University of York, Heslington, York
YO10 5DD, United Kingdom. Brent Everitt \email{brent.everitt@york.ac.uk} 
John Fountian \email{john.fountain@york.ac.uk}.
}

\titlerunning{}
\authorrunning{Brent Everitt and John Fountain}

\begin{document}

\maketitle


\begin{abstract}
This is the second in a series of papers that develops the theory of
reflection monoids, motivated by the theory of reflection groups. 
Reflection monoids were first introduced in \cite{Everitt-Fountain10}.
In this paper we
study their presentations as abstract monoids. Along the way we also find general 
presentations for certain join-semilattices (as monoids under join) which 
we interpret for two special classes of examples: the face lattices of convex polytopes and
the geometric lattices, particularly the intersection lattices of hyperplane arrangements.
Another spin-off is a general presentation for the Renner monoid of an algebraic monoid,
which we illustrate in the special case of the ``classical'' algebraic monoids.
\end{abstract}


\section*{Introduction}

``Numbers measure size, groups measure symmetry'',
and inverse monoids
measure partial symmetry. In \cite{Everitt-Fountain10} we initiated
the formal study of partial \emph{mirror\/} symmetry via the theory of
what we call reflection monoids. The aim is three-fold: (i). to wrap
up a reflection group and a naturally associated combinatorial object
into a single algebraic entity
having nice properties, (ii). to unify various unrelated (until now) parts of the theory
of inverse monoids under one umbrella, and (iii). to provide
workers interested in partial symmetry with the appropriate tools to study the 
phenomenon systematically. 

This paper continues the programme by studying presentations for reflection 
monoids. As one of the distinguishing features of real reflection, or Coxeter groups,
are their presentations, this is an entirely natural thing
to do. Broadly, our 
approach is to adapt the presentation found in \cite{Easdown05} to our purposes. 

Roughly speaking, an inverse monoid (of the type considered in this
paper) 
is made up out of a group $W$ (the \emph{units\/}),
a poset $E$ with joins $\vee $
(the \emph{idempotents\/}) and an action of $W$ on $E$. A presentation for an 
inverse monoid thus has relations pertaining to each of these three components. In particular, we need 
presentations for $W$ as a group and $E$ as a monoid under $\vee$.

For a reflection monoid, $W$ is a reflection group. If it is a real reflection 
group, as all in this paper turn out to be,  then it has a Coxeter
presentation; so that part
is already nicely taken care of. 

The poset $E$ is a commutative monoid of idempotents, and we invest a certain amount of effort
in finding presentations for these (\S\ref{section:idempotents}). We imagine that much of this material
is of independent interest. Here we are motivated by the notion of independence in 
a geometric lattice (see for instance \cite{Orlik80}), which we first generalize to the setting of graded 
atomic $\vee$-semilattices. The idea is that relations arise when we have dependent
sets of atoms. Our first examples are the face monoids of convex polytopes,
and it turns out that simple polytopes have particularly simple presentations. The 
pay-off comes in \S\ref{section:renner}, where these face monoids are the idempotents
in the Renner monoid of a linear algebraic monoid (Renner monoids being to algebraic
monoids as Weyl groups are to algebraic groups). We then specialize to geometric 
lattices--their presentations turn out to be nicer (Theorem \ref{section:idempotents:result350}).
We finally come full circle with the intersection lattices of the reflecting hyperplanes of 
a finite Coxeter group (\S\ref{section:idempotents:coxeter}), where we work though the details
for the classical Weyl groups. These reappear in \S\ref{section:popova:arrangement} as the
idempotents of the Coxeter arrangement monoids.

Historically, presentations for reflection monoids start with Popova's presentation for
the symmetric inverse monoid $\II_n$ \cite{Popova61}. Just as the symmetric group
$\SS_n$ is one of the simplest examples of a reflection group (being the Weyl group of type
$A_{n-1}$) so $\II_n$ is one of the simplest examples of a reflection monoid (the 
Boolean monoid of type $A_{n-1}$; see \cite{Everitt-Fountain10}*{\S 5}). 
Our general
presentation for a reflection monoid
(Theorem \ref{presentations:result1000} of
\S\ref{section:presentation}) 
specializes to Popova's in this special case,
unlike those found in \cites{Godelle10,Solomon02}.
In the resulting presentation there is one relation that seems less obvious than the 
others. This turns out to always be true. The units in a reflection
monoid form a
reflection group $W$ and each relation in this non-obvious family arises from an 
orbit of the $W$-action on the reflecting hyperplanes of $W$. So, the interaction 
between a reflection group and a naturally associated combinatorial object
(in this case the intersection lattice of the reflecting hyperplanes of the group) manifests
itself in the presentation for the resulting reflection monoid. 

Sections \ref{section:popova:boolean} and 
\ref{section:popova:arrangement} work out explicit presentations for
the two main families of reflection monoids
that were introduced in \cite{Everitt-Fountain10}: the Boolean monoids and the Coxeter arrangement monoids.

Finally, we get presentations for the Renner monoids of algebraic
monoids at no extra cost (\S\ref{section:renner}). 
It turns out that the Renner monoids are not
reflection monoids in general (see, e.g.: \cite{Everitt-Fountain10}*{Theorem 8.1})
but more general examples of monoids of partial isomorphisms with unit groups that are
nevertheless reflection groups. In any case, our presentation works
with only minor modifications. 
The result involves fewer generators and relations than that found in \cite{Godelle10}.
The basic principle here is to build an abstract monoid of partial isomorphisms from 
a reflection group acting on a combinatorial description of a rational polytope. This 
abstract monoid is then isomorphic to the Renner monoid of an algebraic monoid--the
reflection group corresponds to the Weyl group of the underlying algebraic group and the
polytope arises from the weights of a representation of the Weyl group
(a reflection group and naturally associated combinatorial structure
being wrapped up!). We work the details 
for the ``classical'' algebraic monoids (special linear, orthogonal, symplectic) as well as 
another nice family of examples introduced by Solomon in \cite{Solomon95}.

\section{Reflection monoids}
\label{section:reflection.monoids}

We start with a brief summary of the reflection group fundamentals that
we will need. The standard references are
\cites{Bourbaki02,Humphreys90,Kane01}.
We then recall the monoids of partial symmetries 
introduced in \cite{Everitt-Fountain10}. 

Let $V$ be a finite dimensional vector space over a field $k$. A
reflection is a diagonalizable linear map $s:V\rightarrow V$
having $\dim V-1$ eigenvalues equal to $1$
and a single eigenvalue $\zeta\not=1$ a root of unity. 
Thus, there is a hyperplane $H_s$ fixed pointwise by $s$, with $s$
acting as multiplication by $\zeta$ on a complementary $k$-line.
A reflection group $W\subset GL(V)$ is a group generated by finitely many 
reflections. 

In this paper we specialize\footnote{For concreteness, as in \cite{Everitt-Fountain10}.}
to the case $k=\R$, where there is a distinguished set $S$ of
generating reflections with $(W,S)$ having the structure of a Coxeter group. 
This structure is encoded (and determined by) the Coxeter symbol: 
it has nodes corresponding to the $s\in S$ with the nodes $s$ and $t$
joined by an edge labelled $m_{st}\in\Z^{> 0}\cup\{\infty\}$ iff
$st$ has order $m_{st}$ in $W$. In practice the label is left on the
symbol only if
$m_{st}\geq 4$; the edge is left unlabelled if $m_{st}=3$;
there is no edge if $m_{st}=2$; and $m_{st}=1$ when $s=t$.

The full set $T$ of reflections
in $W$ is the set of $W$-conjugates of $S$. Write $\AA=\{H_t\subset V\,|\,t\in T\}$ for the set of reflecting
hyperplanes of $W$. Then $W$ naturally acts on $\AA$ and every orbit contains an $H_s$ with $s\in S$.
Moreover, if $s,s'\in S$ then $H_s$ and $H_{s'}$ lie in the same orbit if and only if $s$ and $s'$ are
joined in the Coxeter symbol by a path of edges labeled entirely by \emph{odd\/} $m_{st}$. Thus the number
of orbits of $W$ on $\AA$ is the number of connected components of the
Coxeter symbol once the 
even labeled edges have been dropped. This number will appear later on
as the number of relations
in a certain family in the presentation for a reflection monoid.

All the examples considered in this paper will be further restricted in being finite,
hence of the form $W(\Phi)=\langle
s_v\,|\,v\in\Phi\rangle$, where $\Phi\subset V$ is a finite root system and
$s_v$ the reflection in the hyperplane orthogonal to the root $v$.
Thus $\AA=\{v^\perp\,|\,v\in\Phi\}$ where $^\perp$ is with respect to
a $W$-invariant inner product on $V$.
%
The finite real reflection groups are, up to 
isomorphism, direct products of $W(\Phi)$ for $\Phi$ from a well known list of irreducible
root systems. These $\Phi$ fall into five infinite
families of types $A_{n-1},B_n,C_n,D_n$ (the classical systems) and $I_2(m)$,
and six exceptional cases of types $H_3, H_4, F_4, E_6,E_7$ and $E_8$.
Notable among these are the $\Phi$ 
where $W(\Phi)$ is a finite crystallographic reflection or
\emph{Weyl group\/}: the $W(\Phi)\subset GL(V)$ that leave invariant some
free $\Z$-module $A\subset V$. The irreducible crystallographic groups are those of types $A, B, C, D, E$ and $F$,
together with $I_2(6)$, which in the context of crystallographic groups is often renamed $G_2$.

Table \ref{table:roots1} gives $\Phi$ for the classical Weyl
groups with $\{v_1,\ldots,$ $v_n\}$ an orthonormal basis for $V$.
The root systems of types $B$ and $C$ have the same symmetry, 
but different lengths of roots: type $C$ has roots $\pm 2v_i$ rather
than the $\pm v_i$. 
We have labeled the nodes of the Coxeter symbol with a simple
system. 
If $W=W(\Phi)$ is an irreducible Weyl group then the natural $W$-action on these $\Phi$
has orbits consisting of those roots of a given length. Thus in the classical cases there are 
two orbits in types $B$ and $C$ and a single orbit in types $A$ and $D$. This is just a special
case of the fact stated above for a general Coxeter group.

\begin{table}
\begin{tabular}{l}
\hline
\vrule width 5 mm height 0 mm depth 0 mm
Type 
\vrule width 25 mm height 0 mm depth 0 mm
Root system $\Phi$ 
\vrule width 30 mm height 0 mm depth 0 mm
Coxeter symbol and simple system\\\hline
\begin{pspicture}(0,0)(14,1.5)
\rput(1,.75){$A_{n-1}\,(n\geq 2)$}
\rput(4.5,.75){$\{v_i-v_j \,\,(1\leq i\not= j\leq n)\}$}
\rput(8,0){
\rput(2.5,0.75){\BoxedEPSF{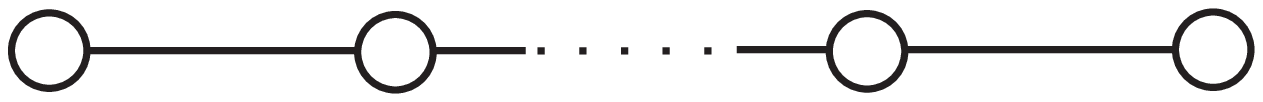 scaled 400}}
\rput(.1,.45){{$v_2-v_{1}$}}\rput(1.55,1.1){{$v_{3}-v_{2}$}}
\rput(3.45,.45){{$v_{n-1}-v_{n-2}$}}\rput(4.85,1.1){{$v_{n}-v_{n-1}$}}
}
\end{pspicture}
\\
\begin{pspicture}(0,0)(14,2.5)
\rput(1,1.25){$D_n\,(n\geq 4)$}
\rput(4.5,1.25){$\{\pm v_i\pm v_j\,\,(1\leq i<j\leq n)\}$}
\rput(8.1,0){
\rput(2.5,1.25){\BoxedEPSF{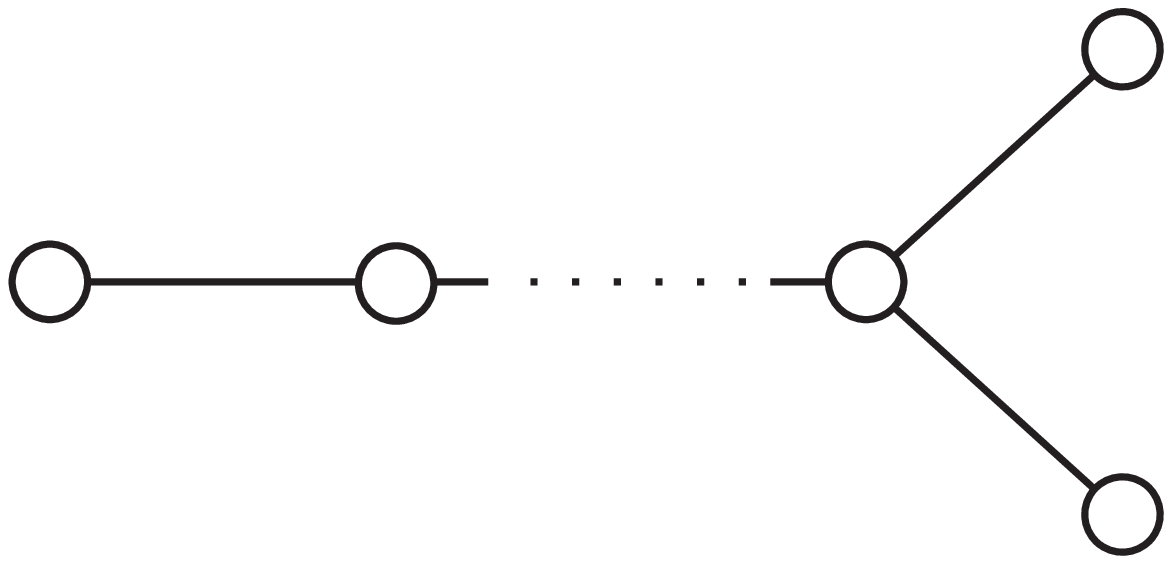 scaled 400}}
\rput(-.5,0){\rput(.8,0.95){{$v_{n}-v_{n-1}$}}
\rput(2.2,1.6){{$v_{n-1}-v_{n-2}$}}
\rput(5,1.25){{$v_3-v_2$}}
\rput(5.2,2.5){{$v_2-v_1$}}
\rput(5.2,0){{$v_1+v_2$}}}
}
\end{pspicture}
\\
\begin{pspicture}(0,0)(14,1.5)
\rput(1,.75){$B_n\,(n\geq 2)$}
\rput(4.5,.75){\begin{tabular}{l}
$\{\pm v_i\,\,(1\leq i\leq n)$, \\
$\pm v_i\pm v_j\,\,(1\leq i<j\leq n)\}$\\
\end{tabular}}
\rput(8.1,0){
\rput(2.5,0.75){\BoxedEPSF{fig2.eps scaled 400}}
\rput(.2,.45){{$v_{n}-v_{n-1}$}}\rput(1.55,1.1){{$v_{n-1}-v_{n-2}$}}
\rput(3.45,.45){{$v_{2}-v_{1}$}}\rput(4.85,1.1){{$v_1$}}
\rput(4.2,1){$4$}
}
\end{pspicture}
\\\hline
\end{tabular}
\caption{
Root systems, simple systems and Coxeter symbols for the classical
Weyl groups. 
}
\label{table:roots1}
\end{table}

As in \cite{Everitt-Fountain10}, when $G\subseteq GL(V)$ is any group
and $X\subseteq V$, a key role
is played by the isotropy groups $G_X=\{g\in G\,|\,vg=v\text{ for all }v\in X\}$.
A theorem of Steinberg \cite[Theorem 1.5]{Steinberg64}
asserts that for $G=W(\Phi)$, the
isotropy group $W(\Phi)_X$ is itself a reflection group; indeed,
generated by the reflections $s_v$ for $v\in\Phi\cap X^\perp$. 

So much for mirror symmetry; now to \emph{partial\/} mirror symmetry,
where we recall the definitions of \cite{Everitt-Fountain10}*{\S 2}.
If $G\subseteq GL(V)$ is a group, then 
a collection $\cS$ of subspaces of $V$ is called a system in $V$ for $G$
if and only if $V\in\cS$, $\cS G=\cS$, and $X\cap Y\in\cS$ for $X,Y\in\cS$. 
A partial linear isomorphism of $V$ is a vector space isomorphism 
$X\rightarrow Y$, for subspaces $X,Y$ of $V$
(including the zero map $\mathbf{0}\rightarrow\mathbf{0}$ from the zero
subspace to itself). Any such can be obtained by restricting to $X$ a full 
isomorphism $g\in GL(V)$. We write $g_X^{}$ for the partial isomorphism
with domain $X$ and effect that of restricting $g$ to $X$.
In this form, two partial linear isomorphisms are composed
as $g_X^{}h_Y^{} = (gh)_{Z}^{}$ with $Z=X \cap\kern1pt Yg^{-1}$.
If $G\subseteq GL(V)$ and $\cS$ is a system for $G$ then the resulting 
monoid of partial isomorphisms is 
$$
M(G,\cS):=\{g_X^{}\,|\, g\in G, X\in\cS\}.
$$
If $G=W$ is a reflection group then $M(W,\cS)$  is called a \emph{reflection monoid\/}.
The monoid $M(G,\cS)$ has units the group $G$ and idempotents the 
partial identities: the $\ve_X^{}:X\rightarrow X$ with $\ve$ the
identity on $V$ and $X\in\cS$. 

The previous paragraph can be mimicked to give partial permutations instead of partial linear
isomorphisms: replace $V$ by a finite set $E$; the group $G$ is now $G\subseteq\SS_E$ and $\cS$ is
a system of subsets of $E$ that contains $E$ itself, and is closed under $\cap$ and the
$G$-action. The resulting $M(G,\cS)$ is a monoid of partial
permutations of $E$. In all the examples in this paper $E$ will turn out to have
more structure: it will be a $\vee$-semilattice with a unique minimal
element $\0$ and with the $G$-action by poset isomorphisms. 
The system of subsets $\cS$ consists of intervals in
$E$, namely, for any $a\in E$ the sets $E_{\geq a}:=\{b\in E\,|\,b\geq a\}$.
Then $E=E_{\geq\0}$, $E_{\geq a}\cdot g=E_{\geq a\cdot g}$ for $g\in G$, and 
$E_{\geq a}\cap E_{\geq b}=E_{\geq a\vee b}$. Ordering $\cS$ by \emph{reverse\/}
inclusion, the map $E\rightarrow\cS$ given by $a\mapsto E_{\geq a}$ is a poset isomorphism
that is equivariant with respect to the $G$-actions on $E$ and $\cS$.

All the monoids considered above are inverse monoids: for any $a\in M$
there is a unique $b\in M$ with $aba = a$ and $bab = b$. Moreover, any $g_X^{}$ can
be written as $g_X^{}=\ve_X^{} g$, a product of an idempotent and a unit. Thus
the monoids above are also factorizable: 
$M=EG$ with $E$ the idempotents and $G$ the units. Indeed it is not hard to show
\cite{Everitt-Fountain10}*{Proposition 9.1}
that the monoids of partial permutations 
$M(G,\cS)$ are \emph{precisely\/} the finite factorizable inverse monoids,
and the reflection monoids are the factorizable inverse monoids generated
by partial reflections (i.e.: the $s_X^{}$ with $s$ a reflection).
In this setting the role of the isotropy group is played by the idempotent stabilizer 
$G_e=\{g\in G\,|\,eg=e\}$. 

\section{Idempotents}
\label{section:idempotents}


A poset with joins and a unique minimal element is a monoid. Finding presentations for
such monoids is the subject of this rather long section.  

\subsection{Generalities}
\label{section:idempotents:generalities}

Let $E$ be a finite commutative monoid of idempotents. It is a fundamental result
that $E$ acquires, via the ordering $x\leq y$ if and only if $xy=y$, 
the structure of a join semi-lattice with a unique minimal element.
Conversely, any join semi-lattice with unique minimal element 
is a commutative monoid of idempotents via $xy:=x\vee y$. 
Moreover, in either case we also have a unique maximal element--the
join of all the elements of (finite) $E$.
From now on we will apply monoid and poset
terminology (see \cite{Stanley97}*{Chapter 3})
interchangeably to $E$ and write $\0$ for the unique minimal element
and $\1$ for the unique maximal one. The reader should beware: the
$\0$ of the poset $E$ is the multiplicative $1$ of the monoid $E$ and
the $\1$ is the multiplicative $0$.
Recall that a poset map $f:E\rightarrow E'$ is a map with $fx\leq' fy$
when $x\leq y$.

All of our examples will turn out to have slightly more structure: 
$E$ is \emph{graded\/} if for every $x\in E$, any two saturated chains
$\0=x_0<x_1\cdots<x_k=x$ have the same length. In this case $E$ has a rank function
$\rk:E\rightarrow\Z^{\geq 0}$ with $\rk(x)=k$. In particular
$\rk(\0)=0$, and if $x$ and $y$ are such that $x\leq z\leq y$ implies $z=x$ or $z=y$,
then $\rk(y)=\rk(x)+1$. 
Write $\rk E:=\rk(\1)$. 
The elements of rank $1$ are called the \emph{atoms\/}, and
$E$ is said to be \emph{atomic\/} if every element is a join of atoms. In particular, an atomic $E$ is
generated as a monoid by its atoms.

For example,  the \emph{Boolean lattice $\BB_X$ of rank $n$\/} is the lattice 
of subsets of
$X=\{1,\ldots,n\}$ ordered by \emph{reverse\/} inclusion. It is graded
with $\rk(Y)=|X\setminus Y|$
and atomic, with atoms the $a_i:=\{1,\ldots,\wh{i},\ldots,n\}$. The monoid operation is just 
intersection.



Writing $\bigvee S$ for the join of the elements in a subset $S\subseteq E$, call a set $S$ of atoms
\emph{independent\/} if $\bigvee S\setminus\{s\}<\bigvee S$ for all $s\in S$, and 
\emph{dependent\/} otherwise; $S$ is \emph{minimally dependent\/} if it is dependent and
every proper subset is independent.
These notions satisfy the following properties, most of which are
clear, although some hints are given:
\begin{description}
\item[(I1).] If $|S|\leq 2$ then $S$ is independent (any two atoms are
  incomparable); in particular, any three element set of dependent atoms is minimally
dependent.
\item[(I2).] If $S$ is dependent then there exists $T\subset S$ with $T$ independent and $\bigvee T=\bigvee S$
(successively remove those $s$ for which $\bigvee S\setminus\{s\}=\bigvee S$).
\item[(I3).] If $T$ is dependent and $T\subseteq S$, then $S$ is dependent. Thus, any subset of an independent
set is independent.
\item[(I4).] If $T$ is independent and $S=T\cup \{b\}$ is dependent then
there is a  $T'\subseteq T$ with $T'\cup \{b\}$ minimally dependent
(this is clear if $|S|=3$; if $S$ arbitrary is not minimally dependent
already then 
there is an $s\in S$ with $S\setminus\{s\}$ dependent, and in particular $s\not=b$. 
The result then follows by induction applied to $S\setminus\{s\}$.)
\item[(I5).] If $S$ is independent then there is an injective map of
  posets $\BB_S\hookrightarrow E$, not necessarily grading preserving
(send $T\subseteq S$ to $\bigvee T$ in $E$); consequently, if $S$ is
independent then $|S|\leq\rk E$. 
\end{description}
There is an obvious analogy here with linear algebra, which 
becomes stronger in \S\ref{section:idempotents:geometric}
when $E$ is a geometric lattice. 



Here is our first presentation. Throughout this paper we adopt the
standard abuse whereby the same symbol is used to denote an element of
an abstract monoid given by a presentation and the corresponding element of the
concrete monoid that is being presented. Apart from the proof of the following (where we
temorarily introduce new notation to separate these out) the context ought to make clear what
is being denoted. 

\begin{proposition}
\label{presentations:result300}
Let $E$ be a finite graded atomic commutative monoid of idempotents with atoms $A$. Then
$E$ has a presentation with:
\begin{align*}
\text{generators:\hspace{1em}}&a\in A.&&\\
\text{relations:\hspace{1em}}
&ab=ba\,(a,b\in A),&&\text{(Idem2)}\\
&a_1\ldots a_k=a_1\ldots a_kb\,(a_i,b\in A),
&&\text{(Idem3)}\\
&\hspace*{0.5cm}\text{ for }a_1,\ldots,a_k\, ,(1\leq k\leq\rk E)\text{ independent and }b\leq\bigvee a_i.&&\\
\end{align*}
\end{proposition}


Notice that when $k=1$ the \emph{(Idem3)\/} relations are $a=a^2$ for
$a\in A$. To emphasise the point we separate these from the rest of
the \emph{(Idem3)\/} relations and call them family \emph{(Idem1)\/}.
Note also that the $\{a_1,\ldots,a_k,b\}$ appearing in \emph{(Idem3)}
are dependent.

\begin{proof}
We temporarily introduce alternative notation for the atoms and then
remove it at the end of the proof: we use Roman letters $a,b,\ldots$ for the atoms
$A$ of $E$ and their Greek equivalents $\aa,\bb,\ldots$ for a set in 1-1
correspondence with $A$.
Let $M$ be the quotient of the free monoid on the $\aa\in A$ by the congruence generated
by the relations \emph{(Idem2)}-\emph{(Idem3)}, with Greek letters rather than Roman.
We have already observed that $E$ is generated by the $a\in A$, and the relations
\emph{(Idem2)}-\emph{(Idem3)}
clearly hold in $E$, so the map $\aa\mapsto a$ induces an epimorphism $M\rightarrow E$.
To see that this map is injective,
we choose representative words: for
any $e\in E\setminus\{\0\}$, let $A_e:=\{a\in A\,|\,a\leq e\}$ and 
$$
\ul{e}=\prod_{a\in A_e}\aa.
$$ 
It remains to show that any word in the $\aa$'s mapping to $e$ can be transformed into 
the representative word $\ul{e}$ using the 
relations \emph{(Idem1)}-\emph{(Idem3)}. Let $\aa_1\ldots\aa_k$ be 
such a word and let $b\in A_e$ be such that
$b\not=a_i$ for any $i$. If no such $b$ exists then the word is
$\ul{e}$ already and we are done. 
Otherwise, there is 
an independent subset $\{a_{i_1}\ldots,a_{i_\ell}\}$ with 
$e=a_{i_1}\vee\cdots\vee a_{i_\ell}$, and so we have an \emph{(Idem3)} relation
$\aa_{i_1}\cdots\aa_{i_\ell}=\aa_{i_1}\cdots\aa_{i_\ell}\bb$. Multiplying both sides by $\aa_1\ldots\aa_k$, 
reordering using \emph{(Idem1)} and removing redundancies using \emph{(Idem2)}, we obtain
$\aa_1\ldots\aa_k=\aa_1\ldots\aa_k\bb$. Repeat this until the word is $\ul{e}$.
\qed
\end{proof}



For a simple example, the Boolean lattice $\BB_X$ of rank $n$ has atoms the 
$a_i=\{1,\ldots,\wh{i},\ldots,n\}$ with 
$\bigvee a_{i_j}$ the set $X$ with the indices $i_j$ omitted. Removing an atom from 
this join has the effect of re-admitting the corresponding index. The resulting join 
is thus strictly smaller than $\bigvee a_{i_j}$, and we conclude that any set of atoms is independent. 
As an \emph{(Idem3)\/} relation in Proposition \ref{presentations:result300} arises
as a result of a set $\{a_1,\ldots,a_k,b\}$ of dependent atoms, 
the \emph{(Idem3)\/} relations
are vacuous when $k>1$ and we have a presentation
with generators $a_1,\ldots,a_n$ 
and relations $a_i^2=a_i$ and $a_ia_j=a_ja_i$ for all $i,j$.

\subsection{Face monoids of polytopes}
\label{section:idempotents:polytopes}

In \S\ref{section:renner} we will encounter a class of commutative monoids of
idempotents that are isomorphic to the face lattices of convex polytopes. It
is to these that we now turn.

A \emph{(convex) polytope\/} $P$ in a real vector space $V$ is the convex hull
of a finite set of points. The standard references for convex polytopes are
\cites{Grunbaum03,Ziegler95}.

An affine hyperplane $H$ supports $P$ whenever $P\cap H\not=\varnothing$ and
$P$ is contained in one of the closed half spaces given by $H$. A subset $f\subseteq P$
is an \emph{$r$-face\/} if $f=H\cap P$ for some supporting hyperplane and
a maximal affinely independent subset of $f$ contains $r+1$ points. We write 
$\dim f=r$. We consider $P$ itself to be a face (and say $P$ is a $d$-polytope
when $\dim P=d$) and $\varnothing$ to be the unique face of dimension $-1$. 
A $(d-1)$-face of a $d$-polytope is called a \emph{facet\/}. 

Let $\FF(P)$ be the faces of $P$ ordered by \emph{reverse\/}
inclusion. Once again this
is the opposite order to that normally used in the polytope
literature. In any case, it is well known
that $\FF(P)$ is a graded ($\rk f=\codim_Pf:=\dim P-\dim f$), atomic lattice with
atoms the facets, join $f_1\vee f_2=f_1\cap f_2$, meet $f_1\wedge f_2$ the smallest face
containing $f_1$ and $f_2$, unique minimal element
$\0=P$ and maximal element $\1=\varnothing$
(hence $\rk\FF(P)=\dim P$). We call the associated monoid the \emph{face monoid of the
polytope $P$\/}.

Two polytopes are combinatorially equivalent if their face lattices are isomorphic as lattices.
The \emph{combinatorial type\/} of a polytope is the isomorphism class of its face lattice,
and when one talks of a combinatorial description of a polytope, one means a description
of $\FF(P)$. 
In this paper, all statements about polytopes are true \emph{up to combinatorial type}.

\begin{example}[the $d$-simplex $\Delta^d$]
\label{eg:simplex}
Let $V$ be a $(d+1)$-dimensional 
Euclidean space
with basis $\{v_1,\ldots,v_{d+1}\}$. 
The convex hull $\Delta^d$ of the basis vectors
$\{v_1,\ldots,v_{d+1}\}$ lies
in the affine hyperplane with equation $\sum x_i=1$, hence the drop in dimension. 
Any subset of the $v_i$ of size $k+1$ spans a $k$-simplex.
If $X=\{1,\ldots, d+1\}$
then $\FF(\Delta^d)$ is isomorphic to the Boolean lattice $\BB_X$. 
Indeed, the map sending $Y\subseteq X$ to the convex hull of the points
$\{v_i\,|\,i\in Y\}$ is our isomorphism. Thus, any set of facets is independent.

In particular $\FF(\Delta^d)$ has a presentation with generators 
$a_1,\ldots,a_{d+1}$ and relations $a_i^2=a_i$ and $a_ia_j=a_ja_i$. 
We will meet this commutative
monoid of idempotents twice more in this paper: as the idempotents of the Boolean reflection 
monoids in \S\ref{section:popova:boolean}, and as the idempotents of the
Renner monoid of the ``classical'' linear monoid $\ov{k^\times\sl_d}$ in 
\S\ref{section:renner:classical1}.
\end{example}

\begin{example}[the polygons $P^2_m$]
\label{eg:polygon}
If the $d$-simplex has as many independent sets of facets as
it possibly can, the polygons are at the other extreme: they have no more than they absolutely
must. Identifying a $2$-dimensional Euclidean space with $\C$, let $P_m^2\,(m>2)$ be the convex hull
of the $m$-th roots of unity. Assume that $m>3$, $P_3^2$ being combinatorially equivalent to 
$\Delta^2$. The following are then clear: any set of $k\geq 3$ facets has empty join and contains
a pair of facets with empty join. Thus, if $f_1\ldots,f_k$ are independent, then $k\leq 2$, and 
we have a presentation with generators $a_1,\ldots,a_m$ and relations 
$a_i^2=a_i$, $a_ia_j=a_ja_i$ (all $i,j$)
and $a_ia_j=a_ia_ja_k$ for 
$|j-i|>1$ and all $k$.  The \emph{(Idem3)} relations are vacuous when $|j-i|=1$.
\end{example}

A $d$-polytope is \emph{simplicial\/} when each facet has the combinatorial type of a 
$(d-1)$-simplex. The $d$-simplex
is simplicial, as is:

\begin{example}[the $d$-octahedron or cross-polytope $\Diamond^d$]
\label{eg:octahedron}
Let $V$ be $d$-dimensional 
Euclidean and
$\Diamond^d$ the convex hull of the vectors $\{\pm v_1,\ldots,\pm
v_d\}$. 
To describe $\Diamond^d$ combinatorially, let
$\pm X=\{\pm 1,\ldots,$ $\pm d\}$ and call a subset $J\subset \pm X$ admissible whenever $J\cap(-J)=\varnothing$.
Alternatively, if $J^+=J\cap\,X$ and $J^-=J\cap(-X)$ then
$-J^+\cap J^-=\varnothing$. 
Note that the admissible sets are closed under passing to subsets
(hence under intersection) but not under unions. 
Let $E_0$ be the admissible subsets of $\pm X$ ordered by reverse inclusion. This poset has a number of 
minimal elements, namely, any set of the form $J^+\cup J^-$ with $J^+\subseteq X$
and $J^-=-X\setminus -J^+$. In particular these sets are completely determined by $J^+$. 
Let $E$ be $E_0$, together with $\pm X$, and ordered by reverse inclusion.
Then the map
sending $J\in E$ to the convex hull of the points
$$
\{v_i\,|\,i\in J^+\}\cup\{-v_{-i}\,|\,i\in J^-\},
$$
is a lattice isomorphism $E\rightarrow\FF(\Diamond^d)$. In particular, if $f_J,f_K$ are faces
corresponding to $J,K\in E$ then $f_J\vee f_K=f_{J\cap K}$.   
We will meet the monoid $E$ again in \S\ref{section:renner:classical2} as the idempotents of the
Renner monoids of the classical monoids
$\ov{k^\times\so_{2d}}$,
$\ov{k^\times\so_{2d+1}}$
and $\ov{k^\times\sp_{2d}}$.
\end{example}

For a convex polytope $P$ there is a \emph{dual polytope\/} $P^*$, 
unique up to combinatorial type, with the
property that $\FF(P^*)=\FF(P)^\opp$,
the opposite lattice to $\FF(P)$, i.e.: $\FF(P)^\opp$ has the same elements
as $\FF(P)$ and order $f_1\leq f_2$ in $\FF(P)^\opp$ if and only if $f_2\leq f_1$ in $\FF(P)$. 
Call $P$ \emph{simple\/} if and only if $P^*$ is \emph{simplicial\/}. Equivalently,
each vertex ($0$-face) of $P$ is contained in exactly $\dim P$ facets.
The $d$-simplex is self dual, corresponding to the fact that a Boolean lattice
is isomorphic as a lattice to its opposite. Two other simple polytopes are:

\begin{example}[the $d$-cube $\Box^d$]
\label{eg:cube}
Let $V$ be $d$-dimensional Euclidean and $\Box^d$ the convex hull of the vectors 
$\{\sum\ve_iv_i\,|\,\ve_i=\pm 1\}$. The $d$-cube is dual to the $d$-octahedron 
$\Diamond^d$, so is simple with $\FF(\Box^d)\cong\FF(\Diamond^d)^\opp$, which in
turn consists of the admissible $J\subset \pm X$, together with $\pm X$,
and ordered by inclusion.
We have $f_J\vee f_K=f_{J\cup K}$ if $J\cup K$ is admissible,
and $f_J\vee f_K=\varnothing$ otherwise.   
\end{example}

\begin{example}[the $d$-permutohedron]
\label{eg:permutohedron}
Let $V$ be $(d+1)$-dimensional Euclidean and let 
the symmetric group $\SS_{d+1}$ act on $V$ via $v_i\pi=v_{i\pi}$ for 
$\pi\in\SS_{d+1}$, writing $v\cdot\SS_{d+1}$ for the orbit of $v\in
V$.
Let $0\leq m_1<\cdots<m_{d+1}$ be integers and define
a $d$-permutohedron $P$ to be the convex hull of the orbit $(\sum
m_i\,v_i)\cdot\SS_{d+1}$.
The combinatorial type of $P$ does not depend on the $m_i$, so we will
just say \emph{the\/} $d$-permutohedron.
The drop in dimension comes about as $P$ lies in the affine hyperplane 
with equation $\sum x_i=\sum m_i$.
The $2$-permutohedron is a hexagon lying in the plane
$x_1+x_2+x_3=\sum m_i$; 
Figure \ref{fig:permutohedron}(c) shows the $3$-permutohedron. Our
interest in permutohedra comes about as the lattice $\FF(P)$ is
isomorphic to the idempotents
of the Renner monoid of \S\ref{section:renner:permutohedron}.

We will describe in some detail a 
combinatorial version of the $d$-permutohedron--it is just a reformulation
of a well known one. 
To this end, an orientation of a $1$-face (i.e.: edge) 
\begin{pspicture}(0,0)(2,.25)
\qdisk(.25,.125){.1}\qdisk(1.75,.125){.1}
\psline[linewidth=.3mm]{-}(.35,.125)(1.65,.125)
\end{pspicture}
of the $d$-simplex $\Delta^d$ has the form
\begin{pspicture}(0,0)(2,.25)
\qdisk(.25,.125){.1}\qdisk(1.75,.125){.1}
\psline[linewidth=.5mm,linecolor=red]{<-}(.35,.125)(1.65,.125)
\end{pspicture}
or 
\begin{pspicture}(0,0)(2,.25)
\qdisk(.25,.125){.1}\qdisk(1.75,.125){.1}
\psline[linewidth=.5mm,linecolor=red]{->}(.35,.125)(1.65,.125)
\end{pspicture}.
If $\Delta^d$ is a $d$-simplex with some subset of its edges oriented,
we say that the set of oriented edges is a \emph{partial
  orientation\/} $O$ of $\Delta^d$.

A partial orientation $O$ is \emph{admissible\/} when (i). any
$2$-face in $\Delta^d$ satisfies
$$
\begin{pspicture}(0,0)(14.3,3)
\rput(1,2){\BoxedEPSF{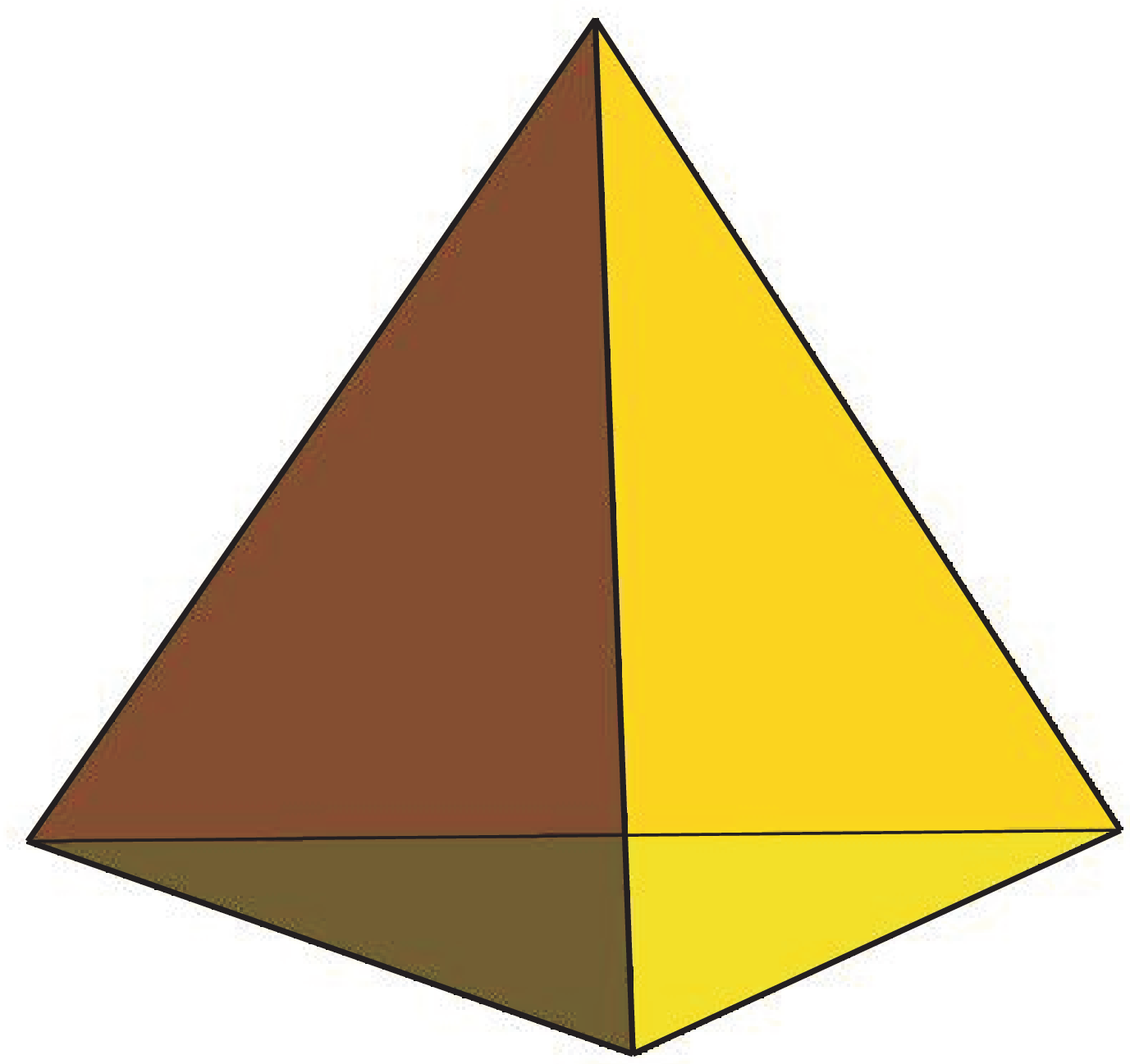 scaled 100}}
\rput(-.3,.5){
\rput(5,1){\BoxedEPSF{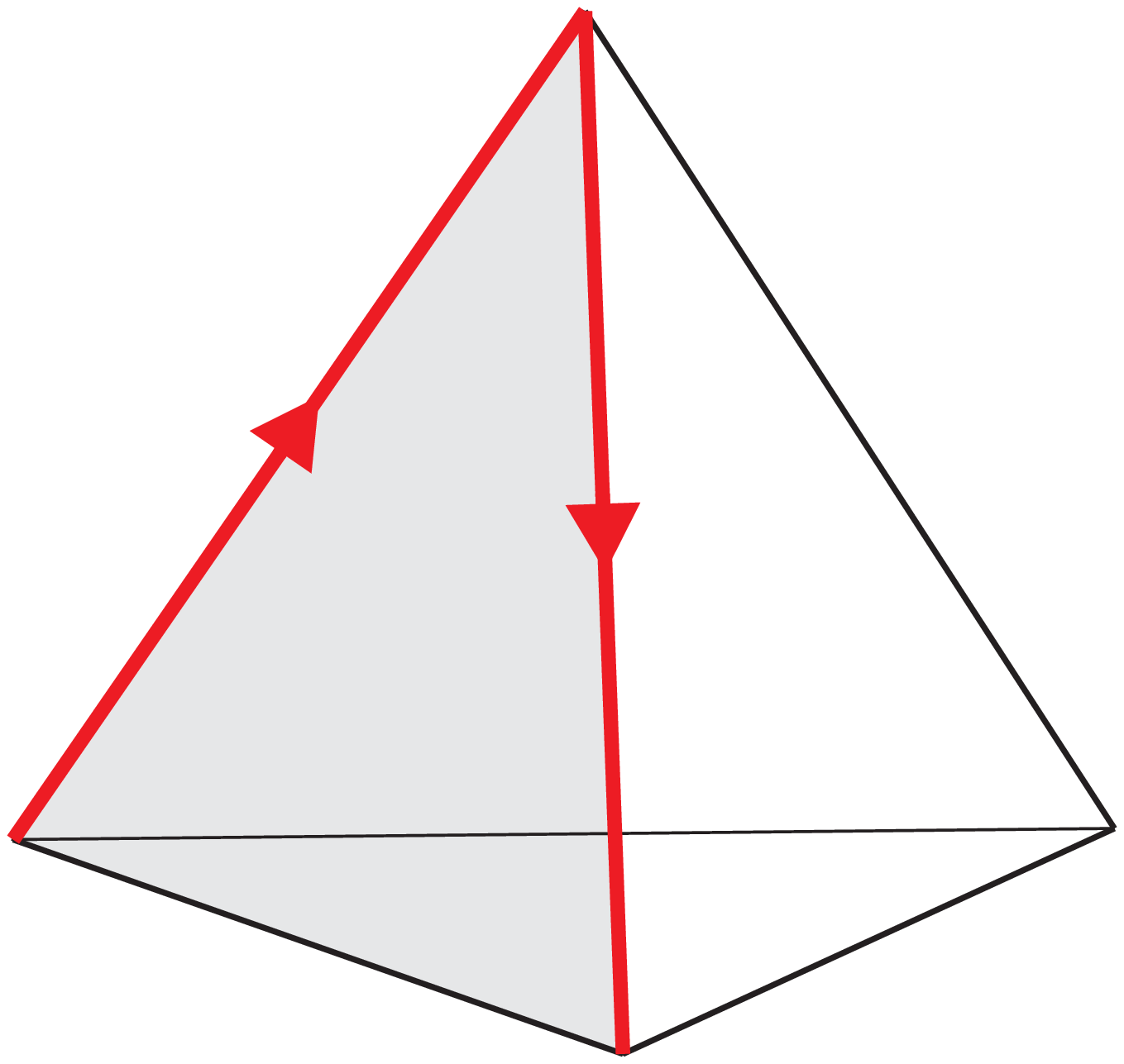 scaled 200}}
\rput(6.5,1){$\in O$}
}
\rput(7.3,1.5){$\Rightarrow$}
\rput(-.7,.5){
\rput(10,1){\BoxedEPSF{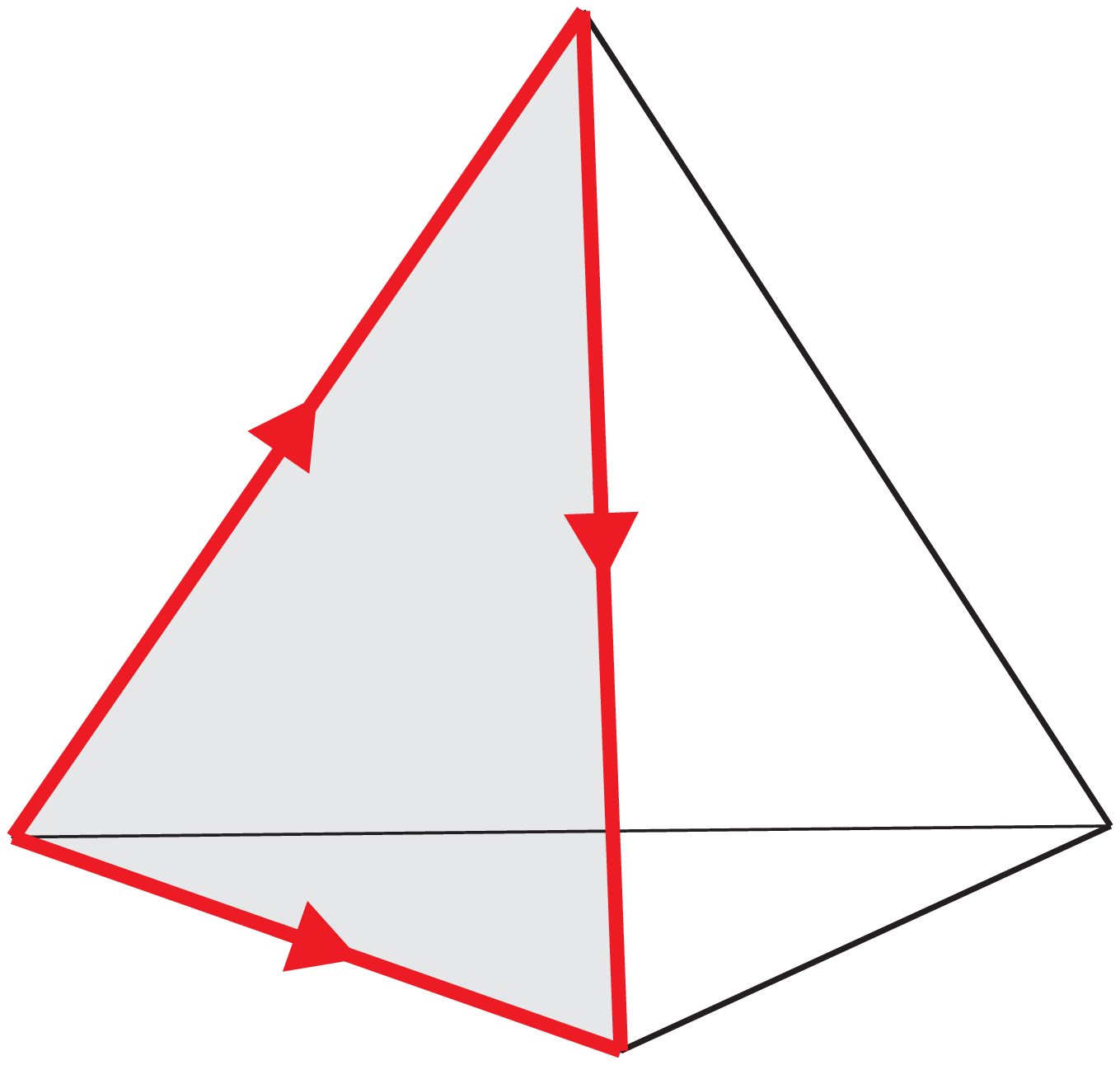 scaled 200}}
\rput(11.5,1){$\in O$}
}
\rput(2,2.2){$\Delta^d$}
\end{pspicture}
$$
and (ii). every $2$-face in $\Delta^d$ has either $0$ or $\geq 2$ of its incident edges in $O$.
We call these two properties \emph{transitivity\/} and \emph{incomparability\/}.

Let $E_0$ be the set of admissible partial orientations of $\Delta^d$ and define $O_1\leq O_2$ 
iff every edge in $O_1$ is also in $O_2$ and with the same
orientation, i.e.: the order is just inclusion in the obvious sense.
This is a partial order
on $E_0$ with a unique minimal element $\varnothing$ (i.e.: no edges oriented) and maximal 
elements the admissible partial orientations where every edge of $\Delta^d$ has been oriented. Formally
adjoin a unique maximal element $\1$ to get the poset $E$.
Define $O_1\vee O_2$ to be the union of the oriented edges in $O_1$ and $O_2$ if this 
gives an admissible partial orientation, or $\1$ if it doesn't. Then $E$ has the
structure of a join semi-lattice. Notice that if there is an edge oriented one way in $O_1$
and the other way in $O_2$ then $O_1\vee O_2$ is not even a partial orientation.
It turns out that this is the only obstacle to $O_1\vee O_2$ being
admissible, as the following show:

\begin{itemize}
\item If $O_1,O_2$ are admissible with $O_1\vee O_2$ a
partial orientation, then $O_1\vee O_2$ is transitive.
\item If $O_1,O_2$ are partial orientations satisfying incomparability
and with $O_1\vee O_2$ a
partial orientation, then $O_1\vee O_2$ satisifies incomparability.
\end{itemize}

Thus for $O_i\in E$ we have $\bigvee O_i<\1$
exactly when $\bigvee O_i$ is a partial orientation, i.e.: 
each edge is oriented consistently (if at all) among the $O_i$. 

For $J$ a non-empty proper subset of $X=\{1,\ldots,d+1\}$, let $\Delta_J$ be the 
sub-simplex of $\Delta^d$ spanned by the vertices $\{v_j\,|\,j\in J\}$ and 
$\Delta_{X\setminus J}$ similarly. Let $O_J$ be the partial orientation where the only
edges oriented are those not contained in either $\Delta_J$ or $\Delta_{X\setminus J}$;
necessarily such edges have one vertex $v_j\,(j\in J)$ and the other
$v_i\,(i\in X\setminus J)$. Orient the edge with the orientation running from the latter
vertex to the former, so that $O_J$ looks as follows:
$$
\begin{pspicture}(0,0)(14.3,2)
\rput(3.75,0){
\rput(0,0){
\pscircle[linewidth=.1mm](2,1){1}
\rput(2,1){$\Delta_J$}
}
\psline[linecolor=red,linewidth=.5mm]{->}(4.4,1.6)(2.6,1.6)
\psline[linecolor=red,linewidth=.5mm]{->}(4.4,1.4)(2.6,1.4)
\psline[linecolor=red,linewidth=.5mm]{->}(4.4,1.2)(2.6,1.2)
\psline[linecolor=red,linewidth=.5mm]{->}(4.4,1)(2.6,1)
\psline[linecolor=red,linewidth=.5mm]{->}(4.4,.8)(2.6,.8)
\psline[linecolor=red,linewidth=.5mm]{->}(4.4,.6)(2.6,.6)
\psline[linecolor=red,linewidth=.5mm]{->}(4.4,.4)(2.6,.4)
\rput(3,0){
\pscircle[linewidth=.1mm](2,1){1}
\rput(2,1){$\Delta_{X\setminus J}$}
}}
\end{pspicture}
$$
We leave it to the reader to show that the $O_J$ are admissible
partial orientations and moreover, are minimal 
non-empty elements in the poset $E$, i.e.: $O_J\in E$, and if $O\in E$ with $O<O_J$ then
$O=\varnothing$.

For any $O\in E$ define a relation $\sim$ on the vertices of
$\Delta^d$ by $u\sim v$ exactly when there is no path of
(consistently) oriented
edges from $u$ to $v$ or from $v$ to $u$. This is easily seen to be
reflexive and symmetric, and also transitive, the last using the
incomparibility and transitivity of the partial orientation $O$. Let
$\{\Lambda_1,\ldots,\Lambda_p\}$ be the resulting equivalence
classes. It is easy to show that given $\Lambda_i,\Lambda_j$ and
vertices $u\in\Lambda_i,v\in\Lambda_j$ that the edge connecting them
lies in $O$, oriented say from $u$ to $v$. Moreover, given any other
such pair $u',v'$, the edge connecting them is also oriented from $u'$
to $v'$. Define an order on the $\Lambda$'s by
$\Lambda_i\preceq\Lambda_j$ whenever the pairs are oriented from
$\Lambda_i$ to $\Lambda_j$ in this way. In particular, $\preceq$ is a
total order and so we write the equivalence classes (after
relabeling) as a tuple $(\Lambda_1,\ldots,\Lambda_p)$, i.e.: we have
an \emph{ordered partition\/}.

For the $O_J$ above we just get $(X\setminus J,J)$ via this
process. If $O\in E$ and $(\Lambda_1,\ldots,\Lambda_p)$ is the
corresponding ordered partition then let
$J_k=\Lambda_k\cup\cdots\cup\Lambda_p$. We leave the reader to see
that we can then write
\begin{equation}
  \label{eq:18}
O=\bigvee_{k=2}^{p} O_{J_k},
\end{equation}
an expression for $O$ as a join of atomic $O_J$. In particular the
$O_J$ comprise \emph{all\/} the atoms in $E$.

\begin{proposition}
\label{presentations:result400}
Let $P$ be the $d$-permutohedron and $E$ the poset of admissible partial
orientations of the $d$-simplex with a formal $\1$ adjoined. 
If $O\in E$ is given by (\ref{eq:18}), let $f_O$ be the convex
hull of those vertices $\sum m_{i\pi} v_i$ such that 
$$
{\textstyle \sum_{j\in J_k}m_{j\pi}=m_1+\cdots+m_{|J_k|}}
$$
for all $k$. Then $O\mapsto f_O$ is an isomorphism $E\cong\FF(P)$  of
lattices. Moreover, facets $f_J:=f_{O_J},f_K:=f_{O_K}$ are disjoint if
and only if 
neither of $J,K$ is contained in the other, i.e.: $J\not=J\cap K\not=K$. 
\end{proposition}


\begin{figure}
  \centering
\begin{pspicture}(0,0)(14.3,6)
\rput(0,0){
\rput(2,3){\BoxedEPSF{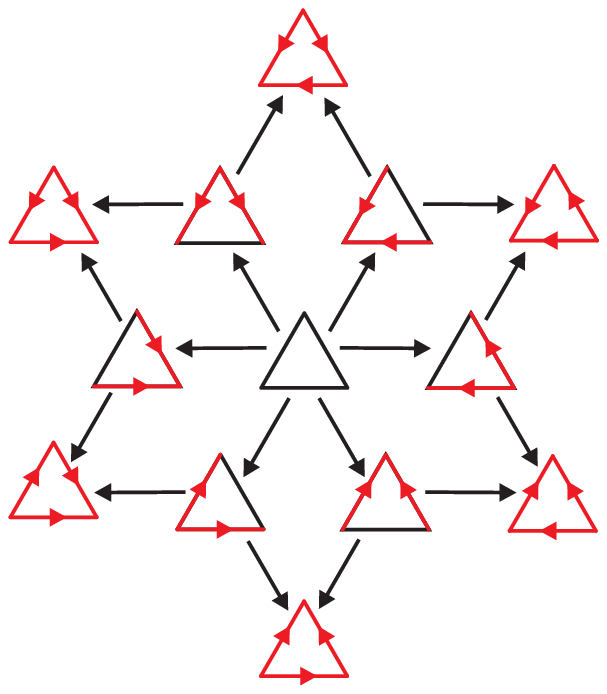 scaled 700}}
\rput(1,5){(a)}
}
\rput(0,-1){
\rput(7,4){\BoxedEPSF{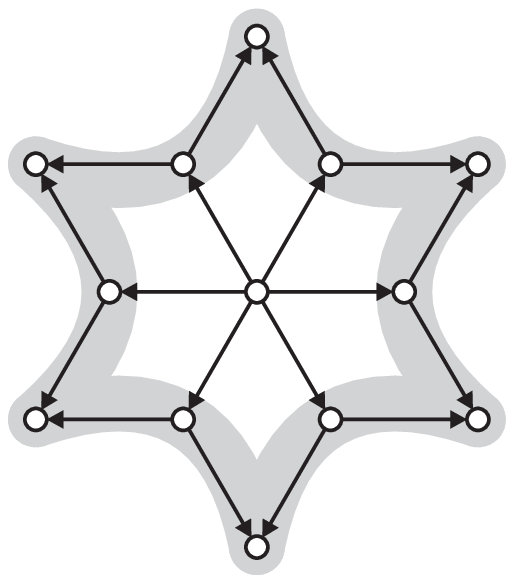 scaled 600}}
\rput(6.2,5.5){(b)}
}
\rput(0,1){
\rput(12,2){\BoxedEPSF{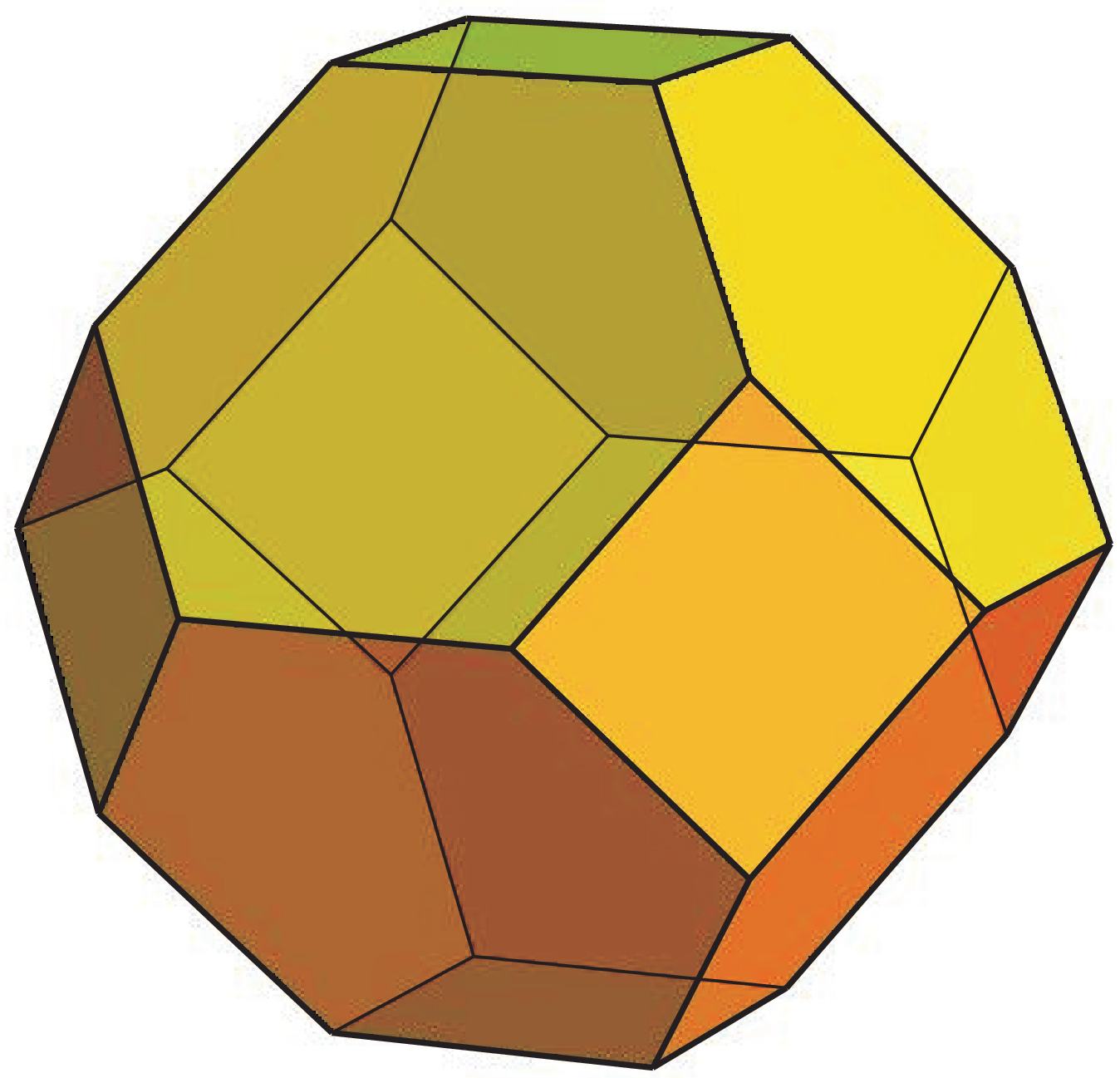 scaled 300}}
\rput(10.2,3.5){(c)}
}
\end{pspicture}
\caption{(a). the poset $E_0$ of partial admissible orientations of
  $\Delta^2$ with $O_1\rightarrow O_2$ indicating $O_1<O_2$ (b). the poset $E_0$
  superimposed on a distorted $2$-permutohedron (or hexagon)
(c). the $3$-permutohedron.}
  \label{fig:permutohedron}
\end{figure}

\begin{proof}
That $O\mapsto f_O$ is a well defined map and a bijection is well known (see,
e.g.: \cite{Ziegler95}*{Lecture 0}). If $O_1\leq O_2$ in $E$ 
then each $J_{2k}$ coincides with some $J_{1k'}$. Thus, if the
$f_{O_i}$ are the convex hulls of sets of vertices $S_i$ as in the
Proposition, we
have $S_2\subseteq S_1$ and so $f_{O_1}\leq f_{O_2}$. This
argument can be run backwards, so that we have a poset isomorphism.
For the final part, $f_J\cap f_K=\varnothing$ iff $O_J\vee O_K=\1$,
and it is easy to check that this happens exactly when $J\not=J\cap K\not=K$. 
\qed
\end{proof}

The $d$-pemutohedron is well known to be simple:  
the vertices correspond to the maximal admissible partial
orientations, hence those with all edges oriented.
Alternatively, the corresponding ordered partition has blocks
$\Lambda_i$ of size $1$, hence we have a total order of $\{1,\ldots,d+1\}$. 
Thus a vertex 
of the permutohedron corresponds to 
an $O$ with the property that the vertices of $\Delta^d$ can be renumbered 
with an edge oriented from $v_i$ to $v_j$ if and only if $i<j$. In particular the $O_J\leq O$
are those with $J=\{k,\ldots,d+1\}$ for  $k>1$, of which there are
exactly $d$. Thus, each vertex of the $d$-permutohedron is contained
in $d$ facets.
\end{example}

Returning to generalities, it turns out that the face lattices of simple polytopes have particularly 
simple presentations as commutative monoids of idempotents. Recalling
the definition of independent atoms from \S\ref{section:idempotents:generalities}, 
we lay the groundwork for this with the
following result:

\begin{proposition}
\label{presentations:result500}
Let $P$ be a simple $d$-polytope. 
\begin{enumerate}
\item If $v$ is a vertex of $P$ then the interval 
$[P,v]:=\{f\in\FF(P)\,|\,P\leq f\leq v\}$
is a Boolean lattice of rank $d$. 
In particular, facets $f_1,\ldots,f_k$ with $\bigvee f_i<\varnothing$ are independent. 
\item Let $P$ be the $d$-cube or the $d$-permutohedron and $f_1,\ldots,f_k\in\FF(P)$
independent facets with $\bigvee f_i=\varnothing$. Then $k\leq 2$. 
\end{enumerate}
\end{proposition}

The first part is standard; indeed it is often stated as an equivalent definition of a
simple polytope as in \cite{Ziegler95}*{Proposition 2.16}. 
The second part is not true for an arbitrary
simple polytope: consider the triangular prism $\Delta^2\times [0,1]$.

\begin{proof}
The claim about the interval follows as $P$ has dual a simplicial poytope, and
$\FF(\Delta^{d-1})$ is Boolean of rank $d$; 
that a collection of facets with non-empty join are independent
follows from this and the comments 
at the end of \S\ref{section:idempotents:generalities} on the Boolean 
lattice of rank $d-1$.
For the second part we show that if $f_1,\ldots, f_k$ are facets with $\bigvee f_i=\varnothing$ then
there are $1\leq j<m\leq k$ with $f_j\vee f_m=\varnothing$; in particular $k\geq 3$ facets
with join $\varnothing$ are dependent. This uses the combinatorial descriptions of the
$d$-cube and $d$-permutohedron. The facets of $\Box^d$ correspond to the admissible 
$J\subset \pm X$ with $|J|=1$, and
$\bigvee f_{J_i}=\varnothing$ exactly when $J=\bigcup J_i$ is not admissible. In particular
there is an $1\leq\ell\leq d$ with $\pm \ell\in J$. But then one of the
admissible sets is $J_j=\{\ell\}$ and another is $J_m=\{-\ell\}$, and so 
$f_{J_j}\vee f_{J_m}=\varnothing$. 
The permutohedron is similar: let the facets $f_i$ correspond to admissible partial
orientations $O_i$ of $\Delta^d$. We have $\bigvee f_i=\varnothing$ exactly when $\bigvee O_i$
is not a partial orientation. Thus there is an edge of $\Delta^d$ and
$O_j,O_m$ with the edge oriented in different directions in these two. But then $f_j\vee f_m=\varnothing$.
\qed
\end{proof}

Part $1$ of Proposition \ref{presentations:result500} means that 
for a simple polytope the \emph{(Idem3)\/}
relations in Proposition \ref{presentations:result300} are vacuous
when $\bigvee a_i<\varnothing$;
part 2 means that for the $d$-cube and $d$-permutohedron the \emph{(Idem3)\/} relations
further reduce to $a_1a_2=a_1a_2b$ for each pair $a_1,a_2$ of disjoint facets. 

\begin{proposition}
\label{presentations:result600}
Let $E$ be the face monoid of a simple polytope $P$ with facets $A$. Then
$E$ has a presentation with:
\begin{align*}
\text{generators:\hspace{1em}}&a\in A.&&\\
\text{relations:\hspace{1em}}
&a^2=a\,(a\in A),&&\text{(Idem1)}\\
&ab=ba\,(a,b\in A),&&\text{(Idem2)}\\
&a_1\ldots a_k=a_1\ldots a_kb\,(a_i,b\in A),
&&\text{(Idem3)}\\
&\hspace*{0.5cm}\text{ for }a_1,\ldots,a_k\,,(2\leq k\leq\dim P)\text{ independent with }\bigvee a_i=\varnothing.&&\\
\end{align*}
\end{proposition}

Combining this presentation with part 2 of Proposition \ref{presentations:result500}
and the combinatorial descriptions of the $d$-cube and $d$-permutohedron
gives:

\paragraph{The $d$-cube $\Box^d$:} has a presentation with generators
$a_{\pm 1},\ldots,a_{\pm d}$ and relations $a_i^2=a_i$ for all
$i$; $a_ia_j=a_ja_i$ for all
$i,j\in\{\pm 1,\ldots,\pm d\}$ and 
$$
a_{i}a_{-i}=a_ia_{-i}a_j
$$
for all $i\in\{1,\ldots,d\}$
and all $j$.

\paragraph{The $d$-permutohedron:}  has a presentation with generators
$a_J^{}$ for $\varnothing\not=J\subsetneq X=\{1,\ldots,d+1\}$ and
relations $a_J^2=a_J^{}$ for all $J$;
$a_J^{}a_K^{}=a_K^{}a_J^{}$ for all $J,K$, and 
$$
a_J^{}a_K^{}=a_J^{}a_K^{}a_L^{}
$$ 
for all $J\not=J\cap K\not=K$ and all $L$.


\begin{proof}[of these presentations]
The facets of the $d$-cube are parametrized by the admissible
$J\subset \pm X$ with $|J|=1$, 
and two such have join $\varnothing$ exactly when they correspond to admissible
$J=\{\ell\}$ and $K=\{-\ell\}$. Similarly, the facets of the $d$-permutohedron are
parametrized by the admissible partial orientations $O_J$, for $J$ a non-empty
proper subset of $X$, and two facets have
join $\varnothing$ exactly when the correspond to $O_J,O_K$ with 
$J\not=J\cap K\not=K$. The presentations follow.
\qed
\end{proof}

\begin{figure}
  \centering
\begin{pspicture}(0,0)(14.3,6)
\rput(1.5,0){
\rput(3,4.5){\BoxedEPSF{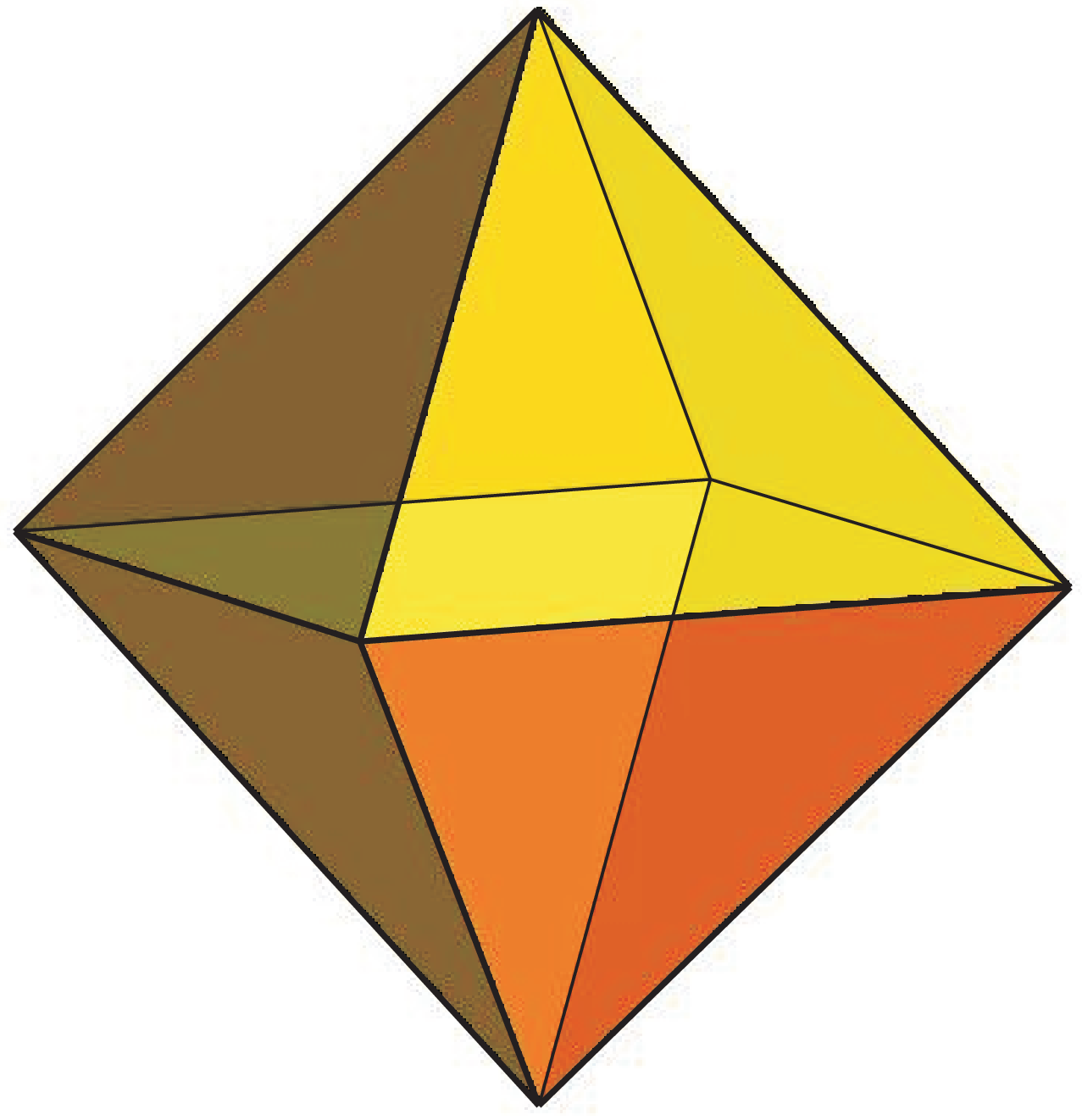 scaled 200}}
\rput(0,0){
\rput(8,3){\BoxedEPSF{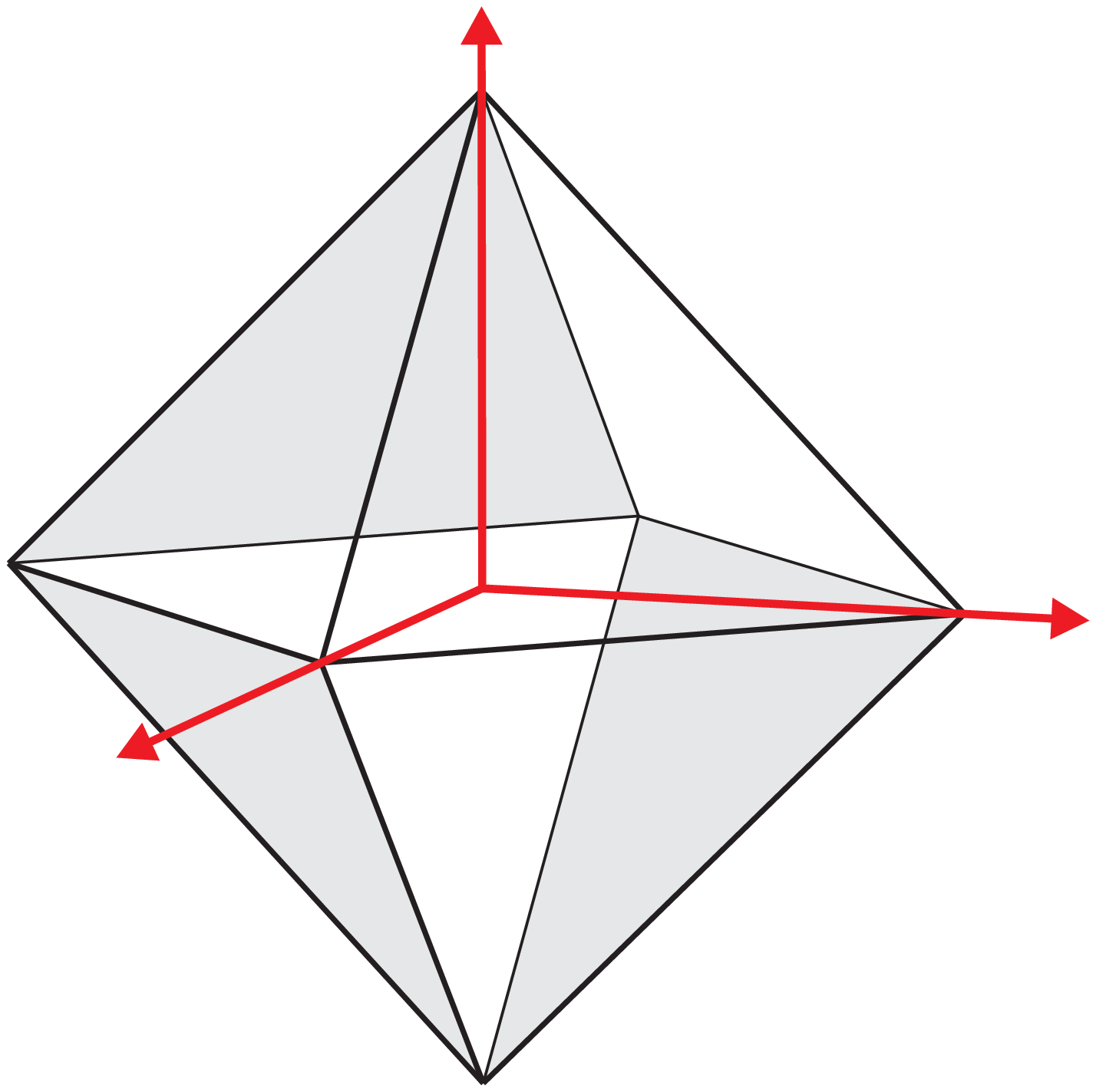 scaled 300}}
\rput(6.2,2){$v_1$}\rput(10.35,2.7){$v_2$}\rput(7.75,5.3){$v_3$}
\rput(6.5,1.3){\BoxedEPSF{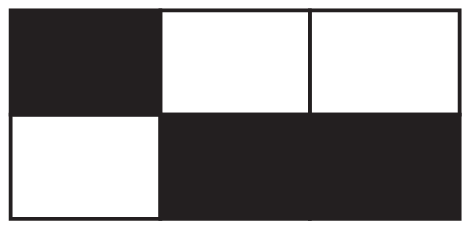 scaled 200}}
\rput(9,1.3){\BoxedEPSF{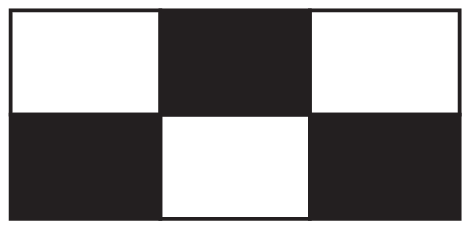 scaled 200}}
\rput(6.35,4.2){\BoxedEPSF{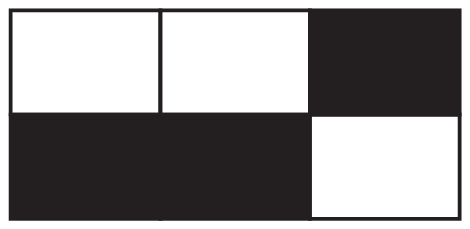 scaled 200}}
}
\rput[lb](2,1){\BoxedEPSF{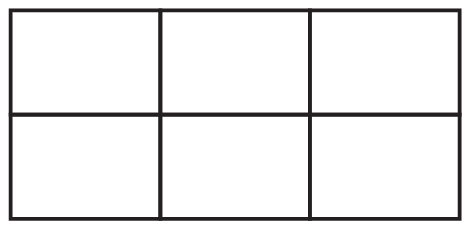 scaled 400}}
\rput(1.5,1.45){$\pm X=$}
\rput(2.3,1.65){$1$}\rput(2.9,1.65){$2$}\rput(3.5,1.65){$3$}
\rput(2.3,1.2){$-1$}\rput(2.9,1.2){$-2$}\rput(3.5,1.2){$-3$}
}
\end{pspicture}
  \caption{Independent triples of facets in the $3$-octahedron $\Diamond^3$: we have
$\{x_1,x_2,x_3\}=X=\{1,2,3\}$ and option (1) of Proposition \ref{presentations:result375}
is chosen for each $j$. The atoms in $E$ are depicted by blackened boxes
and the corresponding facets of the octahedron shaded. Every other triple is equivalent to this
one via a symmetry of $\Diamond^3$.}
  \label{fig:octahedron}
\end{figure}

Finally, we return to the $d$-octahedron $\Diamond^d$, where things are not so simple
(pun intended). Recalling the poset $E$ of Example \ref{eg:octahedron}, 
let $J\subseteq X=\{1,\ldots,d\}$ 
and write $a(J):=J\cup (-X\setminus -J)$ for the atoms in $E$ (note
that $J$ is now a subset of $X$ rather than $\pm X$). The independent sets
can be described:

\begin{proposition}\label{presentations:result375}
Let $\{x_1,\ldots,x_k\}\subseteq X$ with $k$ and $d\geq 3$, and 
$J_{10},\ldots,J_{k0}\subseteq X\setminus\{x_1,\ldots,x_k\}$.
For $j=1,\ldots,k$ we recursively define sets $J_{1j},\ldots,J_{kj}$
as follows: either,
\begin{description}
\item[(0).] do not add $x_j$ to $J_{j,j-1}$ but do add $x_j$ to all other
  $J_{i,j-1}$ for $i\not= j$; or
\item[(1).] do add $x_j$ to $J_{j,j-1}$ but do not add $x_j$ to all other
  $J_{i,j-1}$ for $i\not= j$.
\end{description}
Then, if $J_j:=J_{jk}$, the $a(J_1),\ldots,a(J_k)$ are independent atoms in $E$, and
every set of $k$ independent atoms arises in this way.
\end{proposition}

Thus, at the $0$-th step we have the sets $J_{10},\ldots,J_{k0}$;
at the $1$-st step  either add $x_1$ to $J_{10}$ and not to the
others, or vice-versa; iterate. 

The restriction $d\geq 3$ is partly for convenience, and partly as
$\Diamond^2=P^2_4$ has been done already. 
Figure \ref{fig:octahedron} illustrates the independent triples of facets in the $3$-octahedron:
we have $X=\{1,2,3\}$ and $2^3$ independent triples corresponding to a choice of the 
(0)-(1) options in Proposition \ref{presentations:result375}. Letting $x_j=j$ 
(hence $J_{j0}=\varnothing$) and choosing option (1)
for each $j$ gives the atoms $a(1)=\{1,-2,-3\},a(2)=\{-1,2,-3\}$ and 
$a(3)=\{-1,-2,3\}$ corresponding to the 
shaded triple of faces. Any other triple of independent facets is equivalent to this one via a 
symmetry of the octahedron.

\begin{proof}
By definition a set $a(J_1),\ldots,a(J_k)$ of atoms is independent exactly when
for all $j=1,\ldots,k$ we have
$\bigcap_{i\not=j} a(J_i)\supsetneq \bigcap_i a(J_i)$
Equivalently, for each $j$ there is an $x_j\in\pm X$ with $x_j\not\in J_j$ but
$x_j\in J_\ell\,(\ell\not=j)$. Rephrasing in terms of the $J_j$ rather than the $a(J_j)$, we have 
the $a(J_1),\ldots,a(J_k)$ independent if and only if
for each $j$, either
\begin{description}
\item[(0).] there is an $x_j\in X$ with $x_j\not\in J_j$ and $x_j\in J_\ell\,(\ell\not= j)$, or
\item[(1).] there is an $x_j\in X$ with $x_j\in J_j$ and $x_j\not\in J_\ell\,(\ell\not= j)$.
\end{description}
We claim that the $x_1,\ldots,x_k$ so obtained are distinct. Let $i,j,m$ be distinct 
and suppose that $x_j\in J_j$ and hence $x_j\not\in J_m$, i.e.: we have option (1) above for $j$. 
If $x_i\not\in J_j$ then $x_i\not= x_j$. If $x_i\in J_j$ then this has happened
because option (0) was chosen for $i$, and so in particular $x_i\in
J_m$, and $x_i\not= x_j$
in this case too. Starting instead with $x_j\not\in J_j$, the argument is similar. 

If $\{x_1,\ldots,x_k\}$ are as given in the statement of the Proposition then for each $j$ the set
$J_j:=J_{jk}$ satisfies one of (0) or (1) above, hence the $a(J_j)$
are idependent. On the other hand
if $a(J_1),\ldots,a(J_k)$ is an independent set then we have a set
$\{x_1,\ldots,x_k\}\subseteq X$ by (0) and (1) above, and letting 
$J_{j0}=J_j\cap\setminus\{x_1,\ldots,x_k\}$ gives $J_j=J_{j0}$.
\qed
\end{proof}

Let $\mathit{Ind}_k$ be the set of independent tuples
$(a(J_1),\ldots,a(J_k))$ arising via Proposition \ref{presentations:result375}.

\paragraph{The $d$-octahedron $\Diamond^d$:} has a presentation with generators $a_J^{}$
for $J\subseteq X=\{1,\ldots,d\}$ and relations $a_J^2=a_J^{}$ for
all $J$; $a_J^{}a_K^{}=a_K^{}a_J^{}$ for all $J,K$ and 
$$
a_{J_1}\ldots a_{J_k}=a_{J_1}\ldots a_{J_k} a_{K}
$$
for all $(a(J_1),\ldots,a(J_k))\in\mathit{Ind}_k$ with $2\leq k\leq d$ and
all $a(K)\supseteq\bigcap a(J_i)$.


\subsection{Geometric monoids}
\label{section:idempotents:geometric}

Suppose now that $E$ is a lattice, hence with both joins $\vee$ and meets
$\wedge$. 
A graded atomic lattice $E$ is \emph{geometric\/}
when
\begin{equation}
  \label{eq:6}
\rk(a\vee b)+\rk(a\wedge b)\leq \rk(a)+\rk(b),
\end{equation}
for any $a,b\in E$.
We will call the corresponding commutative monoid 
of idempotents \emph{geometric\/}.

Beginning with a non-example, the face lattices of polytopes are not in general
geometric: if $f_1,f_2$ are facets of the $n$-cube with $f_1\vee f_2=\varnothing$,
then the left hand side of (\ref{eq:6}) is $n$ and the right hand side is $2$. 

The canonical example of a geometric lattice is the collection of all subspaces of a 
vector space under either inclusion/reverse 
inclusion, where (\ref{eq:6}) is a well known equality.
The example that will preoccupy us is the following: a
\emph{hyperplane arrangement\/} is a finite set $\AA$ of linear hyperplanes in a vector space $V$,
and the intersection lattice $\HH$ is the set of all intersections of elements of $\AA$ 
ordered by reverse inclusion, with the null intersection taken to be $V$. The result
is a geometric lattice \cite{Orlik92}*{\S2.1} with $\rk(A)=\codim A$,
atoms the hyperplanes $\AA$; $\0=V$ and $\1=\bigcap_{H\in\AA}H$.
If $\AA$ are the reflecting hyperplanes of a reflection group $W\subset GL(V)$ then 
$\AA$ is called a \emph{reflection\/} or \emph{Coxeter\/}
arrangement. If $W=W(\Phi)$ for $\Phi$ some finite root system, we
will write $\HH(\Phi)$ for the intersection lattice of the Coxeter arrangement.

The linear algebraic analogy of \S\ref{section:idempotents:generalities}
can be pushed a little further in a geometric lattice:
\begin{description}
\item[(I6).] For any set $S$ of atoms we have $\rk(\bigvee S)\leq|S|$, with $S$  independent 
if and only if $\rk(\bigvee S)=|S|$.
\item[(I7).] If $S$ is minimally dependent then $\bigvee S\setminus\{s\}=\bigvee S$ for all
$s\in S$.
\end{description}

That $\rk(\bigvee S)\leq |S|$ is a well known property of geometric lattices that 
follows from (\ref{eq:6})--see for example \cite{Orlik80}.
Indeed, 
(I6) is the normal definition of independence in a geometric lattice.

To see it, we show first by induction on the size of $|S|$
that if $\rk(\bigvee S)<|S|$ then $S$ is dependent:
a three element set with $\rk(\bigvee S)<3$ is the join of any two of
its atoms, 
hence dependent, as the join of two atoms
always has rank two
(the result is vacuous if $|S|=2$ as the join of two distinct atoms
has rank $2$).
If $S$ is arbitrary and $\bigvee S\setminus\{s\}=\bigvee S$ for all
$s$ then $S$ is clearly dependent. Otherwise,
if $\bigvee S\setminus\{s\}<\bigvee S$ for some $s\in S$ with 
$\rk(\bigvee S\setminus\{s\})<\rk(\bigvee S)<|S|$, then
$\rk(\bigvee S\setminus\{s\})<|S\setminus\{s\}|$. By induction,
$S\setminus\{s\}$ is dependent, 
hence so is $S$.

On the other hand, if $\rk(\bigvee S)=|S|$ but $\bigvee S\setminus\{s\}=\bigvee S$ for 
some $s$, then $\rk(\bigvee S\setminus\{s\})=\rk(\bigvee S)=|S|>|\bigvee S\setminus\{s\}|$,
a contradiction. Thus $\rk(\bigvee S)=|S|$ implies that $S$ is independent, and we have 
established (I6).

Condition (I7) is a straightforward comparison
of ranks. Taking three facets of the $2$-cube (square) gives a minimally dependent set $S$
in the face lattice where $\bigvee S\setminus\{s\}=\bigvee S$ is true for only one of the 
three $s$, so this property is not enjoyed by arbitrary graded atomic lattices.

Minimal dependence comes into its own when we have a geometric lattice.
In particular we can replace the 
\emph{(Idem3)\/} relations of Proposition \ref{presentations:result300}
with a smaller set:

\begin{theorem}
\label{section:idempotents:result350}
Let $E$ be a finite geometric commutative monoid of idempotents with atoms $A$. Then
$E$ has a presentation with:
\begin{align*}
\text{generators:\hspace{1em}}&a\in A.&&\\
\text{relations:\hspace{1em}}
&a^2=a\,(a\in A),&&\text{(Idem1)}\\
&ab=ba\,(a,b\in A),&&\text{(Idem2)}\\
&\wh{a}_1\ldots a_k=\cdots=a_1\ldots \wh{a}_k\,(a_i\in A),
&&\text{(Idem3a)}\\
&\hspace*{1cm}\text{ for all }\{a_1,\ldots,a_k\}\text{minimally dependent.}&&\\
\end{align*}
\end{theorem}

\begin{proof}
The
Theorem is proved if we can deduce the \emph{(Idem3)\/} relations 
of Proposition \ref{presentations:result300} 
from the relations above. 
Suppose then that $a_1\ldots a_k=a_1\ldots a_kb$ is 
an \emph{(Idem3)\/} relation with $\{a_1,\ldots,$ $a_k\}$
independent in $E$ and $b\leq\bigvee a_i$. Thus $\{a_1,\ldots,a_k\}$ is
independent and $\{a_1,\ldots,$ $a_k,b\}$ dependent, so by (I4) of 
\S\ref{section:idempotents:generalities} there are $a_{i_1},\ldots,a_{i_k}$ with
$\{a_{i_1},\ldots,a_{i_k},b\}$ minimally dependent. In particular, 
we have $a_{i_1},\ldots a_{i_k}=a_{i_1}\ldots a_{i_k}b$
by \emph{(Idem3a)}, and multiplying both 
sides by $a_1\ldots a_k$ and using \emph{(Idem1)}-\emph{(Idem2)} gives the result.
\qed
\end{proof}

It is sometimes convenient to use the (Idem3a) relations in the form:
$$
\begin{pspicture}(0,0)(14,0.5)
\rput(5,0.25){$a_1\ldots a_k=a_1\ldots\wh{a}_i\ldots a_k$}
\rput(12.5,0.25){\emph{(Idem3b)}}
\end{pspicture}
$$
for all $\{a_1,\ldots,a_k\}$ minimally dependent and all $1\leq i\leq k$. 

\subsection{Coxeter arrangements}
\label{section:idempotents:coxeter}

In \S\ref{section:popova:arrangement} we will encounter a class of
commutative monoids of idempotents isomorphic to 
the Coxeter arrangements 
$\HH(\Phi)$ for $\Phi$ the root systems of types
$A_{n-1},B_n$ and $D_n$. In this section we interpret Theorem
\ref{section:idempotents:result350}
for these monoids. 
We follow a similar pattern to the previous section: first we 
give the arrangement, then a combinatorial description (which as in \S\ref{section:idempotents:polytopes}
means a description of the lattice $\HH$) and then use this to identify the independent
and minimally dependent sets of atoms. It turns out to be convenient
to expand on an idea of Fitzgerald \cite{Fitzgerald03}.

\begin{example}[$\HH(A_{n-1})$ and the partition lattice $\Pi(n)$]
Let $V$ be Euclidean with orthonormal basis $\{v_1,\ldots,v_n\}$ and $\AA$ the
hyperplanes with equations $x_i-x_j=0$ for all $i\not=j$. Equivalently, if $\Phi$
is the type $A_{n-1}$ root system from Table 
\ref{table:roots1} then $\AA$ consists of the hyperplanes
$\{v^\perp\,|\,v\in\Phi\}$ and $W(\Phi)$ is the symmetric group acting on $V$ by permuting
the $v_i$. 

We remind the reader of the well known combinatorial description of $\HH(A_{n-1})$.
Let $X=\{1,\ldots,n\}$ and consider the partitions $\Lambda=\{\Lambda_1,\ldots,\Lambda_p\}$ of
$X$ ordered by refinement: $\Lambda\leq\Lambda'$ iff every block
$\Lambda_i$ of $\Lambda$ is 
contained in some block $\Lambda'_j$ of $\Lambda'$. 
This is a graded atomic lattice with
\begin{equation}
  \label{eq:15}
{\textstyle \rk\Lambda=\sum(|\Lambda_i|-1)}
\end{equation}
and atoms the partitions
having a single non-trivial block of the form $\{i,j\}$. The map sending $(v_i-v_j)^\perp$ to the
atomic partition $\{i,j\}$ extends to a lattice isomorphism $\HH\rightarrow\Pi(n)$
given by $X(\Lambda)\mapsto\Lambda$ where $\sum t_iv_i\in X(\Lambda)$
whenever $t_i=t_j$ for $i,j$ in the same block of $\Lambda$. 

To proceed further we borrow an idea from \cite{Fitzgerald03}: 
for a set $S$ of atoms in either $\HH(A_{n-1})$
or $\Pi(n)$, form the graph $\Gamma_S$ with vertex set $X$ and 
$|S|$ edges of the form:
$$
\begin{pspicture}(0,0)(14.3,1)
\rput(6,.15){
\pscircle(0,.5){.15}\pscircle(2,.5){.15}
\rput(0,.1){$i$}\rput(2,.1){$j$}
\psline[linewidth=.3mm]{-}(.15,.5)(1.85,.5)
}
\end{pspicture}
$$
for each atom $(v_i-v_j)^\perp$ or $\{i,j\}\in S$. 
Recall that a connected graph (possibly with multiple edges and loops)
having fewer edges than vertices cannot contain a circuit.
If $\Lambda=\bigvee S$ is the join in $\Pi(n)$, then the blocks of the
partition $\Lambda$
are the vertices in the connected components of $\Gamma_S$. 
Thus, by (\ref{eq:15}), $S$ is independent when the component corresponding to the block
$\Lambda_i$ has $|\Lambda_i|-1$ edges, i.e.: has a number of edges that is one less than
the number of its vertices. Such a connected graph is a tree, so $\Gamma_S$ is
a forest, and we have our independent sets.

The atoms $S$ are thus dependent when $\Gamma_S$ contains a circuit, and minimally
dependent when $\Gamma_S$ \emph{is\/} just a circuit.  
\end{example}

\begin{example}[$\HH(B_n)$]
Let $V$ be as in the previous example and $\AA$ the hyperplanes with equations
$x_i=x_i\pm x_j=0$ for all $i\not= j$; equivalently, if $\Phi$ is the type
$B_n$ root system from Table \ref{table:roots1}
then $\AA$ consists of the hyperplanes $v_i^\perp$
and $(v_i\pm v_j)^\perp$, with $W(\Phi)$ acting on $V$ by signed permutations
of the $v_i$ (see also the end of \S\ref{section:renner:classical1}).

A combinatorial description of $\HH(B_n)$ appears in 
\cite{Everitt-Fountain10}*{\S6.2} (see also \cite{Orlik92}*{\S6.4}):
a coupled partition is a partition of the form
$\Lambda=\{\Lambda_{11}+\Lambda_{12},\ldots,\Lambda_{q1}+\Lambda_{q2},
\Lambda_1,\ldots,\Lambda_p\}$, where the $\Lambda_{ij}$
and $\Lambda_i$ are blocks and $\Lambda_{i1}+\Lambda_{i2}$ is a
``coupled'' block. The $+$ sign is purely formal. Let $\TT$ be the set of 
pairs $(\Delta,\Lambda)$ where $\Delta\subseteq X=\{1,\ldots,n\}$ and $\Lambda$ is a coupled
partition of $X\setminus\Delta$. An order is defined in 
\cite{Everitt-Fountain10}*{\S5.2} making $\TT$ a graded atomic lattice with
\begin{equation}
  \label{eq:8}
{\textstyle \rk (\Delta,\Lambda)
=|\Delta|+\sum(|\Lambda_{i1}|+|\Lambda_{i2}|-1)+\sum(|\Lambda_i|-1).
}
\end{equation}
Let $X(\Delta,\Lambda)\subseteq V$ be the subspace with 
$v=\sum t_iv_i\in X(\Delta,\Lambda)$ exactly when 
$t_i=0$ for $i\in\Delta$; $t_i=t_j$ if 
$i,j$ lie in the same block of $\Lambda$ (either uncoupled or in a
couple); and $t_i=-t_j$ if $i,j$ lie in different 
blocks of the same coupled block. Then the map 
$X(\Delta,\Lambda)\mapsto(\Delta,\Lambda)$ is a lattice isomorphism 
$\HH(B_n)\rightarrow\TT$. 

If $S$ is a set of atoms in $\HH(B_n)$, let $\Gamma_S$ be the graph with vertex set $\{1,\ldots,n\}$
and edges given by the scheme:
$$
\begin{pspicture}(0,0)(14.3,1.75)
\rput(3,.65){
\pscircle(0,.5){.15}\pscircle(2,.5){.15}
\rput(0,.1){$i$}\rput(2,.1){$j$}\rput(1.1,1){$(v_i-v_j)^\perp$}
\rput(1,-.4){(a)}
\psline[linewidth=.3mm]{-}(.15,.5)(1.85,.5)
}
\rput(7,.65){
\pscircle(0,.5){.15}\pscircle(2,.5){.15}
\rput(0,.1){$i$}\rput(2,.1){$j$}\rput(1.1,1){$(v_i+v_j)^\perp$}
\rput(1,-.4){(b)}
\psline[linewidth=.3mm,doubleline=true]{-}(.15,.5)(1.85,.5)
}
\rput(9,.25){
\pscircle(2,1){.15}
\rput(2,.6){$i$}\rput(2.75,1.25){$v_i^\perp$}
\rput(2,0){(c)}
\psbezier[showpoints=false,linewidth=.3mm]{-}(1.9,1.1)(1.3,2)(2.7,2)(2.1,1.1)
}
\end{pspicture}
$$
A circuit is a closed path of type (a) and (b) edges, and a circuit 
is \emph{odd\/} if it contains an odd number of 
(b) type edges, and \emph{even\/} otherwise. 

If $\bigvee S=X(\Delta,\Lambda)\in\HH(B_n)$, 
then a vertex $i$ of $\Gamma_S$ is contained 
in $\Delta$ if and only if for all $v=\sum t_iv_i\in X(\Delta,\Lambda)$ 
we have $t_i=0$. 
In particular, $i\in\Delta$ if and only if every vertex in the
connected component of $i$ is in $\Delta$. Otherwise, the vertices in this component
form a block or coupled block of $\Lambda$.

If a component contains a vertex $i$ incident
with an edge of type (c) above, then $t_i=0$,
and so $t_j=0$, for all $v\in X(\Delta,\Lambda)$
and all the vertices $j$ in the component. We thus have all the vertices of the component in $\Delta$.
Similarly if the connected component contains an odd circuit, for then $t_i=-t_i$
for each vertex $i$ in the circuit, and all the vertices are in $\Delta$ too. 
On the other hand,
suppose the component has no (c) edges and all circuits even. Label a vertex by $1$, and propagate
the labelling through the component by giving vertices joined by (a) edges the same label and
vertices joined by (b) edges labels that are negatives of each other. The absence of odd circuits
means this labelling can be carried out consistently. Label the remaining vertices of $\Gamma_S$
by $0$, to give an $v\in X(\Delta,\Lambda)$ with $t_i\not=0$ for $i$ some vertex of
our component, and so the component gives a block or coupled block.

We conclude that the vertices of a component of $\Gamma_S$ lie in $\Delta$ exactly when
the component has a (c) edge or contains an odd circuit. 

We claim that $S$ is independent exactly when each component of $\Gamma_S$ has one of
the forms
\begin{description}
\item[(B1).] a tree of (a) and (b) type edges together with at most one (c) type edge; or
\item[(B2).] contains a unique odd circuit, no (c) type edges, and removing one (hence any) edge
of the circuit gives a tree.
\end{description}
For, if the component contains no (c) edges and no odd circuits, then its vertices contribute
a block or coupled block to $\Lambda$, and by (\ref{eq:8}),
its edges are independent exactly when
there are $\sum(|\Lambda_{i1}|+|\Lambda_{i2}|-1)+\sum(|\Lambda_i|-1)$ of them;
in other words, when the number of edges is one less than the number of vertices.  
Thus we have a tree of (a) and (b) edges. 

If the component contains a (c)
edge then its vertices are in $\Delta$, and by (\ref{eq:8})
its edges are independent when there are the same number of them as there are vertices.
Removing the (c) edge gives a connected graph with number of edges one less than the number 
of vertices, hence a tree. The original component was thus a tree of (a) and (b) edges
with a single (c) edge.

Finally, if the component contains an odd circuit, then for the edges to be independent
it cannot have any (c) edges by the previous paragraph. Again the vertices are in $\Delta$
and so for independence the numbers of edges and vertices must be the same. Removing an edge from 
the circuit must give a tree as in the previous paragraph. In particular, the circuit is unique.

Now to the minimally dependent sets.
A \emph{branch vertex\/} of a tree of (a) and (b) edges is a vertex incident with at least three edges.
A \emph{line\/} is a tree of (a) and (b) edges containing at least one edge
and no branch vertices. It 
contains exactly two vertices (its \emph{ends\/}) incident with $<2$ edges.

\begin{proposition}
\label{section:idempotents:result375}
A set $S$ of atoms in $\HH(B_n)$ is minimally dependent precisely when $\Gamma_S$
has one of the forms:
\begin{enumerate}
\item an even circuit; or
\item an odd circuit wih a single (c) edge,
or two odd circuits intersecting only in a single vertex; or
\item a line, each end of which is incident with either a (c) edge or
an odd circuit intersecting the line only in this end vertex.
\end{enumerate}
\end{proposition}

Examples of the third kind are given in Figure \ref{fig0}. 

\begin{figure}
  \centering
\begin{pspicture}(0,0)(14.3,3)
\rput(8,0.5){
\pscircle(0,1){.15}\pscircle(1,0){.15}\pscircle(1,2){.15}\pscircle(2,1){.15}
\pscircle(3.5,1){.15}\pscircle(5,1){.15}
\psline[linewidth=.3mm,doubleline=true]{-}(0.106,1.106)(0.894,1.894)
\psline[linewidth=.3mm,doubleline=false]{-}(0.106,.894)(0.894,.106)
\psline[linewidth=.3mm,doubleline=false]{-}(1.106,.106)(1.894,.894)
\psline[linewidth=.3mm,doubleline=false]{-}(1.106,1.894)(1.894,1.106)
\psline[linewidth=.3mm,doubleline=false]{-}(2.15,1)(3.35,1)
\psline[linewidth=.3mm,doubleline=true]{-}(3.65,1)(4.85,1)
\rput(4,3){\rput{-90}(0,0){
\psbezier[showpoints=false,linewidth=.3mm]{-}(1.9,1.1)(1.3,2)(2.7,2)(2.1,1.1)}}
}
\rput(1,1){
\pscircle(0,.5){.15}\pscircle(1.5,.5){.15}\pscircle(3,.5){.15}
\pscircle(4.5,.5){.15}
\psline[linewidth=.3mm,doubleline=false]{-}(.15,.5)(1.35,.5)
\psline[linewidth=.3mm,doubleline=true]{-}(1.65,.5)(2.85,.5)
\psline[linewidth=.3mm,doubleline=false]{-}(3.15,.5)(4.35,.5)
\rput(1,-1.5){\rput{90}(0,0){
\psbezier[showpoints=false,linewidth=.3mm]{-}(1.9,1.1)(1.3,2)(2.7,2)(2.1,1.1)}}
\rput(3.5,2.5){\rput{-90}(0,0){
\psbezier[showpoints=false,linewidth=.3mm]{-}(1.9,1.1)(1.3,2)(2.7,2)(2.1,1.1)}}
}
\end{pspicture}
  \caption{Minimally dependent sets of atoms in the Coxeter arrangement $\HH(B)$: 
a line with (c) edges at each end (left) and a line with a (c) edge at one end and 
an odd circuit at the other intersecting the line only in this end vertex (right).}
  \label{fig0}
\end{figure}
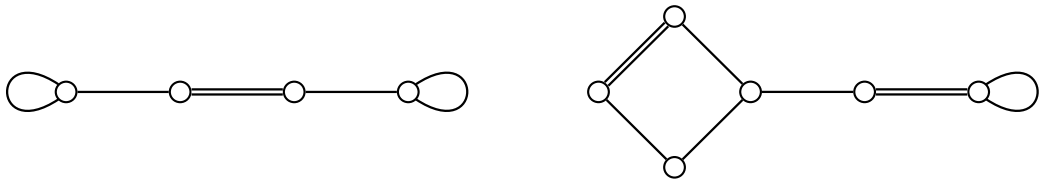

\begin{proof}
It is easy to see that for $S$ to be minimally dependent 
the graph $\Gamma_S$ must be connected.
We proceed by considering the number of 
type (c) edges in $\Gamma_S$. Firstly, there cannot be three or more such edges, for omitting one
would give a connected graph with at least two type (c) edges, whereas
the independent graphs in (B1) and
(B2) have at most one such edge.
If $\Gamma_S$ has two (c) edges then deleting one, $e$ say, gives $\Gamma_S\setminus\{e\}$
a graph of type (B1), hence $\Gamma_S$ is a tree with two type (c) edges attached. 
If this tree has a branch vertex, then there is a branch of the tree incident with no
type (c) edges. Deleting any edges of this branch gives a connected independent graph with 
two (c) edges attached, which cannot be. Thus $\Gamma_S$ is a line with two (c) edges
attached (it cannot be a single vertex with two (c) edges
attached). If an end vertex of the line 
has no (c) edge attached, then deleting the edge incident with this end
also gives a connected independent graph with two (c) edges; thus, each end vertex is attached
to a (c) edge and we have a line with a (c) edge at each end. This is clearly minimally dependent.

Now suppose $\Gamma_S$ contains a single (c) edge $e$, so that 
$\Gamma_S\setminus\{e\}$ has type (B1) with no (c) edges, or is of type (B2). In the first case
we would have $\Gamma_S$ a tree with a single (c) edge, which is independent, so we are
left with the possibility that $\Gamma_S\setminus\{e\}$ is of type (B2). Choose a line or single vertex
connecting the vertex incident with the (c) edge to a vertex of the odd circuit. 
This line can be chosen so as to intersect
the odd circuit in just a single vertex. Suppose $e'$ is an edge not contained in the odd circuit,
or the line, and is not the (c) edge. Then $\Gamma_S\setminus\{e'\}$ is an independent graph, one
component of which contains both an odd circuit and a (c) edge, which is a contradiction. Thus no
such $e'$ can exist, and $\Gamma_S$ \emph{is\/} a single vertex incident with an odd circuit and
a (c) edge as in part 2 of the Proposition,
or a line incident with an odd circuit and a (c) edge as in part 3.
In any case, these are minimally dependent.

Finally, we have the case where $\Gamma_S$ contains no (c) edges. Let $e$ be an edge of $\Gamma_S$,
so that $\Gamma_S\setminus\{e\}$ is a tree or of type (B2). In the
former, arguments like those above give that $\Gamma_S$ is just an even
circuit, which is minimally dependent.

This leaves the possibility that $\Gamma_S\setminus\{e\}$ is of type (B2), and in particular contains an
odd circuit $C_1$. Let $e'$ be an edge of this circuit.
As $\Gamma_S\setminus\{e,e'\}$ is a tree, we have that $\Gamma_S\setminus\{e'\}$
is not a tree but nevertheless independent, hence of type (B2) as well. Thus $\Gamma_S\setminus\{e'\}$
contains an odd circuit $C_2$, and as $e'\in C_1$ and $e'\not\in C_2$, these odd circuits
are distinct. 

We cosider the number of vertices $C_1$ and $C_2$ have in common. Suppose 
first that they have at least two common vertices. Each of $C_1$ and
$C_2$ gives two distinct paths connecting
these common vertices together, and the resulting four paths may or may not be distinct. If they
are distinct then removing an edge from one path gives an independent graph with two distinct
circuits--a contradiction. If these four paths are not distinct then there is a common path connecting
our two vertices as well as two other distinct paths connecting them. 
The common path has either an even or an odd
number of type (b) edges, and the other two paths an odd or even
number respectively, in order to make the
circuits odd overall. In either case, jettisoning the common path gives an independent graph containing
an even circuit--a contradiction again. Thus $C_1$ and $C_2$ can have at most one common vertex.

One common vertex and an edge $e$ not in either $C_1$ or $C_2$ would mean 
$\Gamma_S\setminus\{e\}$ is an independent graph with two circuits. Thus $\Gamma_S$ is a 
single vertex incident with two odd circuits $C_1,C_2$ intersecting only in this vertex. If no common vertices,
choose a line connecting $C_1,C_2$ and intersecting each in a single vertex. Again
we cannot have any other edges not in $C_1,C_2$ or this line, so $\Gamma_S$ is a line, each end 
vertex of which is incident with an odd circuit intersecting the line only in this end vertex.
\qed
\end{proof}  
\end{example}

\begin{example}[$\HH(D_n)$]
This is very similar to the previous example, so we will be briefer.
Let $V$ be as before and $\AA$ the hyperplanes with equations
$x_i\pm x_j=0$ for all $i\not= j$; equivalently, if $\Phi$ is the type
$D_n$ root system of Table \ref{table:roots1} then $\AA$ consists of the hyperplanes
$(v_i\pm v_j)^\perp$, with $W(\Phi)$ acting on $V$ by even signed permutations
of the $v_i$ (see also the end of \S\ref{section:renner:classical1}). 
In particular we have a sub-arrangement of $\HH(B_n)$. 
If $\TT^\circ\subset\TT$ consists of those $(\Delta,\Lambda)$ with
$|\Delta|\not=1$ then the isomorphism $\HH(B_n)\rightarrow\TT$
restricts to an isomorphism $\HH(D_n)\rightarrow\TT^\circ$. 
We have the same expression (\ref{eq:8}) for $\rk(\Delta,\Lambda)$
and the same conditions for an $v$ to lie in $X(\Delta,\Lambda)$ as in
the previous example. 

If $S$ is set of atoms in $\HH(D_n)$, let $\Gamma_S$ be the graph with vertex set $\{1,\ldots,n\}$
and edges of types (a) and (b) above. The arguments from here on are what you get
if you drop the (c) type edges from all the arguments in the previous section. Thus, the vertices
of a component of $\Gamma_S$ lie in $\Delta$ exactly when the component contains an
odd circuit. It follows that $S$ is independent when each component of $\Gamma_S$ is either
\begin{description}
\item[(D1).] a tree of (a) and (b) type edges; or
\item[(D2).] contains a unique odd circuit, removing one (hence any) edge
of which gives a tree.
\end{description}
The equivalent version of Proposition
\ref{section:idempotents:result375} gives $S$ minimally dependent when
$\Gamma_S$ is one of the forms:
\begin{enumerate}
\item an even circuit; or
\item two odd circuits intersecting only in a single vertex; or 
\item a line, each end of which is incident with an odd  circuit
  intersecting $\Gamma_S$ only in this end vertex.
\end{enumerate}  
\end{example}

We are now ready to give our presentations for the three classical reflection arrangements.
In each case we have replaced the \emph{(Idem3a)} family of relations given in Theorem 
\ref{section:idempotents:result350}
by a smaller set.

\paragraph{The intersection lattice $\HH(A_{n-1})$ or partition lattice $\Pi(n)$:}
has generators
$$
\begin{pspicture}(0,0)(14,1)
\rput(4.5,0){
\rput(0,0.2){
\pscircle(0,.5){.15}\pscircle(1.5,.5){.15}
\rput(0,.1){$i$}\rput(1.5,.1){$j$}
\psline[linewidth=.3mm]{-}(.15,.5)(1.35,.5)
}
\rput(3.5,.5){$(1\leq i\not=j\leq n)$}
}
\end{pspicture}
$$
and relations $(A0)$: the generators are commuting idempotents, and 
the relation $(A1)$ of 
Figure \ref{fig:typeA:arragement:relations}, which holds for all triples $\{i,j,k\}$.
See also \cite{Fitzgerald03}*{Theorem 2}. 

\begin{figure}
  \centering
\begin{pspicture}(0,0)(14,2)
\rput(0,0.5){
\rput(3,0.5){$(A1)$}
\rput(4.5,0.5){\BoxedEPSF{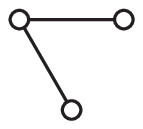 scaled 900}}
\rput(5.75,0.5){$=$}
\rput(7,0.5){\BoxedEPSF{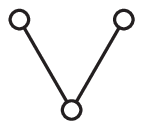 scaled 900}}
\rput(8.25,0.5){$=$}
\rput(9.5,0.5){\BoxedEPSF{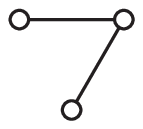 scaled 900}}
}
\end{pspicture}  
  \caption{Relations for the intersection lattice $\HH(A_n)$.}
  \label{fig:typeA:arragement:relations}
\end{figure}

We have pushed the graphical technique to its logical conclusion here. 
When a word in the generators 
is expressed graphically as in Figure
\ref{fig:typeA:arragement:relations}
it is not possible to tell the order in which the 
generators appear, but this doesn't matter as they commute. A pleasant consequence
of the commuting generators is that relations like $(A1)$ can be applied to a fragment of a
graph while leaving the rest untouched.
Observe that multiplying together any two of the graphs in Figure \ref{fig:typeA:arragement:relations}
gives an \emph{(Idem3b)} relation of the form: ``a triangle equals a triangle minus an edge''.

To see this presentation, the
\emph{(Idem3a)} relations for $\HH(A_{n-1})$ 
are
of the form $\Gamma_S=\Gamma_S\setminus\{e\}$ where $\Gamma_S$ is a circuit
and $e$ some edge of it.
Given such a circuit, repeated applications of the relations $(A1)$,
as in say, 
$$
\begin{pspicture}(0,0)(14.3,2)
\rput(.95,0){
\rput(2,1){\BoxedEPSF{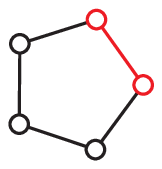 scaled 900}}
\rput(4.65,1){\BoxedEPSF{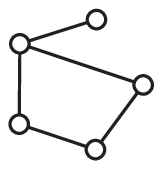 scaled 900}}
\rput(7.3,1){\BoxedEPSF{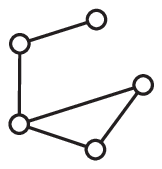 scaled 900}}
\rput(9.95,1){\BoxedEPSF{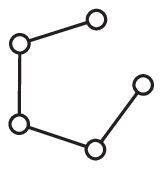 scaled 900}}
\rput(3.3,1){$=$}\rput(6.05,1){$=$}\rput(8.65,1){$=$}
\rput(2.6,1.4){${\red e}$}
\rput(0.6,1){$\Gamma_S=$}\rput(11.75,1){$=\Gamma_S\setminus\{e\}$}
}
\end{pspicture}
$$
allow us to move one end of $e$ anticlockwise around the circuit until
we have a triangle, from which the edge can then be removed
(using the ``triangle equals a triangle minus an edge'' relation
mentioned above). Thus the \emph{(Idem3a)}
relations follow from the relations $(A0)$-$(A1)$.

\paragraph{The intersection lattice $\HH(B_n)$:} has generators 
$$
\begin{pspicture}(0,0)(14,1.5)
\rput(0,-1){
\rput(2.6,1.5){
\pscircle(0,.5){.15}\pscircle(1.5,.5){.15}
\rput(0,.1){$i$}\rput(1.5,.1){$j$}
\psline[linewidth=.3mm]{-}(.15,.5)(1.35,.5)
}
\rput(5,1.5){
\pscircle(0,.5){.15}\pscircle(1.5,.5){.15}
\rput(0,.1){$i$}\rput(1.5,.1){$j$}
\psline[doubleline=true,linewidth=.3mm]{-}(.15,.5)(1.35,.5)
\rput(3.5,.5){$(1\leq i\not=j\leq n)$}
}
\rput(8.75,1){
\pscircle(2,1){.15}
\rput(2,.6){$i$}
\psbezier[showpoints=false,linewidth=.3mm]{-}(1.9,1.1)(1.3,2)(2.7,2)(2.1,1.1)
\rput(3.75,1){$(1\leq i\leq n)$}
}
}
\end{pspicture}
$$
and relations $(B0)$: the generators are commuting idempotents
and the $(B1)$-$(B4)$ of
Figure \ref{fig:typeB:arragement:relations}.
The relations $(B1)$, $(B2)$ and $(B4)$ hold for all triples
$\{i,j,k\}$ and $(B3)$ holds for all pairs $\{i,j\}$.

That these relations hold in $\HH(B_n)$
follows by checking that the corresponding subspaces are the same, i.e.: if $\Gamma_S,\Gamma_{S'}$
are two graphs differing only by applying one of these relations to some fragment, then 
$\bigvee S=\bigvee S'$ in the intersection lattice. 
For example, let the vertices in the relations $(B4)$
be labelled anti-clockwise as $i,j$ and $k$. 
If $v=\sum t_iv_i\in\bigvee S$ (the left hand side)
then we have $t_i=t_j=t_k$ and 
$t_i=-t_k$, hence $t_i=t_j=t_k=0$; similarly for
$v\in\bigvee S'$. As all the other $t$'s are the same we get our equality. 

The presentation follows by showing that if $\Gamma_S$ is one of the graphs in 
Proposition \ref{section:idempotents:result375}, 
and $e$ is some edge of it, then the \emph{(Idem3b)} relation $\Gamma_S=\Gamma_S\setminus\{e\}$
follows from $(B0)$-$(B4)$. 
We start with a series of relations that can be deduced 
from (B0)-(B4): 
\begin{description}
\item[(i).] From the relations $(B1)$ and the argument used in the $\HH(A_{n-1})$
case, if $\Gamma$ is a circuit of type (a) edges, then we have $\Gamma=\Gamma\setminus\{e\}$
for any edge $e$. 
\item[(ii).] If $\Gamma$ is a circuit of type (a) and (b) edges then equality of the first and last
fragment in the relations $(B2)$ 
allows us to move the (b) edges across the (a) edges to give a circuit composed 
entirely of (b) edges, as for example in:
$$
\begin{pspicture}(0,0)(14.3,2)
\rput(3,0){
\rput(2.4,1){\BoxedEPSF{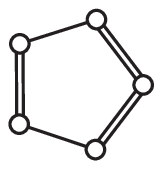 scaled 900}}
\rput(4.85,1){\BoxedEPSF{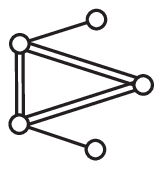 scaled 900}}
\rput(3.6,1){$=$}
\rput(.9,1){$\Gamma=$}\rput(6.4,1){$=\Gamma'$}
}
\end{pspicture}
$$
\item[(iii).] A fragment of $2m$ consecutive (b) edges can, by the relations $(B2)$, be
replaced by a connected fragment containing $m$ consecutive (a) edges:
$$
\begin{pspicture}(0,0)(14.3,1)
\rput(3.5,.5){\BoxedEPSF{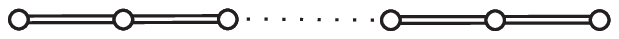 scaled 900}}
\rput(10.5,.6){\BoxedEPSF{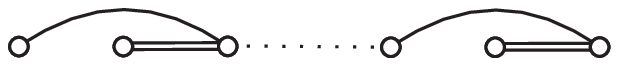 scaled 900}}
\rput(0,0.5){$\Gamma=$}\rput(14,0.5){$=\Gamma'$}\rput(7,0.5){$=$}
\end{pspicture}
$$
\item[(iv).] A fragment of consecutive (a) edges can be augmented:
$$
\begin{pspicture}(0,0)(14.3,1)
\rput(3.5,0.5){\BoxedEPSF{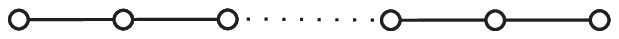 scaled 900}}
\rput(10.5,0.5){\BoxedEPSF{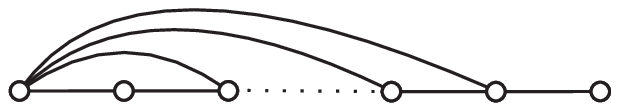 scaled 900}}
\rput(0,0.5){$\Gamma=$}\rput(14,.5){$=\Gamma'$}\rput(7,0.5){$=$}
\end{pspicture}
$$
using the \emph{(Idem3a)} ``a triangle equals a triangle minus an edge'' relations that follow from 
$(B1)$.
\item[(v).] Let $\Gamma$ be a connected graph containing a (c) edge, and $\Gamma'$
a graph with the same vertex set, \emph{each\/} of which is incident with a (c) edge,
and having no other edges. Then repeated applications of the relations $(B3)$ give
$\Gamma=\Gamma'$. 
\item[(vi).]  Finally, let $\Gamma$ contain an odd circuit. Applying (ii) gives the fragment 
below left:
$$
\begin{pspicture}(0,0)(14.3,1)
\rput(3.7,0.5){\BoxedEPSF{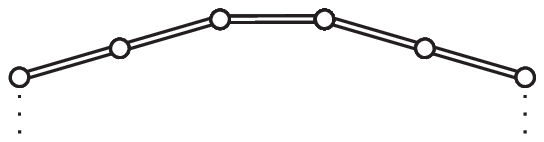 scaled 900}}
\rput(10.3,0.5){\BoxedEPSF{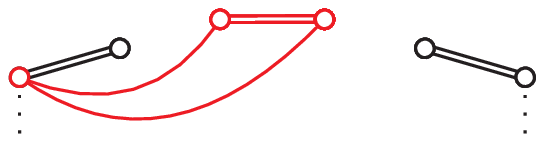 scaled 900}}
\rput(7,0.5){$=$}
\end{pspicture}
$$
and (iii)-(iv) transform this into the form on the right, containing the triangle in red. 
Note that the connectedness of $\Gamma$ is not affected by these moves. The triangle
fragment is the left hand side of the relation (B4), applying which gives a graph
with each component incident with a (c) edge. Thus, if $\Gamma'$ is a graph
on the same vertex set as $\Gamma$, each of which is incident with a (c) edge, and having
no other edges, then by (v) we have $\Gamma=\Gamma'$.
\end{description}

\begin{figure}
  \centering
\begin{pspicture}(0,0)(14,3)
\rput(-4,1.75){
\rput(4.75,0.5){\BoxedEPSF{ref.monoids2a.fig19.eps scaled 900}}
\rput(5.75,0.5){$=$}
\rput(7,0.5){\BoxedEPSF{ref.monoids2a.fig20.eps scaled 900}}
\rput(8.25,0.5){$=$}
\rput(9.25,0.5){\BoxedEPSF{ref.monoids2a.fig21.eps scaled 900}}
}
\rput(6.3,2.25){$(B1)$}
\rput(4,1.75){
\rput(4.75,0.5){\BoxedEPSF{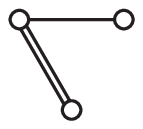 scaled 900}}
\rput(5.75,0.5){$=$}
\rput(7,0.5){\BoxedEPSF{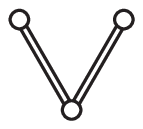 scaled 900}}
\rput(8.25,0.5){$=$}
\rput(9.25,0.5){\BoxedEPSF{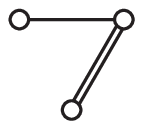 scaled 900}}
}
\rput(7.7,2.25){$(B2)$}
\rput(-4,0.25){
\rput(4.75,0.5){\BoxedEPSF{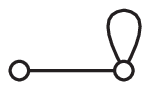 scaled 900}}
\rput(5.75,0.5){$=$}
\rput(7,0.5){\BoxedEPSF{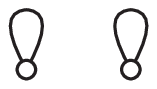 scaled 900}}
\rput(8.25,0.5){$=$}
\rput(9.25,0.5){\BoxedEPSF{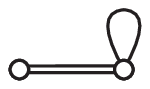 scaled 900}}
}
\rput(6.3,0.75){$(B3)$}
\rput(5,0.25){
\rput(4.5,0.5){\BoxedEPSF{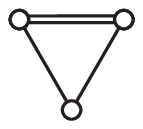 scaled 900}}
\rput(5.75,0.5){$=$}
\rput(7,0.5){\BoxedEPSF{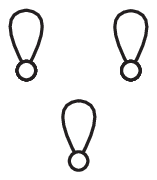 scaled 900}}
}
\rput(7.7,0.75){$(B4)$}
\end{pspicture}  
  \caption{Relations for the intersection lattice $\HH(B_n)$.}
\label{fig:typeB:arragement:relations}
\end{figure}

Now let $\Gamma_S$ be a graph of the form given in part 3 of Proposition 
\ref{section:idempotents:result375} and $e\in\Gamma_S$ some edge. Then every component of
both $\Gamma$ and $\Gamma_S\setminus\{e\}$ contains an odd circuit and/or a (c) edge,
hence by (v)-(vi) there is a $\Gamma'$ such that $\Gamma=\Gamma'=\Gamma\setminus\{e\}$
can be deduced from $(B1)$-$(B4)$. Similarly for $\Gamma_S$ of the form given in part
2 of Proposition \ref{section:idempotents:result375}. 

Finally, let $\Gamma_S$ be an even circuit as in part 1 of Proposition \ref{section:idempotents:result375}
and $e$ a type (b) edge in this circuit (if no such $e$ exists then we are done by (i)). 
Applying (ii) gives the fragment below left, and this equals the fragment 
at the right by (iii):
$$
\begin{pspicture}(0,0)(14.3,1)
\rput(3,0.5){\BoxedEPSF{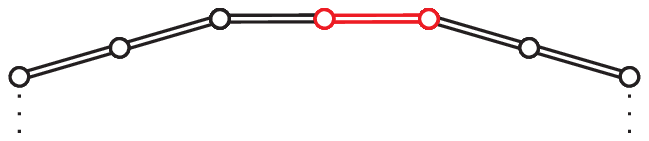 scaled 900}}
\rput(11,0.5){\BoxedEPSF{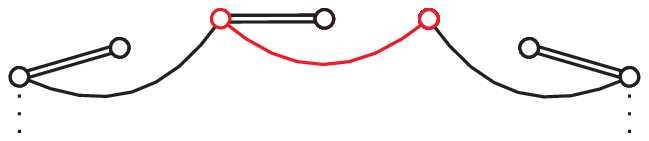 scaled 900}}
\rput(3.5,.6){${\red e}$}\rput(11,.3){${\red e'}$}
\rput(7,0.5){$=$}
\end{pspicture}
$$
Applying (i), we can remove the edge $e'$ and then run the process backwards to get 
$\Gamma_S\setminus\{e\}$.
If instead $e$ is a type (a) edge then the argument is similar.

\paragraph{The intersection lattice $\HH(D_n)$:} has generators
$$
\begin{pspicture}(0,0)(14,1)
\rput(0,-1.5){
\rput(2.6,1.5){
\pscircle(0,.5){.15}\pscircle(1.5,.5){.15}
\rput(0,.1){$i$}\rput(1.5,.1){$j$}
\psline[linewidth=.3mm]{-}(.15,.5)(1.35,.5)
}
\rput(5,1.5){
\pscircle(0,.5){.15}\pscircle(1.5,.5){.15}
\rput(0,.1){$i$}\rput(1.5,.1){$j$}
\psline[doubleline=true,linewidth=.3mm]{-}(.15,.5)(1.35,.5)
\rput(3.5,.5){$(1\leq i\not=j\leq n)$}
}
}
\end{pspicture}
$$
and relations $(D0)$: the generators are commuting idempotents,
$(B1)$ and $(B2)$ of Figure \ref{fig:typeB:arragement:relations},
and $(D1)$-$(D3)$ of
Figure \ref{fig:typeD:arragement:relations}.
The relations $(D1)$ and $(D3)$ in Figure \ref{fig:typeD:arragement:relations} hold for all triples
$\{i,j,k\}$ and the relations $(D2)$ for all $4$-tuples $\{i,j,k,\ell\}$.

The proof of the presentation is very similar to the $\HH(B_n)$ case: check first that the relations hold in 
$\HH(D_n)$ and then show that if $\Gamma_S$ is one of the minimally
dependent graphs listed for $\HH(D_n)$, then
$\Gamma_S=\Gamma_S\setminus\{e\}$ for any edge $e$. 

Note that the relations
(i)-(iv) of the type $\HH(B_n)$ case hold here as well, as they use only the relations
$(B1)$-$(B2)$. Relation (v) in $\HH(B_n)$ is replaced by ($\text{v}'$): if $\Gamma$ connected
contains a pair of vertices that are connected by both and (a) type edge an a (b) type
edge, then repeated applications of $(D3)$ gives $\Gamma=\Gamma'$, where 
$\Gamma'$ has the same vertex set as $\Gamma$, but every edge 
(type (a) or type (b)) of $\Gamma$
has been replaced by a pair consisting of an (a) type edge and a (b) type edge.

The construction in (vi) holds up to the appearance of the red 
triangle that appears on the lefthand side of $(D1)$ above. This can be replaced by the righthand
side of $(D1)$, and repeated application of $(D3)$ gives that if $\Gamma$ connected has an odd circuit then
$\Gamma=\Gamma'$, where $\Gamma'$ is just $\Gamma$ with every edge replaced by 
a pair consisting of an (a) and a (b) edge.

Thus, if $\Gamma_S$ contains two odd circuits joined by a line,
then every component of $\Gamma_S\setminus\{e\}$
does too.
We get $\Gamma=\Gamma'$ and $\Gamma\setminus\{e\}=\Gamma^{''}$ as in the
previous paragraph. Finally, $(D1)$ above gives $\Gamma'=\Gamma^{''}$.
If $\Gamma_S$ is an
even circuit the argument is identical to the $\HH(B_n)$ case.

\section{A presentation for reflection monoids}
\label{section:presentation}

We now return to the specifics of reflection monoids and give a
presentation (Theorem \ref{presentations:result1000} below) for those
reflection monoids $M(W,\cS)$ where $W\subset GL(V)$ is a finite reflection group
and $\cS$ a graded atomic system of subspaces of $V$ for $W$. We also give the
analogous presentation when $\cS$ is a system of subsets of some set $E$.
The main technical tool is the presentation for factorizable
inverse monoids found in \cite{Easdown05}.

\begin{figure}
  \centering
\begin{pspicture}(0,0)(14,3)
\rput(-4,1.75){
\rput(4.75,0.5){\BoxedEPSF{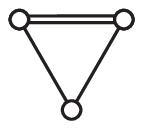 scaled 900}}
\rput(5.75,0.5){$=$}
\rput(7,0.5){\BoxedEPSF{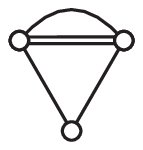 scaled 900}}
}
\rput(4.3,2.25){$(D1)$}
\rput(4,1.75){
\rput(4,0.5){\BoxedEPSF{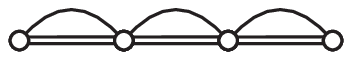 scaled 900}}
\rput(6.05,0.5){$=$}
\rput(8,0.5){\BoxedEPSF{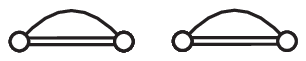 scaled 900}}
}
\rput(5.7,2.25){$(D2)$}
\rput(3,0){
\rput(-4,0.25){
\rput(4.75,0.5){\BoxedEPSF{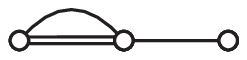 scaled 900}}
\rput(6.3,0.5){$=$}
\rput(7.75,0.5){\BoxedEPSF{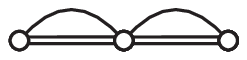 scaled 900}}
\rput(9.3,0.5){$=$}
\rput(10.75,0.5){\BoxedEPSF{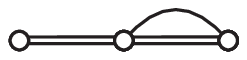 scaled 900}}
}
\rput(8.4,0.75){$(D3)$}
}
\end{pspicture}  
  \caption{Relations for the intersection lattice $\HH(D_n)$: relations $(B1)$ and $(B2)$ from 
Figure \ref{fig:typeB:arragement:relations} together with relations $(D1)$-$(D3)$ above.}
\label{fig:typeD:arragement:relations}
\end{figure}

Let $V$ be a finite dimensional real vector space and $W\subset GL(V)$
a finite real reflection 
group with generating reflections $S$ and 
$T=W^{-1}SW$ the full set of reflections. 
Let $\cA=\{H_t\subset V\,|\,t\in T\}$ be the reflecting hyperplanes of $W$.
Suppose also that:
\begin{description}
\item[(P1).] $\cS$ is a finite system of subspaces in $V$ for $W$, and that via $X\leq Y$
if and only if $X\supseteq Y$, the system is a graded (by $\rk X=\codim X:=\dim
V-\dim X$) atomic $\vee$-semilattice with atoms $A$.
We have $\rk\cS=\dim V-\dim\bigvee_\cS X$.
The $W$-action preserves the grading, and in particular we have $A
W=A$.
If $a_1,\ldots,a_k$ are distinct atoms let $O_k$ be a set of orbit
representatives for the $W$-action
$\{a_1,\ldots,a_k\}\stackrel{w}{\mapsto}\{a_1w,\ldots,a_kw\}$.
\item[(P2).] We use the Greek
  equivalents of Roman letters to indicate a fixed word
for an element in terms of generators. In particular, 
the reflection group $W$ has a presentation with generators the $s\in S$ and 
relations 
$(st)^{m_{st}}=1$ for $s,t\in S$, where 
$m_{st}=m_{ts}\in\Z^{\geq 1}\cup\{\infty\}$ with $m_{st}=1$ if and only if
$s=t$. For each $w\in W$ we fix a word
$\omega$ for $w$ in the reflections $s\in S$ 
(subject to $\ss=s$).
\item[(P3).] Now for the action of $W$ on $\cS$: 
For each $a\in A$ fix a representative atom $a'\in O_1$ and a $w\in W$
with $a=a'w$ subject to $w=1$ if $a\in O_1$. Now define 
the word $\aa$ to be $\omega^{-1}a'\omega$.
If $w$ is an arbitrary element of $W$ and $a\in A$ then by $\aa^\omega$ we mean the
word obtained in this way for $aw\in A$. Note that this is \emph{not\/} necessarily 
$\omega^{-1}a\omega$.
For $e\in\cS$, fix a join $e=\bigvee a_i\,(a_i\in A)$ and define $\ve:=\prod\aa_i$. 
\item[(P4).] Let $\{H_{s_1},\ldots,H_{s_\ell}\}$ be representatives for the $W$-action on $\cA$
with the $s_i\in S$. 
For example, drop the even labeled edges in the Coxeter symbol for $W$
and choose one $s$ from each component of the resulting graph.
For each $i=1\,\ldots,\ell$ consider the set of $X\in\cS$ with the
property that $H_{s_i}\supseteq X$. If this set is non-empty them form
the pairs $(e,s_i)$ for each $e\in\cS$ minimal in this set. Let
\emph{Iso\/} be the set of all such pairs.
\end{description}

With the notation established we have:

\begin{theorem}
\label{presentations:result1000}
Let $W\subset GL(V)$ be a finite real reflection group and $\cS$ a
graded atomic system of
subspaces for $W$. Then
the reflection monoid $M(W,\cS)$ has a presentation with
\begin{align*}
\text{generators:\hspace{1em}}&s\in S,a\in O_1.&&\\
\text{relations:\hspace{1em}}
&(st)^{m_{st}}=1,\,(s,t\in S),&&\text{(Units)}\\
&a^2=a,\,(a\in O_1),&&\text{(Idem1)}\\
&\aa_1 \aa_2=\aa_2 \aa_1,\,(\{a_1,a_2\}\in O_2),&&\text{(Idem2)}\\
&\aa_1\ldots\aa_{k-1}=\aa_1\ldots\aa_{k-1}\aa,\,(\{a_1,\ldots,a_{k-1},a\}\in O_{k})&&\\
&\hspace*{0.5cm}\text{ with }a_1,\ldots,a_{k-1}\,,(3\leq k\leq\rk\cS)\text{ independent and }a\leq\bigvee a_i,&&\text{(Idem3)}\\
&s\aa=\aa^{s}s,\,(s\in S,a\in A),&&\text{(RefIdem)}\\
&\ve s=\ve,\,(e,s)\in\mathit{Iso}.&&\text{(Iso)}\\
\end{align*}
\end{theorem}

To prove Theorem \ref{presentations:result1000} we start with
a presentation for an arbitrary factorizable inverse monoid \cite{Easdown05}*{Theorem 6}, interpret
the various ingredients in the setting of a reflection monoid, and then  remove relations and generators.

Suppose then that $M$ is a factorizable
inverse monoid with units $W=W(M)$ and idempotents
$E=E(M)$. Let $\langle S\,|\,R_W\,\rangle$ and $\langle A\,|\,R_E\,\rangle$
be monoid presentations for $W$ and $E$. 
For $w\in W$, fix a word $\omega$ for $w$ in the
$s\in S$ and similarly for $e\in E$  fix a word $\varepsilon$ in the $a\in A$, with the
usual conventions applying when $w\in S$ and $e\in A$.
For $w\in W$ and $e\in E$ we have $w^{-1}ew\in E$, and by
$\ve^\omega$ we mean the chosen word for $w^{-1}ew$ in the
$a\in A$ (it turns out
that we will only have need for the notation $\ve^\omega$ in the case
that $w\in S$ and $e\in A$).
For each $e\in E$ let $W_e=\{w\in W\,|\,ew=e\}$ be the idempotent stabilizer, and $S_e\subseteq W_e$ 
a set of monoid generators for $W_e$. 

\begin{theorem}[\cite{Easdown05}*{Theorem 6}]
\label{presentations:result2000}
The factorizable inverse monoid $M$ has a presentation with,
\begin{align*}
\text{generators:\hspace{1em}}&s\in S,a\in A,&&\\
\text{relations:\hspace{1em}}
&R_W,R_E,&&\\
&sa=a^ss,\,(s\in S,a\in A),&&\\
&\ve\omega=\ve,\,(e\in E,w\in S_e).&&\\
\end{align*}
\end{theorem}

This theorem will give presentations for \emph{arbitrary\/} reflection
monoids. Here, $W$ is a reflection group, and in the real case
we take the standard Coxeter presentation for $W$. 
If moreover $W$ is finite, then by Steinberg's theorem the $W_e$ are
parabolic subgroups of $W$, so $S_e$ consists of reflections.
Although the presentation thus obtained looks much the same as that in
Theorem \ref{presentations:result2000}, it is, in fact, more precise
and economical. However, under the assumptions
(P1)-(P4) we can be much more explicit and as all the natural examples
satisfy these conditions, we concentrate on this case.

We now interpret the various ingredients in the presentation.
The $s$ of Theorem \ref{presentations:result2000} are the
generating reflections $s$ of the reflection group $W$.
Identifying $X\in\cS$ with the partial identity on $X$, the $A$ 
of Theorem \ref{presentations:result2000} are the
atomic subspaces $A$ of the system $\cS$.
If $s\in S$ and $a\in A$ then $as\in A$, 
so that $\aa^s$ is just another one of the symbols in $A$, and we write 
$a^s$ for this symbol.
If $X\in\cS$ then
$W_X$ is the isotropy group $W_X=\{w\in W\,|\,yw=y\text{ for all }y\in
X\}$, a group generated by reflections, 
and so we can take $S_X$ to consist of those $t\in T$ with $H_t\supseteq X$. 

As an intermediate step we thus have the presentation for $M(W,\cS)$ with
\begin{align*}
\text{generators:\hspace{1em}}&s\in S,a\in A,&&\\
\text{relations:\hspace{1em}}&(st)^{m_{st}}=1\,(s,t\in S),&&(a)\\
&a^2=a\,(a\in A),&&(b)\\
&a_1a_2=a_2a_1\,(a_1,a_2\in A),&&(c)\hspace*{1cm}(*)\\
&a_1\ldots a_{k-1}=a_1\ldots a_{k-1}a\,(a_i,a\in A)&&(d)\\
&\hspace*{0.5cm}\text{ for }a_1,\ldots,a_{k-1}\,,(3\leq k\leq\rk\cS)\text{ independent and }a\leq\bigvee a_i.&&\\
&sa=a^s s\,(s\in S,a\in A),&&(e)\\
&\ve\tau=\ve,\,(e\in\cS, t\in S_e).&&(f)\\
\end{align*}

\begin{remark*}
There are many more generators and relations in $(*)$ than in Theorem \ref{presentations:result1000}. 
For example, with the Boolean reflection monoid of type $A_{n-1}$ (\S\ref{section:popova:boolean}),
the presentation $(*)$ gives $n-1$ unit generators,
$n$ idempotent generators,
$n$ relations of the form $\aa^2=\aa$ and
$n(n-1)$ relations of the form 
$\aa_1\aa_2=\aa_2\aa_1$.
There are $2^n$ subspaces in the Boolean system, and if 
$Y=X(J)$ is one of them,
then a standard generating set for $W_Y$ has $n-|J|-1$ reflecting generators, for
a total of $2^{n-1}n(n-1)$ relations of the form 
$\ve\tau=\ve$. Theorem \ref{presentations:result1000}
on the other hand gives $n-1$ unit generators, a single idempotent generator, a single \emph{(Idem1)\/}
relation, a single \emph{(Idem2)\/}  relation, and a single \emph{(Iso)\/} relation.
\end{remark*}

Deducing Theorem \ref{presentations:result1000} is now a matter of thinning out
relations and generators from $(*)$, using the $W$-action on $\cS$.

\begin{lemma}\label{presentations:result300}
For $j=1,\ldots,k$ let $a_j\in A$ and $a'_j=a_j\cdot w$ with 
$W_i(x_1,\ldots,x_k)$ for $i=1,2$ words in the free monoid on the $x_i$. Then
the relation $W_1(a_1,\ldots,a_k)=W_2(a_1,\ldots,a_k)$ 
and the relations (e) imply the relation
$W_1(a'_1,\ldots,a'_k)=W_2(a'_1,\ldots,a'_k)$.
\end{lemma}


\begin{proof}
If $s\in S,a\in A$ and $as=a'\in A$, then relations $(e)$ of $(*)$ give a relation 
$sa=a's$, hence
$a'=sas$. By induction, if $a'=aw$ for some $w\in W$ we have the relation 
$a'=\omega^{-1}a\omega$.
Thus, for all $j$ we have 
$a'_j=\omega^{-1}a_j\omega$ 
so that
$W_i(a'_1,\ldots,a'_k)=W_i(\omega^{-1}a_1\omega,\ldots,\omega^{-1}a_k\omega)
=\omega^{-1}W_i(a_1,\ldots,a_k)\omega$, and the result
follows. 
\qed
\end{proof}

\begin{lemma}
\label{presentations:result400}
Let $e\in\cS$ and $t\in T$ be such that $H_t\supseteq e$. Then there is an 
$(e',s)\in\mathit{Iso}$ and $w\in W$ with $t=w^{-1}sw$ and $e'w\supseteq e$ and
the relation $\ve\tau=\ve$ of $(*)$ implied by the 
(Iso) relation $\ve'\ss=\ve'$ and the relations (a)-(e).
\end{lemma}

\begin{proof}
There is an $s$ with $H_{s}w=H_t$. Thus $H_{s}\supseteq ew^{-1}$ and there
is an $e'$ with $H_{s}\supseteq e'\supseteq ew^{-1}$
and $(e',s)\in\mathit{Iso}$. This pair satisfies the requirements of
the Lemma and moreover,
$e=e\cdot e'w$
in (the monoid) $\cS$, 
so that the relations $(a)$-$(e)$ of $(*)$
give $\ve=\ve\omega^{-1}\ve'\omega$ and
$\tau=\omega^{-1}s\omega$. Thus
$$
\ve\tau
=\ve\omega^{-1}\ve'\omega\omega^{-1}s\omega
=\ve\omega^{-1}\ve' s\omega
=\ve\omega^{-1}\ve'\omega
=\ve.
$$
\qed
\end{proof}

\begin{proof}[of Theorem \ref{presentations:result1000}]
Lemma \ref{presentations:result400} allows us to thin
out family $(f)$ in $(*)$ to give the \emph{(Iso)\/} relations of Theorem \ref{presentations:result1000}.
Lemma \ref{presentations:result300} allows us to thin out the families $(b)$-$(d)$
in $(*)$ to involve just the orbit representatives as in Theorem
\ref{presentations:result1000}. 
The \emph{(Units)\/} relations we leave
untouched. 
Finally a generator $a\in A$
can be expressed as $\aa=\omega^{-1}a_0\omega$ for some $a_0\in O_1$, and this
allows us to thin these generators, and replace each occurrence of $a$
in the \emph{(RefIdem)\/} relations by
$\aa$.
\qed
\end{proof}

\begin{remarks*} There are a number of variations on Theorem
\ref{presentations:result1000}: 
  \begin{enumerate}
 \item There is a completely analogous presentation when $M(W,\cS)$ is
   a monoid of partial permutations. Let $E$ be a set, $W=(W,S)$ a
   (not necessarily finite) Coxeter system acting faithfully on $E$
   and $\cS$ a graded atomic system of
   subsets of $E$ for $W$. For $t\in T=W^{-1}SW$ let $H_t\subseteq E$ be
   the set of fixed points of $t$ and $\AA=\{H_t\,|\,t\in T\}$. Notice
   that $H_t\cdot w=H_{w^{-1}tw}$ so there is an induced $W$-action on
   $\AA$. There is one condition that we must impose: for any
   $e\in\cS$ the isotropy group $W_e$ is generated by reflections;
   indeed, by the $t\in T$ with $H_t\supseteq e$. Adapting (P1)-(P4),
   the presentation of Theorem \ref{presentations:result1000} now goes
   straight through for $M(W,\cS)$.
\item If $\cS$ is a geometric lattice, then the \emph{(Idem3)\/} relations
of Theorem \ref{presentations:result1000} can be replaced by 
\begin{align*}
&\wh{\aa}_1\ldots\aa_{k}=\cdots=\aa_1\ldots\wh{\aa}_{k},\,(a_1,\ldots,a_k\in O_{k})&&\\
&\hspace*{1cm}\text{ with }a_1,\ldots,a_k\text{ minimally dependent
  and }3\leq k\leq \rk\cS.&&\text{\emph{(Idem3a)}}\\
\end{align*}
\item 
The sets $O_k$ of (P1) can sometmes be hard to describe; their
definitions can be varied in two ways--either by changing the group
or the set on which it acts (or both, as in
\S\ref{section:renner:classical2}).
This has the effect of introducing more relations. 
  If $W'$ is a subgroup of $W$ we can replace
  $O_k$ by $O'_k$, a set of orbit representatives for the $W$-action
  restricted to $W'$. On the other hand, it may be more convenient to
  describe orbit
  representatives for the $W$-action 
$(a_1,\ldots,a_k)\stackrel{w}{\mapsto}(a_1\cdot w,\ldots,a_k\cdot w)$ on
ordered \emph{$k$-tuples\/} of distinct atoms. The commuting of the
idempotents then allows us to return to sets $\{a_1,\ldots,a_k\}$.
\item We have kept things very explicit with $W=(W,S)$ a Coxeter
  system, but Theorem \ref{presentations:result1000} can easily be
  generalized without changing its form too seriously. Let $(W,S)$
  with $W$ an arbitrary group and $S$ an arbitrary set of
  generators. The definitions of $T,H_t$ and $\AA$ are the same, and
  we impose the condition that for $e\in\cS$ the group $W_e$ is
  generated by elements of $T$. The only change to Theorem
  \ref{section:renner:classical2} is then to the \emph{(Units)}
  relations, which are replaced by a set of relators for $W$.
 \item All the reflection groups in this paper are over $k=\R$. For
   an arbitrary field $k$ the \emph{(Units)} relations will be different
   (reflection groups are Coxeter groups only when $k=\R$) and it may
   be that the $W_e$ are not generated by reflections,
   although it is known that they are when $k=\C$ \cite{Steinberg64} or a finite
   $\F_q$ \cite{Nakajima79}.
  \end{enumerate}
\end{remarks*}

\section{Boolean reflection monoids}
\label{section:popova:boolean}

In \cite{Everitt-Fountain10}*{\S5} we introduced the Boolean reflection monoids,
formed from a  Weyl group $W(\Phi)$ for $\Phi=A_{n-1}, B_n$ or $D_n$,
and the Boolean system $\BB$.  
In this section we find the presentations given by Theorem \ref{presentations:result1000}. 
In particular,
we recover Popova's presentation \cite{Popova61}
for the symmetric inverse monoid by interpreting it as the Boolean reflection
monoid of type $A$. 

Recall from \cite{Everitt-Fountain10}*{\S5} that $V$ is a Euclidean space
with basis $\{v_1,\ldots,v_n\}$ and inner product $(v_i,v_j)=\delta_{ij}$,
with $\Phi\subset V$ a root system 
from Table \ref{table:roots1} and $W(\Phi)\subset GL(V)$ the associated 
reflection group. The Coxeter generators for $W(\Phi)$ are given in the third column 
of Table \ref{table:roots1}: let $s_i$ $(1\leq i\leq n-1$) be the reflection
in the hyperplane orthogonal to $v_{i+1}-v_i$, with $s_0$ the reflection in $v_1$
(type $B$) or in $v_1+v_2$ (type $D$).

For $J\subseteq X=\{1,\ldots,n\}$ let
$$
X(J)=\bigoplus_{j\in J}\R v_j\subseteq V,  
$$
and $\BB=\{X(J)\,|\,J\subseteq X\}$ with $X(\varnothing)=\mathbf{0}$.
Then by \cite{Everitt-Fountain10}*{\S5}, $\BB$ is a system in $V$
for $W(\Phi)$--the \emph{Boolean\/} system--and $M(\Phi,\BB):=M(W(\Phi),\BB)$ is called the
\emph{Boolean reflection monoid of type $\Phi$}.

We've obviously seen the poset $(\BB,\supseteq)$ before: it is isomorphic
to the Boolean lattice $\BB_X$
of \S\ref{section:idempotents:generalities}, with atoms $A$ the
$a_i:=X(1,\ldots,\wh{i},\ldots,n)=v_i^\perp$. For $k\leq n$ the $W(\Phi)$-action
on $A$ is $k$-fold transitive, so the $O_k$ each contain a single element.
We choose $O_1=\{a_1\}$ and $O_2$ the pair $\{a_1,a_2\}$. Rather than $a_1$ we will write
$a\in O_1$ for our single idempotent generator. If $i>1$ then
let $w$ be the reflection $s_{v_1-v_i}$ so that $a_i=aw$; let 
$\omega:=s_1\ldots s_{i-1}$ so that 
\begin{equation}
  \label{eq:19}
\aa_i:=(s_{i-1}\ldots s_1)a(s_1\ldots s_{i-1})
\end{equation}
Any $e\in\BB$ can be written
uniquely as $e=a_{i_1}\vee\cdots\vee a_{i_k}$ for $i_1<\cdots<i_k$; write 
$\ve:=\aa_{i_1}\ldots\aa_{i_k}$. 

The result is that the Boolean reflection monoids have generators 
the $s_i$ and a single idempotent $a$, with the \emph{(Idem1)\/} relations 
$a^2=a$, and the \emph{(Idem2)\/} relations $a\aa_2=\aa_2a$,
or $as_{1}as_{1}=s_{1}as_{1}a$. 

We saw at the end of \S\ref{section:idempotents:generalities} that any set
of atoms in $\BB$ is independent, so the \emph{(Idem3)\/} relations are vacuous. Note also
the ``thinning'' effect of the $W(\Phi)$-action: the $n$ generators and
$n+\frac{1}{2}n(n-1)$ relations of \S\ref{section:idempotents:generalities} have been
reduced to just one generator and two relations.

Now to the \emph{(Iso)\/} relations. Dropping the even labeled edges
from the symbols in Table \ref{table:roots1} and choosing an $s\in S$
from each resulting component gives representatives $H_{s_1}$ in types
$A$ and $D$ and $H_{s_0},H_{s_1}$ in type $B$. If $X(J)\in\BB$ is to
be minimal with $H_{s_1}\supseteq X(J)$ then $J$ is minimal with
$1,2\not\in J$, i.e.: $J=\{3,\ldots,n\}$ and $X(J)=a_1\vee a_2$
(compare this with the calculation at the end of Example
\ref{eg:general} in \S\ref{section:renner:classical1}). 
Similarly with $H_{s_0}$ we have $X(J)=a_1$, and so
$$
\begin{tabular}{cc}
\hline
$\Phi$&$\mathit{Iso}$\\\hline \vrule width 0 mm height 5 mm depth 0 mm
$A_{n-1}$
&
$(a_1\vee a_2,s_1)$
\\
$B_n$
&
$(a_1\vee a_2,s_1), (a_1,s_0)$
\\
$D_n$
&
$(a_1\vee a_2,s_1)$\vrule width 0 mm height 0 mm depth 4 mm
\\\hline
\end{tabular}
$$
The \emph{(Iso)} relations are thus $as_1a=as_1as_1$
in all cases, together with $as_0=a$ in type $B$.

This completes the presentation given by Theorem \ref{presentations:result1000}
for the Boolean reflection monoids. But it turns out that the \emph{(RefIdem)\/} relations 
can be significantly reduced in number. For all three $\Phi$ we have \emph{(Units)\/} relations
$(s_is_{i+1})^3=1$ for $1\leq i\leq n-2$, which we use in their ``braid'' form,
$s_{i+1}s_is_{i+1}=s_is_{i+1}s_i$. Then:

\begin{lemma}\label{popova:boolean:result100}
The relations $s_i\aa_j=\aa_j^{s_i}s_i$
for $1\leq i\leq n-1$ and $1\leq j\leq n$ are implied by 
the (Units) relations and the relations
$s_ia=\aa^{s_i}s_i$ for $1\leq i\leq n-1$, ie.: the relations 
$s_ia=as_i$, $(i\not=1)$.
\end{lemma}

\begin{proof}
We have 
$$
a_js_i=
\left\{
\begin{array}{ll}
a_{j-1},&i=j-1,\\
a_{j+1},&i=j,\\
a_j,&i\not= j-1,j,\\
\end{array}
\right.
$$
hence $\aa_j^{s_i}$ is one of the words $\aa_{j-1}$ $(i=j-1)$ or
$\aa_{j+1}$ $(i=j)$ or $\aa_j$ (otherwise) chosen in (\ref{eq:19}). 
There are then four cases to consider:
(i). $1\leq i<j-1$: 
\begin{align*}
s_i\aa_j
&=s_i(s_{j-1}\ldots s_{1})a(s_{1}\ldots s_{j-1})
=(s_{j-1}\ldots s_{i}s_{i+1}s_i\ldots s_{1})a(s_{1}\ldots s_{j-1})\\
&=(s_{j-1}\ldots s_{i+1}s_is_{i+1}\ldots s_{1})a(s_{1}\ldots s_{j-1})
=(s_{j-1}\ldots s_{1})s_{i+1}a(s_{1}\ldots s_{j-1})\\
&=(s_{j-1}\ldots s_{1})a s_{i+1}(s_{1}\ldots s_{j-1})
=(s_{j-1}\ldots s_{1})a(s_{1}\ldots s_{i+1}s_is_{i+1}\ldots s_{j-1})\\
&=(s_{j-1}\ldots s_{1})a(s_{1}\ldots s_is_{i+1}s_i\ldots s_{j-1})
=(s_{j-1}\ldots s_{1})a(s_{1}\ldots s_{j-1})s_i\\
&=\aa_js_i,
\end{align*}
where we have used the braid relations and the commuting of $s_{i+1}$ and 
$a$.
(ii). $j<i\leq n-1$: $s_i$ commutes with $s_1,\ldots,s_{j-1}$ and $a$,
giving the result immediately.
(iii). $i=j-1$: $s_{j-1}\aa_j
=s_{j-1}(s_{j-1}\ldots s_{1})a(s_{1}\ldots s_{j-1})
=(s_{j-2}\ldots s_{1})a(s_{1}\ldots s_{j-1})s_{j-1}s_{j-1}
=\aa_{j-1}s_{j-1}$
(iv). $i=j$: $s_{j}\aa_j
=s_{j}(s_{j-1}\ldots s_{1})a(s_{1}\ldots$ $s_{j-1})
=(s_{j}\ldots s_{1})a(s_{1}\ldots s_{j-1})s_js_j
=\aa_{j+1}s_j$.
\qed
\end{proof}

Putting it all together in the type $A$ case 
we get the following presentation:

\label{popova:boolean:result150}
\paragraph{The Boolean reflection monoid of type $A$:} 
$$
\begin{pspicture}(0,0)(14,2.5)
\rput(0,0){
\rput(4,1){\BoxedEPSF{fig2.eps scaled 400}}
\rput(1.65,.6){$s_1$}\rput(3,.6){$s_2$}
\rput(4.9,.6){$s_{n-2}$}\rput(6.35,.6){$s_{n-1}$}
}
\rput(8,1.5){$\begin{array}{rl}
M(A_{n-1},\BB)=\langle s_1,\ldots,s_{n-1},a\,|\,
&(s_is_j)^{m_{ij}}=1,\,a^2=a,
\vrule width 0 mm height 0 mm depth 3 mm\\
&s_ia=a s_i\,(i\not=1),
\vrule width 0 mm height 0 mm depth 3 mm\\
&as_1a=a s_{1}a s_{1}=s_{1}a s_{1}a\rangle\\
\end{array}$}
\end{pspicture}
$$

Recall that the $m_{ij}$ can be read off the Coxeter symbol,
with the nodes joined
by an edge labelled $m_{ij}$ if $m_{ij}\geq 4$, an unlabelled edge
if $m_{ij}=3$, no edge if $m_{ij}=2$ (and $m_{ij}=1$ when $i=j$).
The relation $s_ia=\aa^{s_i}s_i$ is vacuous when $i=1$.

\begin{figure}
  \centering
\begin{pspicture}(0,0)(14.3,3)
\rput(0,0){
\rput(0,0){\qdisk(2,1){.1}\rput(2,.7){$1$}
\psbezier[showpoints=false,linewidth=.3mm]{->}(1.8,1.2)(1.3,2)(2.7,2)(2.2,1.2)
}
\rput(2.75,1){$\ldots$}
\rput(1.5,0){\qdisk(2,1){.1}
\psbezier[showpoints=false,linewidth=.3mm]{->}(1.8,1.2)(1.3,2)(2.7,2)(2.2,1.2)
}
\qdisk(4.5,1){.1}\rput(4.5,.7){$i$}
\psbezier[showpoints=false,linewidth=.3mm]{<->}(4.5,1.2)(4.5,2)(5.5,2)(5.5,1.2)
\qdisk(5.5,1){.1}\rput(5.5,.7){$i+1$}
\rput(5,2.25){$s_i\,(1\leq i<n-1)$}
\rput(4.5,0){\qdisk(2,1){.1}
\psbezier[showpoints=false,linewidth=.3mm]{->}(1.8,1.2)(1.3,2)(2.7,2)(2.2,1.2)
}
\rput(7.25,1){$\ldots$}
\rput(6,0){\qdisk(2,1){.1}\rput(2,.7){$n$}
\psbezier[showpoints=false,linewidth=.3mm]{->}(1.8,1.2)(1.3,2)(2.7,2)(2.2,1.2)
}}
\rput(8,0){
\rput(0,0){\qdisk(2,1){.1}}\rput(2,.7){$1$}
\rput(1,0){\qdisk(2,1){.1}\rput(2,.7){$2$}
\psbezier[showpoints=false,linewidth=.3mm]{->}(1.8,1.2)(1.3,2)(2.7,2)(2.2,1.2)
}
\rput(3.75,1){$\ldots$}\rput(3.5,2.25){$a$}
\rput(2.5,0){\qdisk(2,1){.1}\rput(2,.7){$n$}
\psbezier[showpoints=false,linewidth=.3mm]{->}(1.8,1.2)(1.3,2)(2.7,2)(2.2,1.2)
}}
\end{pspicture}
  \caption{Type $A$ Boolean reflection monoid generators as partial permutations.}
  \label{fig1}
\end{figure}
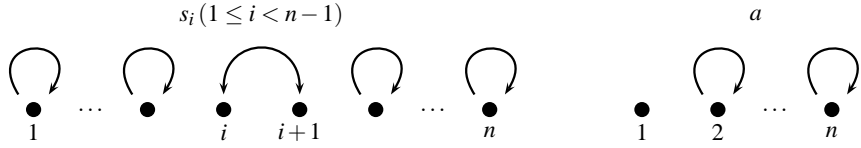

\begin{remark*}
We saw in \cite{Everitt-Fountain10}*{\S3.1} that $M(A_{n-1},\BB)$ is isomorphic
to the symmetric inverse monoid $\II_n$ with the generators $s_i$ and $a$  
corresponding to the partial permutations in Figure \ref{fig1}--we thus recover
Popova's presentation \cite{Popova61} for the symmetric inverse monoid.   
\end{remark*}

Now to the type $B$ Boolean reflection monoids, where we have one piece of unfinished business,
namely that 
Lemma \ref{popova:boolean:result100} leaves unresolved the status of the 
\emph{(RefIdem)\/} relations when $s=s_0$:

\begin{lemma}\label{popova:boolean:result200}
If $\Phi=B_n$ then the relations $s_0\aa_j=\aa_j^{s_0}s_0$
for $1\leq j\leq n$ are implied by the (Units) relations and
the relation $s_0\aa_2=\aa_2s_0$, i.e.: $s_0s_1as_1=s_1as_1s_0$.
\end{lemma}

\begin{proof}
For $3\leq j\leq n$ we have 
$s_0\aa_j
=s_0(s_{j-1}\ldots s_{1})a(s_{1}\ldots s_{j-1})$
$=(s_{j-1}\ldots s_2s_0)\aa_{2}(s_{2}$ $\ldots s_{j-1})$
$=(s_{j-1}\ldots s_{2})\aa_{2}s_0(s_{2}\ldots s_{j-1})$
$=\aa_{j}s_0$.
\qed
\end{proof}


\label{popova:boolean:result250}
\paragraph{The Boolean reflection monoid of type $B$:} 
$$
\begin{pspicture}(0,0)(14,2.75)
\rput(-.4,0.25){
\rput(4,1){\BoxedEPSF{fig2.eps scaled 400}}
\rput(1.65,.6){$s_0$}\rput(3,.6){$s_1$}
\rput(4.9,.6){$s_{n-2}$}\rput(6.35,.6){$s_{n-1}$}
\rput(2.4,1.25){$4$}
}
\rput(8.4,1.5){$\begin{array}{rl}
M(B_n,\BB)=\langle s_0,\ldots,s_{n-1},a\,|\,
&(s_is_j)^{m_{ij}}=1,\,a^2=a,
\vrule width 0 mm height 0 mm depth 3 mm\\
&s_ia=a s_i\,(i\not=1),\,as_0=a,
\vrule width 0 mm height 0 mm depth 3 mm\\
&s_0s_{1}a s_{1}=s_{1}a s_{1}s_0,
\vrule width 0 mm height 0 mm depth 3 mm\\
&as_1as_1=s_1as_1a=as_1a\rangle.\\
\end{array}$}
\end{pspicture}
$$

\begin{remark*}
We saw in \cite{Everitt-Fountain10}*{\S5} that just as the Weyl group $W(B_n)$ is isomorphic 
to the group $\SS_{\pm n}$ of signed permutations of
$X=\{1,2,\ldots,n\}$
(see also \S\ref{section:renner:classical1}), so the Boolean reflection
monoid $M(B_n,\BB)$ is isomorphic to the monoid of \emph{partial\/} signed permutations
$\II_{\pm n}:=\{\pi\in\II_{X\cup -X}\,|\, (-x)\pi=-(x\pi)\text{ and }x\in\dom\,\pi
\Leftrightarrow -x\in\dom\,\pi\}$.
\end{remark*}

And so finally to the type $D$ Boolean reflection monoids, where one can prove in a manner analogous to
Lemma \ref{popova:boolean:result200}  that
the relations $s_0\aa_{j}=\aa_{j}^{s_0}s_0$
for $1\leq j\leq n$ are implied by $s_0a=\aa_{2}s_0$, 
$s_0\aa_{3}=\aa_{3}s_0$
and the relations for $W$.

\paragraph{The Boolean reflection monoid of type $D$:} 
$$
\begin{pspicture}(0,0)(14,2.75)
\rput(-1.4,0){
\rput(4,1){\BoxedEPSF{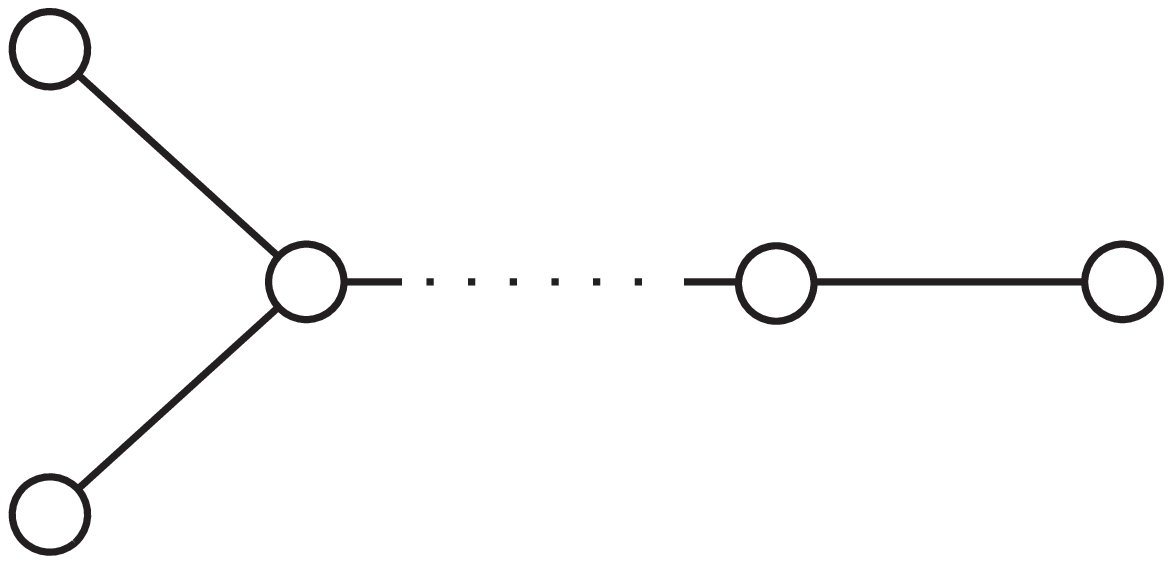 scaled 400}}
\rput(1.75,.45){$s_0$}\rput(1.75,2.35){$s_{1}$}\rput(2.95,.6){$s_{2}$}
\rput(4.85,.6){$s_{n-2}$}\rput(6.25,.6){$s_{n-1}$}
}
\rput(8,1.5){$\begin{array}{rl}
M(D_n,\BB)=\langle s_0,\ldots,s_{n-1},a\,|\,
&(s_is_j)^{m_{ij}}=1,\,a^2=a,
\vrule width 0 mm height 0 mm depth 3 mm\\
&s_ia=a s_i\,(i>1),\,s_0a=s_1as_1s_0,
\vrule width 0 mm height 0 mm depth 3 mm\\
&as_1a=as_1as_1=s_1as_1a,
\vrule width 0 mm height 0 mm depth 3 mm\\
&s_0s_2s_1as_1s_2=s_2s_1as_1s_2s_0\rangle.\\
\end{array}$}
\end{pspicture}
$$

Unfortunately $M(D_n,\BB)$ doesn't seem to have a nice interpretation in terms of 
partial permutations to go with the isomorphism between the Weyl group
$W(D_n)$ and the group of even signed permutations of
$\{1,\ldots,n\}$--see also \S\ref{section:renner:classical1}.

\section{Coxeter arrangement monoids}
\label{section:popova:arrangement}

We now repeat \S\ref{section:popova:boolean} for the Coxeter arrangement monoids
of \cite{Everitt-Fountain10}*{\S6}. 
Let $W=W(\Phi)\subset GL(V)$
be a reflection group with reflecting hyperplanes $\AA=\{v^\perp\,|\,v\in\Phi\}$
and $\cH=\cH(\Phi)$ the intersection lattice of
\S\S\ref{section:idempotents:geometric}-\ref{section:idempotents:coxeter}: 
this is a system for $W$ in $V$--the 
\emph{Coxeter arrangement\/} system. 
We write $M(\Phi,\HH)$ for the
resulting \emph{Coxeter arrangement monoid\/} of type $\Phi$. We use the notation for the Coxeter
generators from \S\ref{section:popova:boolean}. 

The atoms $A$ for the system and the hyperplanes $\AA$ coincide now.
Drop even labeled edges from the symbols in Table \ref{table:roots1}
to get $O_1=\{a\}$ in types $A$ and $D$, or $\{a_1,a_2\}$ in type $B$, where
$$
a=a_1:=(v_2-v_1)^\perp
\text{ and }
a_2:=v_1^\perp;
$$
giving generators the $s_i$ of \S\ref{section:popova:boolean} and $a$ for types $A$ and $D$, or the
$s_i$ and $a_1,a_2$ for type $B$.

The \emph{(Iso)\/} relations are particularly simple when the system and 
the intersection lattice are the same:
the $(e,s)\in\mathit{Iso}$ consist of the representative $a=H_{s}$ and
the $s$ above. Thus, the relations 
are $as_1=a$ for types $A$ and $D$, or $a_1s_1=a_1$
and $a_2s_0=a_2$ in type $B$.

We deal with the remaining relations on a case by case basis.

\subsection{The Coxeter arrangement monoids of type $A$}
\label{section:popova:arrangement:A}

We have
$\AA=\{a_{ij}:=(v_i-v_j)^\perp\,|\,1\leq i<j\leq n\}$, and write
\begin{equation}
  \label{eq:1}
\aa_{ij}:=\left\{\begin{array}{ll}
(s_{i-1}\ldots s_{1})(s_{j-1}\ldots s_{2})a(s_{2}\ldots s_{j-1})(s_{1}\ldots s_{i-1}),
&\text{for }2\leq i<j\leq n,\\
(s_{j-1}\ldots s_{2})a(s_{2}\ldots s_{j-1})
&\text{for }i=1,2\leq j\leq n,
\end{array}\right.
\end{equation}
with $\aa_{12}:=a_1$.

The isomorphisms
$\cH(A_{n-1})\rightarrow\Pi(n)$ of \S\ref{section:idempotents:coxeter}
and the well known isomorphism 
$W(A_{n-1})\rightarrow\SS_n$, the second written as
$g(\pi)\mapsto\pi$, gives the
$W(A_{n-1})$-action on $\HH(A_{n-1})$ as $X(\Lambda)g(\pi)=X(\Lambda\pi)$,
where $\Lambda\pi=\{\Lambda_1\pi,\ldots,\Lambda_p\pi\}$.
Thus, as
$\SS_n$ acts $4$-fold transitively on $\{1,\ldots,n\}$, 
we take $O_2$ to be  $\{a_{12},a_{34}\}$ and $\{a_{12},a_{23}\}$ when
$n\geq 4$, giving \emph{(Idem2)\/}
relations $a\aa_{34}=\aa_{34}a$ and $a\aa_{23}=\aa_{23}a$. When $n=2$
there is only one idempotent (hence no \emph{(Idem2)} relations) and
when $n=3$ we take $O_2$ to be $\{a_{12},a_{23}\}$.

We have the presentation for $\HH(A_{n-1})$ of \S\ref{section:idempotents:geometric}.
Lemma \ref{presentations:result300} and the triple transitivity of $\SS_n$ on
$\{1,\ldots,n\}$ reduce the (A1) relations to:
$$
\begin{pspicture}(0,0)(14.3,1.25)
\rput(0,0){
\rput(4.5,0.5){\BoxedEPSF{ref.monoids2a.fig19.eps scaled 900}}
\rput(4,1.2){${\scriptstyle 1}$}\rput(5,1.2){${\scriptstyle 3}$}\rput(4.5,-.2){${\scriptstyle 2}$}
\rput(5.75,0.5){$=$}
\rput(7,0.5){\BoxedEPSF{ref.monoids2a.fig20.eps scaled 900}}
\rput(6.5,1.2){${\scriptstyle 1}$}\rput(7.5,1.2){${\scriptstyle 3}$}\rput(7,-.2){${\scriptstyle 2}$}
\rput(8.25,0.5){$=$}
\rput(9.5,0.5){\BoxedEPSF{ref.monoids2a.fig21.eps scaled 900}}
\rput(9,1.2){${\scriptstyle 1}$}\rput(10,1.2){${\scriptstyle 3}$}\rput(9.5,-.2){${\scriptstyle 2}$}
}
\end{pspicture}
$$
that is, $a\aa_{13}=a\aa_{23}=\aa_{13}\aa_{23}$.

As with the Boolean monoids, the \emph{(RefIdem)\/} relations
can be reduced in number. 

\begin{lemma}
\label{popova:reflectionarrangements:result200}
The relations 
$s_i\aa_{{jk}}=\aa_{{jk}}^{s_i}s_i$ for 
$1\leq j<k\leq n$ and $1\leq i\leq n-1$
are implied by the the (Units) relations and the relations $s_ia=\aa_{{12}}^{s_i}s_i$
for $1\leq i\leq n-1$, i.e.: the relations $s_ia=as_i$, $(i\not= 2)$.
\end{lemma}

The proof, which is a similar but more elaborate version of that for Lemma
\ref{popova:boolean:result100},
is left to the reader.
Putting it all together we get

\paragraph{The Coxeter arrangement monoid of type $A_{n-1}$:} 
\label{popova:reflectionarrangements:result250}
$$
\begin{pspicture}(0,0)(14,3)
\rput(0,-.3){
\rput(-.6,.75){
\rput(4,1){\BoxedEPSF{fig2.eps scaled 400}}
\rput(1.65,.6){$s_1$}\rput(3,.6){$s_2$}
\rput(4.9,.6){$s_{n-2}$}\rput(6.35,.6){$s_{n-1}$}
}
\rput(8,1.9){$\begin{array}{rl}
M(A_{n-1},\HH)=\langle s_1,\ldots,s_{n-1},a\,|\,
&(s_is_j)^{m_{ij}}=1,\,a^2=a,\, as_1=a,
\vrule width 0 mm height 0 mm depth 3 mm\\
&s_ia=as_i\,(i\not=2), 
\vrule width 0 mm height 0 mm depth 3 mm\\
&a\aa_{23}=\aa_{23}a,a\aa_{34}=\aa_{34}a, 
\vrule width 0 mm height 0 mm depth 3 mm\\
&a\aa_{13}=a\aa_{23}=\aa_{13}\aa_{23}\rangle.\\
\end{array}$}
}
\end{pspicture}
$$
for $n\geq 4$ and where $\aa_{ij}$ is given by (\ref{eq:1}). We leave
the reader to make the necessary adjustments in the $n=2,3$ cases. See also \cite{Fitzgerald03}.

\begin{remark*}
We saw in \cite{Everitt-Fountain10}*{\S2.2} that the type $A$ Coxeter arrangement monoid 
is isomorphic to the monoid of uniform block permutations (see for example
\cite{Fitzgerald03}):
the elements of this monoid have the form 
$\lf\pi\rf_\Lambda^{}$ with $\pi\in\SS_n$ and
$\Lambda=\{\Lambda_1,\ldots,\Lambda_p\}$ a partition of $\{1,\ldots,n\}$.
We have $\lf\pi\rf_\Lambda^{}=\lf\tau\rf_\Delta^{}$
if and only if $\Lambda=\Delta$ and $\Lambda_i\pi=\Delta_i\tau$ for all $i$,
and product $\lf\pi\rf_\Lambda^{}\lf\tau\rf_\Delta^{}=\lf\pi\tau\rf_\Gamma^{}$,
where $\Gamma=\Lambda\vee\Delta\pi^{-1}$ and 
$\vee$ is the join in the partition lattice $\Pi(n)$.
The units are thus the $\lf\pi\rf_\Lambda^{}$ where all the blocks $\Lambda_i$ have 
size one and the idempotents are the $\lf\ve\rf_\Lambda^{}$ where $\ve$ is the 
identity permutation. The generators in the presentation above
correspond to the uniform block permutations in Figure \ref{fig3}, where the partitions are given 
in red.
\end{remark*}

\begin{figure}
  \centering
\begin{pspicture}(0,0)(14.3,3)
\rput(-.5,0){
\rput(0,0){
\rput(0,0){\qdisk(2,1){.1}
\psframe[linewidth=.3mm,linecolor=red](1.6,.5)(2.4,1.5)\rput(2,.75){$1$}
\psbezier[showpoints=false,linewidth=.3mm]{->}(1.8,1.2)(1.3,2)(2.7,2)(2.2,1.2)
}
\rput(2.75,1){$\ldots$}
\rput(1.5,0){\qdisk(2,1){.1}
\psframe[linewidth=.3mm,linecolor=red](1.6,.5)(2.4,1.5)
\psbezier[showpoints=false,linewidth=.3mm]{->}(1.8,1.2)(1.3,2)(2.7,2)(2.2,1.2)
}
\qdisk(4.5,1){.1}\rput(4.5,.75){$i$}
\psframe[linewidth=.3mm,linecolor=red](4.1,.5)(4.9,1.5)
\psbezier[showpoints=false,linewidth=.3mm]{<->}(4.5,1.2)(4.5,2)(5.5,2)(5.5,1.2)
\qdisk(5.5,1){.1}\rput(5.5,.75){$i+1$}
\psframe[linewidth=.3mm,linecolor=red](5.1,.5)(5.9,1.5)
\rput(5,2.25){$s_i\,(1\leq i<n-1)$}
\rput(4.5,0){\qdisk(2,1){.1}
\psframe[linewidth=.3mm,linecolor=red](1.6,.5)(2.4,1.5)
\psbezier[showpoints=false,linewidth=.3mm]{->}(1.8,1.2)(1.3,2)(2.7,2)(2.2,1.2)
}
\rput(7.25,1){$\ldots$}
\rput(6,0){\qdisk(2,1){.1}\rput(2,.75){$n$}
\psframe[linewidth=.3mm,linecolor=red](1.6,.5)(2.4,1.5)
\psbezier[showpoints=false,linewidth=.3mm]{->}(1.8,1.2)(1.3,2)(2.7,2)(2.2,1.2)
}}
\rput(8,0){
\rput(0,0){\qdisk(2,1){.1}\rput(2,.75){$1$}
\psframe[linewidth=.3mm,linecolor=red](1.6,.5)(3.4,1.5)
\psbezier[showpoints=false,linewidth=.3mm]{->}(1.8,1.2)(1.3,2)(2.7,2)(2.2,1.2)
}
\rput(1,0){\qdisk(2,1){.1}\rput(2,.75){$2$}
\psbezier[showpoints=false,linewidth=.3mm]{->}(1.8,1.2)(1.3,2)(2.7,2)(2.2,1.2)
}
\rput(2,0){\qdisk(2,1){.1}\rput(2,.75){$3$}
\psframe[linewidth=.3mm,linecolor=red](1.6,.5)(2.4,1.5)
\psbezier[showpoints=false,linewidth=.3mm]{->}(1.8,1.2)(1.3,2)(2.7,2)(2.2,1.2)
}
\rput(4.75,1){$\ldots$}\rput(3.5,2.25){$a$}
\rput(3.5,0){\qdisk(2,1){.1}\rput(2,.75){$n$}
\psframe[linewidth=.3mm,linecolor=red](1.6,.5)(2.4,1.5)
\psbezier[showpoints=false,linewidth=.3mm]{->}(1.8,1.2)(1.3,2)(2.7,2)(2.2,1.2)
}}}
\end{pspicture}
  \caption{Type $A$ Coxeter arrangement monoid generators as uniform block permutations.}
  \label{fig3}
\end{figure}
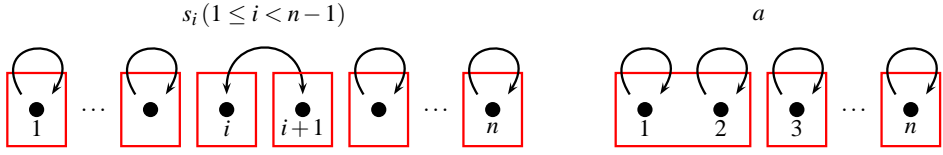

\subsection{The Coxeter arrangement monoids of type $B$}
\label{section:popova:arrangement:B}

As in \S\ref{section:popova:arrangement:A} we work with $n\geq 4$ and
leave the simpler cases $n=2,3$ to the reader.
We have $\cA$ the $a_{ij}$ from \S\ref{section:popova:arrangement:A} together with
$$
\{d_{ij}:=(v_i+v_j)^\perp\,|\,1\leq i<j\leq n\}
\text{ and }\{e_i:=v_i^\perp\,|\,1\leq i\leq n\}.
$$
Let $\aa_{ij}$ be the expression defined in (\ref{eq:1}), except with
$a_1$ instead of $a$, and
\begin{equation}
  \label{eq:2}
\delta_{ij}=
\left\{\begin{array}{ll}
(s_{i-1}\ldots s_{1})(s_{j-1}\ldots s_{2})s_0a_1 s_0(s_{2}\ldots s_{j-1})
(s_{1}\ldots s_{i-1}),
&2\leq i<j\leq n,\\
(s_{j-1}\ldots s_{2})s_0a_1 s_0(s_{2}\ldots s_{j-1}),
&i=1,2<j\leq n,\\
\end{array}\right.   
\end{equation}
with $\delta_{12}:=s_0a_1s_0$; and 
\begin{equation}
  \label{eq:3}
\ve_i:=(s_{i-1}\ldots s_{1})a_2(s_{1}\ldots s_{i-1}).  
\end{equation}
for $i>1$ with $\ve_1:=a_2$.
One can build a combinatorial model for the action of $W(B_n)$ on
$\cH(B_n)$ much as in \S\ref{section:popova:arrangement:A}:
the isomorphism $\cH(B_n)\rightarrow\TT$ of \S\ref{section:idempotents:geometric}
and the well known isomorphism 
$W(B_n)\rightarrow\SS_n\ltimes\mathbf{2}^n$ 
(see \cite{Everitt-Fountain10}*{\S6.2} for notation)
give the $W(B_n)$-action on $\cH(B_n)$ as $X(\Delta,\Lambda)g(\pi,T)$ $=X(\Delta\pi,
\Lambda^T\pi)$.
One deduces from this that $O_2=\{\{a_1,a_{34}\}, \{a_1,a_{23}\},
\{a_1,d_{12}\},\{a_2,a_{23}\}, \{a_1,a_2\}$
$\{a_2,e_2\}\}$,
and hence that \emph{(Idem2)\/} relations are 
$$
a_1\aa_{34}=\aa_{34}a_1,
a_1\aa_{23}=\aa_{23}a_1,
a_1\delta_{12}=\delta_{12}a_1,
a_2\aa_{23}=\aa_{23}a_2,
a_1a_2=a_2a_1
\text{ and }
a_2\ve_2=\ve_2a_2.
$$
We have the presentation for $\cH(B_n)$ of \S\ref{section:idempotents:geometric}, hence the 
relations $a_1\aa_{13}=a_1\aa_{23}=\aa_{13}\aa_{23}$ of
\S\ref{section:popova:arrangement:A}; the family (B2) reduces to:
$$
\begin{pspicture}(0,0)(14.3,1.25)
\rput(0,0){
\rput(4.75,0.5){\BoxedEPSF{ref.monoids2a.fig22.eps scaled 900}}
\rput(4.25,1.2){${\scriptstyle 1}$}\rput(5.25,1.2){${\scriptstyle 3}$}\rput(4.75,-.2){${\scriptstyle 2}$}
\rput(5.75,0.5){$=$}
\rput(7,0.5){\BoxedEPSF{ref.monoids2a.fig23.eps scaled 900}}
\rput(6.5,1.2){${\scriptstyle 1}$}\rput(7.5,1.2){${\scriptstyle 3}$}\rput(7,-.2){${\scriptstyle 2}$}
\rput(8.25,0.5){$=$}
\rput(9.25,0.5){\BoxedEPSF{ref.monoids2a.fig24.eps scaled 900}}
\rput(8.75,1.2){${\scriptstyle 1}$}\rput(9.75,1.2){${\scriptstyle 3}$}\rput(9.25,-.2){${\scriptstyle 2}$}
}
\end{pspicture}
$$
that is, 
$a_1\delta_{13}=\delta_{13}\delta_{23}=a_1\delta_{23}$;
families (B3) and (B4) become,
$$
\begin{pspicture}(0,0)(14.3,1.25)
\rput(-3,0){
\rput(4.75,0.5){\BoxedEPSF{ref.monoids2a.fig25.eps scaled 900}}
\rput(4.2,0){${\scriptstyle 1}$}\rput(5.2,0){${\scriptstyle 2}$}
\rput(5.75,0.5){$=$}
\rput(7,0.5){\BoxedEPSF{ref.monoids2a.fig26.eps scaled 900}}
\rput(6.5,0){${\scriptstyle 1}$}\rput(7.5,0){${\scriptstyle 2}$}
\rput(8.25,0.5){$=$}
\rput(9.25,0.5){\BoxedEPSF{ref.monoids2a.fig27.eps scaled 900}}
\rput(8.7,0){${\scriptstyle 1}$}\rput(9.7,0){${\scriptstyle 2}$}
}
\rput(5,0.25){
\rput(4.5,0.3){\BoxedEPSF{ref.monoids2a.fig28.eps scaled 900}}
\rput(4,1){${\scriptstyle 1}$}\rput(5,1){${\scriptstyle 3}$}\rput(4.5,-.4){${\scriptstyle 2}$}
\rput(5.75,0.5){$=$}
\rput(7,0.5){\BoxedEPSF{ref.monoids2a.fig29.eps scaled 900}}
\rput(6.5,.45){${\scriptstyle 1}$}\rput(7.5,.45){${\scriptstyle 3}$}\rput(7,-.4){${\scriptstyle 2}$}
}
\end{pspicture}
$$
or $a_1a_2=\ve_2a_2=\delta_{12}a_2$
and
$a_1\aa_{23}\delta_{13}=a_2\ve_2\ve_3$.
Finally, there is the thinning of the (RefIdem) relations:

\begin{lemma}
\label{popova:reflectionarrangements:result300}
The (RefIdem) relations can be deduced from the (Units) relations 
and the relations
$s_ia_1=a_1s_i\, (i\not=0,2)$, 
$s_ia_2=a_2s_i\,(i\not=1)$, 
$s_0\aa_{2j}=\aa_{2j}s_0\,(j>2)$,
$s_0\delta_{2j}=\delta_{2j}s_0\,(j>2)$,
$s_1\delta_{12}=\delta_{12}s_1$
and $s_0\ve_2=\ve_2s_0$.
\end{lemma}


We thus have our presentation:

\paragraph{The Coxeter arrangement monoid of type $B_n$:} 
$$
\begin{pspicture}(0,0)(14,4.5)
\rput(-0.1,1.8){
\rput(-1.8,-.3){
\rput(4,1){\BoxedEPSF{fig2.eps scaled 350}}
\rput(1.9,.6){$s_0$}\rput(3.2,.6){$s_1$}
\rput(4.9,.6){$s_{n-2}$}\rput(6.1,.6){$s_{n-1}$}
\rput(2.5,1.25){$4$}
}
\rput(7.3,.4){$\begin{array}{rl}
M(B_n,\HH)=\langle s_0,\ldots,s_{n-1},a_1,a_2\,|\,
&(s_is_j)^{m_{ij}}=1,\,a_j^2=a_j,a_1s_1=a_1,a_2s_0=a_2,
\vrule width 0 mm height 0 mm depth 3 mm\\
&s_ia_1=a_1s_i\,(i\not=0,2),s_ia_2=a_2s_i\,(i\not=1),
\vrule width 0 mm height 0 mm depth 3 mm\\
&s_0\aa_{2j}=\aa_{2j}s_0,\,(j>2),
s_0\delta_{2j}=\delta_{2j}s_0,\,(j>2),
\vrule width 0 mm height 0 mm depth 3 mm\\
&s_1\delta_{12}=\delta_{12}s_1,s_0\ve_2=\ve_2s_0, a_j\aa_{23}=\aa_{23}a_j,
\vrule width 0 mm height 0 mm depth 3 mm\\
&a_1a_2=a_2a_1=a_2\ve_2=\ve_2a_2=\delta_{12}a_2, a_1\delta_{12}=\delta_{12}a_1,
\vrule width 0 mm height 0 mm depth 3 mm\\
&a_1\aa_{13}=a_1\aa_{23}=\aa_{13}\aa_{23},a_1\aa_{34}=\aa_{34}a_1,
\vrule width 0 mm height 0 mm depth 3 mm\\
&a_1\delta_{13}=\delta_{13}\delta_{23}=a_1\delta_{23},
a_1\aa_{23}\delta_{13}=a_2\ve_2\ve_3\rangle.
\vrule width 0 mm height 0 mm depth 3 mm\\
\end{array}$}}
\end{pspicture}
$$
where $\aa_{ij}, \delta_{ij}$ and $\ve_i$ are given by (\ref{eq:1})-(\ref{eq:3}).

\begin{remark*}
Just as the type $A$ Coxeter arrangement monoid is isomorphic to the 
monoid of uniform block permutations, so the type $B$ reflection monoid is
isomorphic to the monoid of ``uniform block signed permutations''.
See \cite{Everitt-Fountain10}*{\S6.2} for details.
\end{remark*}

\subsection{The Coxeter arrangement monoids of type $D$}
\label{section:popova:arrangement:D}

The $\cA$ are the $a_{ij}$ and $d_{ij}$ of \S\ref{section:popova:arrangement:B};
let $\aa_{ij}$ be defined as in (\ref{eq:1}) and 
\begin{equation}
  \label{eq:4}
\delta_{ij}=
\left\{\begin{array}{ll}
(s_{i-1}\ldots s_{1})(s_{j-1}\ldots s_{2})g^{-1}a 
g(s_{2}\ldots s_{j-1})(s_{1}\ldots s_{i-1}),
&2\leq i<j\leq n,\\
(s_{j-1}\ldots s_{2})g^{-1}ag(s_{2}\ldots s_{j-2}),
&i=1,2<j\leq n,\\
\end{array}\right.   
\end{equation}
with $g=s_{2}s_{1}s_0s_2$ and $\delta_{12}:=g^{-1}ag$.

There is a combinatorial model for the action of $W(D_n)$ on
$\HH(D_n)$, much as with types $A$ and $B$.
We refer the reader to
\cite{Everitt-Fountain10}*{\S6.2} or \cite{Orlik92}*{\S6.4} for details,
noting that 
$O_2=\{\{a,a_{34}\}, \{a,$ $a_{23}\}, \{a,d_{12}\}\}$ 
when $n>4$,
while for $n=4$ we have
$\{a,d_{34}\}$ as well.
The \emph{(Idem2)\/}
relations are thus
$$
a\aa_{34}=\aa_{34}a,
a\aa_{23}=\aa_{23}a,
a\delta_{12}=\delta_{12}a,
$$
together with $a\delta_{34}=\delta_{34}$ when $n=4$.

The presentation for $\HH(D_n)$ of \S\ref{section:idempotents:geometric} together with 
Lemma \ref{presentations:result300} give the relations
$a\aa_{13}=a\aa_{23}=\aa_{13}\aa_{23}$
and 
$a\delta_{13}=\delta_{13}\delta_{23}=a\delta_{23}$
of \S\ref{section:popova:arrangement:B}, as well as 
$$
a\aa_{23}\delta_{23}=a\delta_{12}\delta_{23}
=a\aa_{23}\delta_{12}\delta_{23},
$$
$a\aa_{23}\delta_{13}=a\aa_{13}\aa_{23}\delta_{13}$
and 
$$
a\aa_{23}\aa_{34}\delta_{12}\delta_{23}\delta_{34}
=a\aa_{34}\delta_{12}\delta_{34}.
$$

Finally,
the \emph{(RefIdem)\/} relations can be deduced from the relations
$s_ia=as_i\, (i\not=2)$,
$s_0\aa_{3k}=\aa_{3k}s_0\,(k>3)$,
$s_0\delta_{3k}=\delta_{3k}s_0\,(k>3)$ and $s_3\delta_{12}=\delta_{12}s_3$.
All of which leads us to:

\paragraph{The arrangement monoid of type $D_n\,(n>4)$:}
$$
\begin{pspicture}(0,0)(14,4)
\rput(-0.4,1.9){
\rput(-1.5,-1.2){
\rput(4,1){\BoxedEPSF{fig3a.eps scaled 350}}
\rput(2.1,-0.25){$s_0$}\rput(2.1,2.25){$s_{1}$}\rput(3,.6){$s_{2}$}
\rput(4.85,.6){$s_{n-2}$}\rput(6.25,.6){$s_{n-1}$}
}
\rput(7.7,0.2){$\begin{array}{rl}
M(D_n,\HH)=\langle s_0,\ldots,s_{n-1},a\,|\,
&(s_is_j)^{m_{ij}}=1,\,a^2=a,as_1=a,\,
s_ia=as_i\,(i\not=2),
\vrule width 0 mm height 0 mm depth 3 mm\\
&s_0\aa_{3j}=\aa_{3j}s_0, s_0\delta_{3j}=\delta_{3j}s_0,\,(\text{both }j>3),
\vrule width 0 mm height 0 mm depth 3 mm\\
&s_3\delta_{12}=\delta_{12}s_3,\,a\aa_{34}=\aa_{34}a,
a\delta_{12}=\delta_{12}a,
\vrule width 0 mm height 0 mm depth 3 mm\\
&a\aa_{13}=a\aa_{23}=\aa_{23}a=\aa_{13}\aa_{23},
a\delta_{13}=\delta_{13}\delta_{23}=a\delta_{23},
\vrule width 0 mm height 0 mm depth 3 mm\\
&a\aa_{23}\delta_{23}=a\delta_{12}\delta_{23}
=a\aa_{23}\delta_{12}\delta_{23}, 
a\aa_{23}\delta_{13}=a\aa_{13}\aa_{23}\delta_{13},
\vrule width 0 mm height 0 mm depth 3 mm\\
&a\aa_{23}\aa_{34}\delta_{12}\delta_{23}\delta_{34}
=a\aa_{34}\delta_{12}\delta_{34}\rangle.
\vrule width 0 mm height 0 mm depth 3 mm\\
\end{array}$}}
\end{pspicture}
$$
together with $a\delta_{34}=\delta_{34}a$ when $n=4$, and
where $\aa_{ij}$ and $\delta_{ij}$ are given by (\ref{eq:1}) and (\ref{eq:4}).

\section{Renner monoids}
\label{section:renner}

\subsection{Generalities}
\label{section:renner:general}

The
principal objects of study in this section are algebraic monoids: affine algebraic
varieties that carry the structure of a monoid. The theory builds on
that of linear algebraic groups, and there are many parallels
between the two.  Standard references for both the groups and the
monoids are \cite{Borel91,Humphreys90,Humphreys75,Putcha88,Renner05}. 
The beginner should start with the survey \cite{Solomon95}.

Much of the structure of an algebraic group is encoded by the Weyl
group $W$. The analogous role is played for algebraic monoids by the
Renner monoid $R$. It turns out that the Renner monoid can be realized
as a monoid $M(W,\cS)$ of partial permutations. Moreover, the system
$\cS$ is isomorphic (as a $\vee$-semilattice with $\0$) to the face
lattice $\FF(P)$ of a convex polytope $P$. We use these facts to
obtain presentations for Renner monoids. Very different presentations
have been found by Godelle \cite{Godelle10} (see also \cite{Godelle10a}) using a completely
different approach.

We start by establishing notation from algebraic groups and monoids.
Let $k=\ov{k}$ be an algebraically closed field and $M$ an irreducible algebraic monoid over 
$k$. We will assume throughout that $M$ has a $0$. Let $G$ be the group of units, and assume
that $G$ is a reductive algebraic group. In particular, $M$ is reductive. 

All the examples in this paper will arise via the following construction.
Let $G_0$ be a connected semisimple algebraic group and 
$\rho:G_0\rightarrow GL(V)$ a rational representation with finite kernel. 
Let $M=M(G_0,\rho):=\ov{k^\times \rho(G_0)}$, where $k^\times=k\setminus\{0\}$.
Then $M$ is a reductive irreducible algebraic monoid with $0$ and units $G:=k^\times\rho(G_0)$--see 
\cite{Solomon95}*{\S2}. If $G_0\subset GL_n$ is a classical algebraic group and 
$\rho:G_0\hookrightarrow GL_n$ is the natural representation
then we call the resulting $M$ a \emph{classical algebraic
monoid\/}. Thus, if $G_0=\sl_n, \so_n$ 
and $\sp_n$, we have the general linear monoid
$\m_n=\ov{k^\times\sl_n}$ (all $n\times n$ matrices over $k$), the orthogonal monoids
$\mso_n=\ov{k^\times\so_n}$ and the symplectic monoids 
$\msp_n=\ov{k^\times\sp_n}$. 

Returning to generalities, let $T\subset G$ be a maximal torus and $\ov{T}\subset M$
its (Zariski) closure. Let $\XXX(T)=\text{Hom}(T,k^\times)$ be the character group
and $\XXX:=\XXX(T)\otimes\R$. 
Then $\XXX(T)$ is a free $\Z$-module
with rank equal to $\dim T$.
In the construction above, if $T_0\subset G_0$ 
is a maximal torus then $T=k^\times\rho(T_0)\subset G$ is a maximal torus with
$\dim T=\dim T_0+1$. If $v\in\XXX(T_0)$ then the map 
$t\rho(t')\mapsto v(t'),\,(t\in k^\times,t'\in T)$ is a character in $\XXX(T)$,
and so we can identify 
$\XXX(T_0)$ with a submodule of $\XXX(T)$ with
$\rk_\Z\XXX(T)=\rk_\Z\XXX(T_0)+1$. If
$\XXX_0=\XXX(T_0)\otimes\R\subset\XXX$ then $\dim\XXX=\dim\XXX_0+1$. 

Let $\Phi=\Phi(G,T)\subset\XXX(T)$ be the root system 
determined by $T$. If $\Phi(G_0,T_0)$ is the system for $(G_0,T_0)$ above, 
then by \cite{Solomon95}*{\S 2} or \cite{Springer98}*{Chapter 7}
the character $t\rho(t')\mapsto v(t'),\,(t\in k^\times,t'\in T)$ is a root in
$\Phi(G,T)$. Thus we can identify $\Phi(G_0,T_0)$ with a subset of $\Phi(G,T)$ where
$|\Phi(G,T)|=\dim G-\dim T=(\dim G_0+1)-(\dim T_0+1)=|\Phi(G_0,T_0)|$. 
In particular, we can identify the root systems of $G_0$ and $G$. The
roots systems for the examples considered in this section are given in Table 
\ref{table:classical}.

If $v\in\Phi$ let $s_v$ be the reflection of $\XXX$ in $v$ and
$W(\Phi)=\langle s_v\,|\,v\in\Phi\rangle$ the resulting reflection
group. Let $\Delta\subset\Phi$ be a simple
system (determined by the choice of a Borel subgroup $T\subset B$) 
so that $W(\Phi)$ is a Coxeter system $(W,S)$ with $S$ the
reflections $s_v$ in the simple roots $v\in\Delta$. Let
$W(G,T)=N(T)/T$ be the Weyl group. If $w\in W$ and $\ov{w}\in G$
with $w=\ov{w} T$ then we will abuse notation throughout and write
$w^{-1}tw$ rather than $\ov{w}^{-1}t\ov{w}$. In particular, $W$
acts faithfully on $\XXX$ via 
$v^w(t)=v(w^{-1}tw)$, realizing an injection
$W\hookrightarrow GL(\XXX)$ and an isomorphism $W(G,T)\cong
W(\Phi)$. We will identify these two groups in what follows and just
write $W$ for both. If $G=k^\times\rho(G_0)$ we
identify the Weyl groups $W(G_0,T_0)$ and $W(G,T)$ via the
identifications of their root systems. 

We will also have need for the duals of these notions:
let $\XXX^\vee(T)=\text{Hom}(k^\times,T)$ be the cocharacter group of 
$T$ (i.e.: $1$-parameter subgroups of $T$). The Weyl group acts on
$\XXX^\vee(T)$ via $\lambda\mapsto\lambda w$ where
$(\lambda w)(t)=w^{-1}\lambda(t) w$ for $t\in k^\times$. If 
$\langle\cdot,\cdot\rangle:\XXX(T)\times\XXX^\vee(T)\rightarrow\Z$ is the natural
pairing, then the coroots are the
$\Phi^\vee=\{v^\vee\in\XXX^\vee(T)\,|\,v\in\Phi\text{ and }\langle v,v^\vee\rangle=2\}$ 
and the simple coroots are the $\Delta^\vee=\{v^\vee\,|\,v\in\Delta\}$.

Let $E(\ov{T})$ be the idempotents in $\ov{T}$. Thus, $E(\ov{T})$ is a finite commutative 
monoid of idempotents, and we adopt the partial order of 
\S\ref{section:idempotents:generalities}. The resulting poset is actually a 
graded atomic lattice with $\rk(e)=\dim T-\dim Te$ and atoms 
$A=\{e\in E(\ov{T})\,|\,\dim Te=\dim T-1\}$. 
The Weyl group $W$ acts faithfully on $E(\ov{T})$ via $e\mapsto w^{-1}ew$, giving an injection
$W\hookrightarrow \SS_{E(\ov{T})}$. This action preserves the partial order and
the grading, so in particular restricts to the
atoms $A$.

\begin{table}
\centering
\begin{tabular}{lccc}
\hline
$M$&$G_0$&$\Phi$&polytope $P$\\
\hline
general linear $\m_{n}$&$\sl_{n}$&$A_{n-1}$&$\Delta^n$\\
special orthogonal $\mso_{2\ell+1}$&$\so_{2\ell+1}$&$B_\ell$&$\Diamond^\ell$\\
symplectic $\msp_{2\ell}$&$\sp_{2\ell}$&$C_\ell$&$\Diamond^\ell$\\
special orthgonal $\mso_{2\ell}$&$\so_{2\ell}$&$D_\ell$&$\Diamond^\ell$\\
Solomon's example \S\ref{section:renner:permutohedron}&$\sl_{n}$&$A_{n-1}$&$(n-1)$-permutohedron\\
\hline
\end{tabular}
\caption{Basic data for the algebraic monoids considered in \S\ref{section:renner}.}
\label{table:classical}
\end{table}

It turns out that there is a convex polytope $P$ (see \S\ref{section:idempotents:polytopes})
with face lattice $\FF(P)$ 
isomorphic to the lattice $E$. We describe, following \cite{Solomon95}*{\S5}, how this polytope comes
about in the situation that 
$M=\ov{k^\times \rho(G_0)}$ for $\rho:G_0\rightarrow GL(V)$.
Let $m=\dim V$, $\ell=\dim T_0$ and 
$\Phi_0=\Phi(G_0,T_0)$ with simple roots $\Delta_0=\{v_1,\ldots,v_\ell\}$ 
and simple coroots $\Delta_0^\vee$.
For each $i=1,\ldots,\ell$  and simple coroot $v_i^\vee$, let
$\chi^\vee_i:=\rho v_i^\vee\in\XXX^\vee(T)$. We can write 
\begin{equation}
  \label{eq:24}
\chi^\vee_i(t):=\chi(\mathbf{a}_i)^\vee(t)=
\diag(t^{a_{i1}},\ldots,t^{a_{im}})  
\end{equation}
with the $\mathbf{a}_i=(a_{i1},\ldots,a_{im})\in\Z^m$. 
Let $\R^\ell$ be the space of column vectors and 
$P$ the convex hull in $\R^\ell$ of the $m$
vectors $(a_{1j},\ldots,a_{\ell j})^T$. Thus, if $A$ is
the $\ell\times m$ matrix with rows the $\mathbf{a}_i$ then $P$ is the
convex hull in $\R^\ell$ of the columns. 

If $f\in\FF(P)$ is a face of $P$ then define $e_f:=\sum_j E_{jj}$, the sum over
those $1\leq j\leq m$ such that $(a_{1j},\ldots,a_{\ell j})^T\in f$, and with
$E_{ij}$ the matrix with $1$ in row $i$, column $j$, and $0$'s elsewhere.
Then the map
$\zeta:\FF(P)\rightarrow E(\ov{T})$ given by $\zeta(f)=e_f$ is an isomorphism of posets.
The $P$ for the examples of this section are given in Table
\ref{table:classical} (these will be justified later). Actually, even more is true. 
The Weyl group acts on $\XXX^\vee(T)$ via $(\lambda w)(t)=\rho(w)^{-1}\lambda(t)\rho(w)$
so that $\chi(\mathbf{a}_i)^\vee w=\chi(\mathbf{b}_i)^\vee$, with 
$\mathbf{b}_i=(b_{i1},\ldots,b_{im})$ a permutation of $\mathbf{a}_i$.
In particular, $W$ permutes the vertices of $P$ inducing an action of $W$ on
$\FF(P)$. Then the poset isomorphism $\zeta:\FF(P)\rightarrow E(\ov{T})$ is 
equivariant with respect to the Weyl group actions on $\FF(P)$ and $E(\ov{T})$. 

Let $R=\overline{N_G(T)}/T$ be the Renner monoid of $M$--a finite factorizable inverse
monoid with units $W$ and idempotents
$E(\ov{T})$. It turns out that $R$ is \emph{not\/} in general a
reflection monoid, although it is the image of a reflection monoid
with units $W$ and system of subspaces in $\XXX$ (see \cite{Everitt-Fountain10}*{Theorem
8.1}).

For us the Renner monoid will be a monoid of partial permutations using the construction
described at the end of \S\ref{section:reflection.monoids}. To see why 
we will need a result from \cite{Everitt-Fountain10}
which we restate here in abbreviated form:

\begin{proposition}[\cite{Everitt-Fountain10}*{Proposition 2.1}]
\label{section:reflectionmonoids:result200}
Let $M=EG$ and $N=FH$  be factorizable inverse monoids, and 
$\theta:G\rightarrow H$ and $\zeta:E\rightarrow F$ isomorphisms,
such that
\begin{itemize}
\item $\zeta$ is equivariant: $(geg^{-1})\zeta=(g\theta)(e\zeta)(g\theta)^{-1}$ for all
$g\in G$ and $e\in E$, and 
\item $\theta$ respects stablizers: $G_e\theta=H_{e\zeta}$ for all $e\in E$.
\end{itemize}
Then the map $\varphi:M\rightarrow N$ given by
$(eg)\varphi=(e\zeta)(g\theta)$ is an isomorphism.
\end{proposition}

Roughly speaking, two factorizable inverse monoids are the same if their
units are the same, their idempotents are the same, and the actions of
the units on the idempotents are the same.

Now let $E=\FF(P)$ above and $\cS_P$ be the system of intervals for $W$
given, as at the end of \S\ref{section:reflection.monoids}, 
by $E_{\geq f}=\{f'\in\FF(P)\,|\,f'\subseteq f\}$. Let $M(W,\cS_P)$ be the 
resulting monoid of partial permutations, in which every element can be 
written in the form $\id_{E_{\geq f}} w$ for $f$ a face of $P$ and $w\in W$.
The following is then an immediate application of Proposition 
\ref{section:reflectionmonoids:result200} (with $\theta$ the identity):

\begin{proposition}
\label{section:rennermonoids:result1000}
If  $W$ is the Weyl group of $G=G(M)$, $\cS_P$ the system arising from
the polytope $P$ and 
$R$ is the Renner monoid of $M$, then the map 
$\id_{E_{\geq f}} w\mapsto e_f w$ is an isomorphism $M(W,\cS_P)\rightarrow R$.
\end{proposition}


For $e\in E(\ov{T})$ let 
$\Phi_e=\{v\in \Phi\,|\,s_v a=as_v\text{ for all }a\in E(\ov{T})_{\geq e}\}$. The proof 
of \cite{Everitt-Fountain10}*{Theorem 9.2} shows that if $X=E(\ov{T})_{\geq e}$, the isotropy group
$W_X$ is equal to $W(\Phi_e)$, the subgroup of $W$ generated by the
$s_v\,(v\in\Phi_e)$. Moreover, for $v\in\Phi$ and $t=s_v$, we have
$H_t\supseteq E(\ov{T})_{\geq e}$ if and only if $v\in\Phi_e$. 

The conditions of Remark 1 at the end of \S\ref{section:presentation}
are thus satisfied and
we are ready to set up our presentation for the Renner monoid:

\begin{description}
\item[(R1).] 
Let $A=\{e\in E\,|\,\dim Te=\dim T-1\}$ be the atoms of $E(\ov{T})$.
Let $O_k$ be sets defined as in (P1) of \S\ref{section:presentation}.
\item[(R2).] As before $W$ has a presentation with generators the
  $s\in S$ for $S=\{s_v\,|\,v\in\Delta\}$ and 
relations $(st)^{m_{st}}=1$. For each $w\in W$ we fix an expression
$\omega$ for $w$ in the simple reflections $s\in S$ (subject to $\ss=s$).
\item[(R3).] The action of $W$ on $\cS$ is represented notationally 
as before: 
for $a\in A$ fix an $a'\in O_1$ and a $w\in W$ with $a=w^{-1}a'w$
(subject to $w=1$ when $a\in O_1$) and define
$\aa:=\omega^{-1}a'\omega$.
If $w$ is an arbitrary element of $W$ and $a\in A$ then by $\aa^\omega$ we mean the
word obtained in this way for $w^{-1}aw\in A$. As before
this is not necessarily $\omega^{-1}\aa\omega$.
For $e\in\cS$, fix a join $e=\bigvee a_i\,(a_i\in A)$ and define $\ve:=\prod\aa_i$. 
\item[(R4).] For $e\in E(\ov{T})$ let 
$\Phi_e=\{v\in \Phi\,|\,s_v a=as_v\text{ for all }a\in E(\ov{T})_{\geq e}\}$
and let $\{v_1,\ldots,v_\ell\}$
be representatives, with $v_i\in\Delta$, for the $W$-action on $\Phi$.
For $i=1,\dots,\ell$, enumerate the 
pairs $(e,s_i:=s_{v_i})$ where 
$e\in E(\ov{T})$ is minimal in the partial order on $E(\ov{T})$ with the property that $v_i\in\Phi_{e}$.
Let $\mathit{Iso}$ be the set of all such pairs.
\end{description}

With the notation established we can now state the result, the proof of which is a direct translation
of Theorem \ref{presentations:result1000} using Remark 1 at the end of \S\ref{section:presentation}.

\begin{theorem}\label{presentations:result3000}
Let $M$ be an reductive irreducible algebraic monoid with $0$. Then
the Renner monoid of $M$ has a presentation with
\begin{align*}
\text{generators:\hspace{1em}}&s\in S,a\in O_1.&&\\
\text{relations:\hspace{1em}}
&(st)^{m_{st}}=1,\,(s,t\in S),&&\text{(Units)}\\
&a^2=a,\,(a\in O_1),&&\text{(Idem1)}\\
&\aa_1\aa_2=\aa_2\aa_1,\,(\{a_1,a_2\}\in O_2),&&\text{(Idem2)}\\
&\aa_1\ldots\aa_{k-1}=\aa_1\ldots\aa_{k-1}\aa,\,(\{a_1,\ldots,a_{k-1},a\}\in O_{k})&&\\
&\hspace*{0.5cm}\text{ with }a_1,\ldots,a_{k-1}\,,(3\leq k\leq\dim T)\text{ independent and }a\leq\bigvee a_i,&&\text{(Idem3)}\\
&s\aa=\aa^{s}s,\,(s\in S,a\in A),&&\text{(RefIdem)}\\
&\ve s=\ve,\,(e,s)\in\mathit{Iso}.&&\text{(Iso)}\\
\end{align*}
\end{theorem}

All the presentations in this section can be obtained in an
algorithmic way, and so can be implemented in a computer algebra
package for specific calculations.

\subsection{The classical monoids I}
\label{section:renner:classical1}

We illustrate the results of the previous section by giving presentations for the Renner monoids
of the $\ov{k^\times G_0}\subseteq\m_n$ where $G_0$ is one of the
classical groups $\sl_n,\sp_n,\so_n$ (see also \cite{Godelle10}).
We see from Table \ref{table:classical} that while the root systems for
$\so_{2\ell+1}$ and $\sp_{2\ell}$ are different, the resulting Weyl
groups $W(B_\ell)$ and $W(C_\ell)$ turn out to be isomorphic. Indeed,
the Weyl groups $W(A_{n-1}),W(B_n)\cong W(C_n)$ and $W(D_n)$ all have
alternative descriptions as permutation groups: namely, the symmetric group
$\SS_n$ and the groups of signed and even signed permutations
$\SS_{\pm n}$ and $\SS_{\pm n}^e$ (see below for the definitions of
these).

The same is true for the Renner monoids: $\mso_{2\ell+1}$ and
$\msp_{2\ell}$ have isomorphic Weyl groups and isomorphic idempotents,
both $\cong\FF(\Diamond^\ell)$, so it is not surprising that
their Renner monoids are isomorphic. Indeed, the four Renner monoids
can be realized as monoids of partial permutations, with units one
of $\SS_\ell,\SS_{\pm\ell}$ or $\SS_{\pm\ell}^e$, and $E$ one of the
combinatorial descriptions of $\FF(P)$ given in \S\ref{section:idempotents:polytopes}.

Consequently there are two ways to get their presentations, and 
for variety we illustrate both.
For $\m_n=\ov{k^\times\sl_n}$ we just apply (R1)-(R4) and Theorem \ref{presentations:result3000}
directly. For the other three we work instead with their realizations as monoids of 
partial permutations, applying the adapted versions of (P1)-(P4), 
as in Remark 1 at the end of \S\ref{section:presentation}, and then
Theorem \ref{presentations:result1000}.
We then give an isomorphism from these to the Renner monoids. 

Throughout 
$\t_n\subset \gl_n$ is the group of invertible diagonal matrices. 

\begin{example}[the general linear monoid $\m_n$]
\label{eg:general}
Let $G_0=\sl_n$ with $T_0=\sl_n\cap \t_n$ a maximal torus; $G=k^\times G_0=\gl_n$ with 
maximal torus $T=k^\times T_0=\t_n$. The general linear monoid is then 
$\m_n=\ov{k^\times \sl_n}$ with $\ov{T}$ the diagonal matrices.

For $\diag(t_1,\ldots,t_n)\in T$ let $v_i\in\XXX(T)$ be given by 
$v_i\,\diag(t_1,\ldots,t_n)=t_i$. Then $\XXX(T)$ is the free $\Z$-module with basis 
$\{v_1,\ldots,v_n\}$ and $\XXX(T_0)$ the submodule consisting of those
$\sum t_i v_i$
with $\sum t_i=0$. The root system $\Phi(G_0,T_0)=\Phi(G,T)$ has type $A_{n-1}$:
$$
\{v_i-v_j\, (1\leq i\not= j\leq n)\},
$$
with simple system $\Delta=\{v_{i+1}-v_{i}\,(1\leq i\leq n-1)\}$ arising from the Borel subgroup
of upper triangular matrices. 

In this case the Weyl group $W(G,T)$ can be identified with a
subgroup of $G$, namely the set of permutation matrices 
$A(\pi):=\sum_i E_{i,i\pi}$ as $\pi$ varies over the symmetric group $\SS_n$. Indeed, the 
Weyl group is easily seen to be isomorphic to $\SS_n$, but we will stay inside the world of algebraic
groups in this example. The isomorphism $W(G,T)\rightarrow W(A_{n-1})$ is induced by
$A(i,j)\mapsto s_{v_i-v_j}$. 

The idempotents $E=E(\ov{T})$ are the diagonal matrices $\diag(t_1,\ldots,t_n)$ with 
$t_i\in\{0,1\}$ for all $i$. 
Alternatively, for $J\subseteq X=\{1,\ldots,n\}$, let 
$e_J^{}:=\sum_{j \in J} E_{jj}$, so that $E(\ov{T})$ consists of the $e_J^{}$
for $J\in\BB_X$ (and indeed, $E(\ov{T})$ is easily seen to be isomorphic to $\BB_X$,
but again we stay inside algebraic groups). 
The Weyl group action on $E(\ov{T})$ is given by
$$
e_J^{}\mapsto A(\pi)^{-1}e_J^{} A(\pi)=e_{J\pi}.
$$

Let $e_i:=e_J^{}$ for $J=\{1,\ldots,\widehat{i},\ldots,n\}$. 
Running through (R1)-(R4), the atoms in $E(\ov{T})$ are 
$A=\{e_i\,|\,1\leq i\leq n\}$.
There is a single $W$-orbit on $A$ and we choose $e:=e_1$ for $O_1$. There
is a single $W$-orbit on pairs of atoms and we choose the pair $\{e,e_2\}$ for $O_2$. 
We will see below that there is no need for $O_k$ for $k>2$. If
$e_i\in A$, $(i>1)$,
we have $e_i=s_{i-1}\ldots s_{1}es_{1}\ldots s_{i-1}$, so let
$$
\ve_i=s_{i-1}\ldots s_{1}e s_{1}\ldots s_{i-1},
$$
with $\ve_1=e$.
Let $e_J\in E$ with $X\setminus J=\{i_1,\ldots,i_k\}$, giving
$e_J=e_{i_1}\vee\cdots\vee e_{i_k}$, and let
$$
\ve_J=\ve_{i_1}\ldots\ve_{i_k}.
$$
We have $A(\pi)^{-1}e_J^{} A(\pi)=e_J^{}$ exactly when $J\pi=J$;
moreover,
$E_{\geq e_J}=\{e_I\,|\,J\supseteq I\}$. 
The result is that
$$
\Phi_{e_J}=\{v_i-v_j\,|\,i,j\not\in J\}.
$$
There is a single $W$-orbit on the roots $\Phi$ and we choose $v_2-v_1\in\Delta$
as representative. If $e_J$ is to be minimal in $E$ with the property that $v_2-v_1\in\Phi_{e_J}$
then $J$ is minimal (under reverse inclusion!) with $1,2 \not\in J$. Thus
$J=\{3,\ldots,n\},$ and the set $\mathit{Iso}$ consists of the single pair $(e_{\{3,\ldots,n\}},s_1)$
with $\ve_{\{3,\ldots,n\}}=\ve_1\ve_2=es_1es_1$.

\paragraph{A presentation 
of the Renner monoid for\/ $\m_n$:} 
By Theorem \ref{presentations:result3000} 
we have generators $s_1,\ldots,s_{n-1},$ $e$
with \emph{(Units)\/} relations $(s_is_j)^{m_{ij}}=1$, where
the $m_{ij}$ are given by the Coxeter symbol
$$
\begin{pspicture}(0,0)(14,1)
\rput(3,-0.5){
\rput(4,1){\BoxedEPSF{fig2.eps scaled 400}}
\rput(1.65,.6){$s_1$}\rput(3,.6){$s_2$}
\rput(4.9,.6){$s_{n-2}$}\rput(6.35,.6){$s_{n-1}$}
}
\end{pspicture}
$$
(recalling, as in \S\ref{section:popova:boolean}, that the nodes are joined
by an edge labeled $m_{ij}$ if $m_{ij}\geq 4$, an unlabelled edge
if $m_{ij}=3$, no edge if $m_{ij}=2$, and $m_{ij}=1$ when $i=j$).
The \emph{(Idem1)\/} relation is $e^2=e$ the \emph{(Idem2)\/}
relations are 
$$
\ve_1\ve_2=\ve_2\ve_1,\text{ or, }es_1es_1=s_1es_1e.
$$
We saw in \S\ref{section:idempotents:generalities}
that in $\BB_X$ (or in \S\ref{section:idempotents:polytopes} that in $\FF(\Delta^n)$) all
subsets of atoms are independent and so the \emph{(Idem3)\/} relations are vacuous. 
The \emph{(RefIdem)\/} relations are $s_i\ve_j=\ve_j^{s_i}s_i$ for
$1\leq i\leq n-1$ and $1\leq j\leq n$, but just as in Lemma \ref{popova:boolean:result100}
of \S\ref{section:popova:boolean} we can prune these down to
$s_ie=\ve^{s_i}s_i\text{ for }1\leq i\leq n-1$.
We have $s_1es_1=s_2$ and $s_ies_i=e\,(i>1)$ so that 
$\ve^{s_1}=\ve_2=s_1es_1$,
$\ve^{s_i}=e\,(i>1)$ and the relations are
$$
s_ie=es_i\,(i>1).
$$
Finally, the \emph{(Iso)\/} 
are $\ve_Js_1=\ve_J$ for $J=\{3,\ldots,n\}$, or 
$$
es_1e=es_1es_1.
$$
\begin{remark*}
It is well known that $R$ is isomorphic to the symmetric inverse monoid $\II_n$ where
the $s_i$ correspond to the (full) permutation $(i,i+1)$ and $e$ to the partial identity
on the set $\{2,\ldots,n\}$ (see also Figure \ref{fig1}). Thus we have, yet again, the Popova
presentation that we found in \S\ref{section:popova:boolean} for the Boolean monoid
$M(A_{n-1},\BB)$. The symmetric inverse monoid, in the context of Renner monoids, is often
called the Rook monoid.
\end{remark*}
\end{example}

As promised we now introduce two families of monoids of partial permutations.
Let $\pm X=\{\pm 1,\ldots,\pm \ell\}$ and define the group 
$\SS_{\pm X}$ of signed permutations of $X$ to be
$$
\SS_{\pm X}=\{\pi\in\SS_{X\cup -X}\,|\,(-x)\pi=-x\pi\text{ for all }x\in\pm X\}.
$$
(the reason for the change in notation from $n$ to $\ell$ will become
apparent in Example \ref{eg:symplectic} below).
A signed permutation $\pi$ is \emph{even\/} if the number of $x\in X$ with
$x\pi\in-X$ is even, and the even signed permutations $\SS_{\pm X}^e$ form a subgroup
of index two in $\SS_{\pm X}$. 

The symmetric group is a subgroup in an obvious
way: let $\pi\in\SS_{\pm X}$ be such that $x$ and $x\pi$ have the same
sign for all $x\in\pm X$. In particular $\pi$ is even. Any such $\pi$ has a unique expression
$\pi=\pi_+\pi_-$ with $\pi_+\in\SS_X$, $\pi_-\in\SS_{-X}$ and
$\pi_+(x)=\pi_-(-x)$. The map $\pi\mapsto\pi_+$ is then an isomorphism
from the set of such $\pi$ to $\SS_X$. We will just write
$\SS_X\subset\SS_{\pm X}$ (or $\subset\SS_{\pm X}^e$) from now on to mean this subgroup.

We require Coxeter system structures for $\SS_{\pm X}$ and $\SS_{\pm X}^e$. Indeed, we have
$\SS_{\pm X}\cong W(B_\ell)\cong W(C_\ell)$ via 
$s_{v_1}$ or $s_{2v_1}\mapsto(1,-1)$ and 
$s_{v_{i+1}-v_{i}}\mapsto(i,i+1)(-i,-i-1)$ and 
$\SS_{\pm X}\cong W(D_\ell)$ via $s_{v_1+v_2}\mapsto(1,-2)(-1,2)$ and 
$s_{v_{i+1}-v_{i}}\mapsto(i,i+1)(-i,-i-1)$.

Now to a system of subsets for $\SS_{\pm X}$ and $\SS_{\pm X}^e$. 
In \cite{Everitt-Fountain10}*{\S 5} we used the elements of $\BB_X$ to
give a system for $\SS_{\pm X}$ and this lead to the monoid $\II_{\pm n}$ of partial signed permutations.
Here we want something different.
Recall from Example
\ref{eg:octahedron} the poset $E$ of admissible subsets of $\pm X$, with
$\pm X$ adjoined. If $\pi\in\SS_{\pm X}$ and $J$ is admissible,
then it is easy to see that $J\pi$ is also admissible, and so the
action of $\SS_{\pm X}$ on $\pm X$ restricts to $E$.
Our system consists of the intervals $E_{\geq J}=\{I\in
E\,|\,J\supseteq I\}$ as in \S\ref{section:reflection.monoids}.

Write $M(\SS_{\pm X},\cS)$ and $M(\SS_{\pm X}^e,\cS)$ for the resulting monoids
of partial permutations. 

\subsection{A brief interlude}
\label{section:renner:interlude}

We 
detour to parametrize the orbits of the action 
$(J_1,\ldots,J_k)\stackrel{\pi}{\mapsto}(J_1\pi,\ldots,J_k\pi)$ of 
the symmetric group $\SS_X$ on 
$k$-tuples $(J_1,\ldots,J_k)$ of distinct subsets of $X$.
This description will be useful in obtaining the sets $O_k$ for
the monoids of partial permutations that appear in 
\S\S\ref{section:renner:classical2}-\ref{section:renner:permutohedron}.
The results of this subsection may well be part of the folklore of the
combinatorics of the symmetric group, but for completeness we include
a full discussion.
Let $X=\{1,\ldots,\ell\}, Y=\{1,\ldots,k\}$ with $\BB_Y$ the Boolean lattice on $Y$, ordered as
usual by reverse inclusion, and $[0,\ell]\subset\Z$, with
this interval inheriting the usual order from $\Z$.

Let $f:\BB_Y\rightarrow [0,\ell]$ be a poset map
and define
$f^*:\BB_Y\rightarrow\Z$ (not necessarily a poset map) by
\begin{equation}
  \label{eq:5}
{\textstyle f^*(I)=\sum_{J\supseteq I}(-1)^{|J\setminus I|}f(J).}
\end{equation}
Then $f$ is a \emph{characteristic map\/} if $f^*(I)\geq 0$ for all
$I$, and $f^*(\varnothing)=0$. 

A $k$-tuple $(J_1,\ldots,J_k)$ of subsets of $X$ gives rise to a
characteristic map as follows: define 
$f:\BB_Y\rightarrow [0,\ell]$ by $f(I)=|\bigcap_{i\in I}J_i|$ for $I$
non empty and $f(\varnothing)=|\bigcup_{i=1}^k J_i|$.
If $J\supseteq I$ then $\bigcap_J J_i\subseteq\bigcap_I J_i$, so that
$f(J)\leq f(I)$ and $f$ is a poset map. 
The number $f^*(I)$ is the 
cardinality of the set
$${\textstyle 
\bigl(\bigcap_I J_i\bigr)\setminus\bigl(\bigcup_{J\supset
  I}\bigcap_J J_i\bigr),
}
$$
so that $f^*(I)\geq 0$ for all $I$, and $f^*(\varnothing)=0$ (by inclusion-exclusion).

In fact every characteristic map arises from a tuple
$(J_1,\ldots,J_k)$ of distinct subsets in this way. 
For, let $f$ be an arbitrary characteristic map,
and let disjoint sets $K_I$, $(\varnothing\not=I\in\BB_Y)$ be defined by
first setting
$K_Y:=\{1,\ldots,f^*(Y)=f(Y)\}$ if $f(Y)>0$, or $K_Y:=\varnothing$ if
$f(Y)=0$. 
Now choose some total ordering $\preceq$ on $\BB_Y$
having minimal element $Y$, and for general
$I$ let $K_I$ be the
next $f^*(I)$ points of 
$[0,\ell]\setminus\bigcup_{J\prec\,I}K_J$. 
Although the choice of $\preceq$ is not important, for
definiteness we take $J\prec I$ when $|I|<|J|$ and order sets of the
same size lexicographically. 
Then for $i=1,\ldots,k$, let $J_i=\bigcup K_I$, the
(disjoint) union over those $I$ with $i\in I$. Finally, let $f'$ be
the characteristic map of the resulting tuple $(J_1,\ldots,J_k)$.

We claim that this construction makes sense and that
$f=f'$. Firstly, it is the fact that $K_I$ is to have $f^*(I)$
elements that forces the $f^*(I)\geq 0$ condition in the
definition of characteristic map.  
Next, recall that for $J\supseteq I$, the
M\"{o}bius function $\mu$ of a Boolean lattice is given by
$\mu(J,I)=(-1)^{|J\setminus I|}$. Thus, M\"{o}bius inversion applied
to (\ref{eq:5}) (see,
e.g.: \cite{Stanley97}*{\S3.7}) gives $f(I)=\sum_{J\supseteq I}f^*(J)$ for all $I$ in $\BB_Y$. In particular, 
$\ell\geq f(\varnothing)=\sum_J f^*(J)=\sum_J |K_J|$, and so there are enough
elements in the interval $[0,\ell]$ to house the disjoint sets $K_J$. Now let 
$\varnothing\not=I\in\BB_Y$. Then 
$$
{\textstyle f'(I)=\bigl|\bigcap_{i\in I}
  J_i\bigr|=\bigl|\bigcup_{J\supseteq I}K_J\bigr|=\sum_{J\supseteq I}f^*(J)=f(I).}
$$
Finally, $f'(\varnothing)=|\bigcup_{i=1}^k J_i|=|\bigcup_J K_J|=\sum_{J\not=\varnothing}f^*(J)$.
In particular we have
$f(\varnothing)-f'(\varnothing)=f^*(\varnothing)=0$. Thus $f=f'$. 
If $f$ is a characteristic map then we write $(J_1,\ldots,J_k)_f$ for
the tuple arising from it.

It is easy to see that
two $k$-tuples of subsets of $X$ lie in the same $\SS_X$-orbit exactly when the
corresponding 
characteristic maps are identical. Thus, 

\begin{lemma}\label{renner:result100}
Let $X=\{1,\ldots,\ell\}$ and $Y=\{1,\ldots,k\}$.
Then the orbits of the diagonal action of $\SS_X$ on 
$k$-tuples of distinct subsets of $X$ are
parametrized by the 
characteristic maps, i.e.: the 
poset maps $f:\BB_Y\rightarrow[0,\ell]$ satisfying
$f^*(I)\geq 0$ for all $I$ and $f^*(\varnothing)=0$, where $f^*$ is
defined by (\ref{eq:5}).
\end{lemma}

We write $\mathit{Char}_k$ for the set of characteristic maps 
$f:\BB_Y\rightarrow[0,\ell]$ when $|Y|=k$. Although $\mathit{Char}_k$  depends on both $k$
and $\ell$, in the examples below $\ell$ will be fixed. 

For fixed $k$ the possible characteristic maps in
$\mathit{Char}_k$ can be enumerated by letting $f(Y)=f^*(Y)=n_0\geq
0$. If $I=Y\setminus\{i\}$ then $f^*(I)=f(I)-f(Y)\geq 0$ gives $f(I)=n_i\geq
n_0$. In general, if $I=Y\setminus J$, $(J\subset Y)$ then $f(I)$ can
equal any $n_J^{}\in [0,\ell]$ satisfying 
$n_J^{}\geq\sum_{K\subseteq  J}(-1)^{|J\setminus K|}n_K^{}$ 
(and $f(\varnothing)=\sum_{J\not=\varnothing}(-1)^{|J|+1}n_J^{}$).

For example, if
$k=1$ then $\BB_Y$ is the two element poset $Y<\varnothing$. We have
$f^*(Y)=f(Y)\geq 0$ and $f(\varnothing)=f(Y)$. Thus
$\mathit{Char}_1$ consists of the $f(\varnothing)=f(Y)=n_0$, for each
$n_0\in[0,\ell]$, of which there are $\ell+1$. This coincides
with the fact that $\SS_X$ acts $t$-fold transitively on $X$ for each
$0\leq t\leq \ell$, hence there are $\ell+1$ orbits. 
Figure \ref{fig;characteristic.maps} shows the possibilities for
$k=1,2$ and $3$. 
For example, explicit orbit representatives
$(J_1,J_2,J_3)_f$ when $k=3$ can be obtained as follows:
let
$n_0,\ldots,n_3,n_{12},n_{13},n_{23}$ be integers satisfying the
conditions on the far right of Figure
\ref{fig;characteristic.maps}. The following picture depicts
$X=\{1,\ldots,\ell\}$, with $1$ at the left:
$$
\begin{pspicture}(0,0)(14,1.5)
\rput(-2.5,-.2){
\psframe[fillstyle=solid,fillcolor=lightgray,linecolor=white,linewidth=.05mm](2.2,0.5)(3,1.5)
\psframe[fillstyle=solid,fillcolor=lightgray,linecolor=white,linewidth=.05mm](4.4,0.5)(7.2,1.5)
\psframe[fillstyle=solid,fillcolor=lightgray,linecolor=white,linewidth=.05mm](13.2,0.5)(16.2,1.5)
\rput(2.6,1){$n_0$}
\rput(3.7,1){$n_1-n_0$}
\rput(5.1,1){$n_2-n_0$}
\rput(6.5,1){$n_3-n_0$}
\rput(8.7,1){$n_{12}-n_1-n_2+n_0$}
\rput(11.7,1){$n_{13}-n_1-n_3+n_0$}
\rput(14.7,1){$n_{23}-n_2-n_3+n_0$}
\psline[linewidth=.2mm]{-}(2.2,.5)(16.2,.5)
\psline[linewidth=.2mm]{-}(2.2,1.5)(16.2,1.5)
\psline[linewidth=.2mm]{-}(2.2,0.5)(2.2,1.5)
\psline[linewidth=.2mm]{-}(3,0.5)(3,1.5)
\psline[linewidth=.2mm]{-}(4.4,0.5)(4.4,1.5)
\psline[linewidth=.2mm]{-}(5.8,0.5)(5.8,1.5)
\psline[linewidth=.2mm]{-}(7.2,0.5)(7.2,1.5)
\psline[linewidth=.2mm]{-}(10.2,0.5)(10.2,1.5)
\psline[linewidth=.2mm]{-}(13.2,0.5)(13.2,1.5)
\psline[linewidth=.2mm]{-}(16.2,0.5)(16.2,1.5)
\rput(2.6,.3){$\scriptstyle{\{1,2,3\}}$}
\rput(3.7,.3){$\scriptstyle{\{2,3\}}$}
\rput(5.1,.3){$\scriptstyle{\{1,3\}}$}
\rput(6.5,.3){$\scriptstyle{\{1,2\}}$}
\rput(8.7,.3){$\scriptstyle{\{3\}}$}
\rput(11.7,.3){$\scriptstyle{\{2\}}$}
\rput(14.7,.3){$\scriptstyle{\{1\}}$}
}
\rput(14.3,0.75){$(\dag)$}
\end{pspicture}
$$
and the number in each box gives the number of points in the box
(so the
left most box represents the points $\{1,\ldots,n_0\}$, the second the
points $\{n_0+1,\ldots,n_1\}$, and so on). Each box is also labeled
below by a subset of $Y$. Then $J_i$ is
the union of those boxes for which $i$ appears in the subset
below it; e.g.: $J_1$ is the union of the grey boxes.

\begin{figure}
\begin{pspicture}(0,0)(14,4.25)
\rput(0,.2){
\rput(0,0){
\rput(0,0.75){
\psline[linewidth=.2mm]{-}(1,0.5)(1,2)
\pscircle[linewidth=.2mm,fillcolor=white,fillstyle=solid](1,0.5){1.5mm}
\pscircle[linewidth=.2mm,fillcolor=white,fillstyle=solid](1,2){1.5mm}
\rput(1.35,0.5){$n_0$}\rput(1.35,2){$n_0$}
}
\rput(1,.5){$0\leq n_0\leq \ell$}
}
\rput(0,0.25){
\rput(5,1){
\rput{45}{\psframe[linewidth=.2mm](1.7,1.7)
\pscircle[linewidth=.2mm,fillcolor=white,fillstyle=solid](0,0){1.5mm}
\pscircle[linewidth=.2mm,fillcolor=white,fillstyle=solid](0,1.7){1.5mm}
\pscircle[linewidth=.2mm,fillcolor=white,fillstyle=solid](1.7,0){1.5mm}
\pscircle[linewidth=.2mm,fillcolor=white,fillstyle=solid](1.7,1.7){1.5mm}
}
\rput(0.35,0){$n_0$}\rput(-1.55,1.2){$n_1$}\rput(1.6,1.2){$n_2$}\rput(1,2.4){$n_1+n_2-n_0$}
}
\rput(-3,-1){
\rput(8,1.35){$0\leq n_0\leq n_1,n_2\leq\ell$}\rput(8,0.9){$n_1+n_2-n_0\leq\ell$}
}}
\rput(2.25,0){
\rput(5.25,0){
\psline[linewidth=.2mm]{-}(2,0)(0.5,1.25)
\psline[linewidth=.2mm]{-}(2,0)(2,1.25)
\psline[linewidth=.2mm]{-}(2,0)(3.5,1.25)
\psline[linewidth=.2mm]{-}(0.5,1.25)(0.5,2.5)
\psline[linewidth=.2mm]{-}(0.5,1.25)(2,2.5)
\psline[linewidth=.2mm]{-}(2,1.25)(0.5,2.5)
\psline[linewidth=.2mm]{-}(2,1.25)(3.5,2.5)
\psline[linewidth=.2mm]{-}(3.5,1.25)(2,2.5)
\psline[linewidth=.2mm]{-}(3.5,1.25)(3.5,2.5)
\psline[linewidth=.2mm]{-}(0.5,2.5)(2,3.75)
\psline[linewidth=.2mm]{-}(2,2.5)(2,3.75)
\psline[linewidth=.2mm]{-}(3.5,2.5)(2,3.75)
\pscircle[linewidth=.2mm,fillcolor=white,fillstyle=solid](2,0){1.5mm}
\pscircle[linewidth=.2mm,fillcolor=white,fillstyle=solid](0.5,1.25){1.5mm}
\pscircle[linewidth=.2mm,fillcolor=white,fillstyle=solid](2,1.25){1.5mm}
\pscircle[linewidth=.2mm,fillcolor=white,fillstyle=solid](3.5,1.25){1.5mm}
\pscircle[linewidth=.2mm,fillcolor=white,fillstyle=solid](0.5,2.5){1.5mm}
\pscircle[linewidth=.2mm,fillcolor=white,fillstyle=solid](2,2.5){1.5mm}
\pscircle[linewidth=.2mm,fillcolor=white,fillstyle=solid](3.5,2.5){1.5mm}
\pscircle[linewidth=.2mm,fillcolor=white,fillstyle=solid](2,3.75){1.5mm}
\rput(2.4,0){$n_0$}
\rput(0.9,1.25){$n_1$}\rput(2.4,1.25){$n_2$}\rput(3.9,1.25){$n_3$}
\rput(0.95,2.5){$n_{12}$}\rput(2.45,2.5){$n_{13}$}\rput(3.95,2.5){$n_{23}$}
\rput(3.3,3.75){$\sum n_{ij}-\sum n_i+n_0$}
}
\rput(11,2.3){$0\leq n_0\leq n_i\leq \ell$}
\rput(11,1.75){$n_i+n_j-n_0\leq n_{ij}\leq\ell$}
\rput(11,1.25){$\sum n_{ij}-\sum n_i+n_0\leq \ell$}
}}
\end{pspicture}
  \caption{The sets $\mathit{Char}_1$, $\mathit{Char}_2$ and $\mathit{Char}_3$: the elements of
the Boolean lattice $\BB_Y$, ($|Y|=1,2$ and $3$) are labeled by their images under 
a characteristic map $f:\BB_Y\rightarrow [0,\ell]$.}
  \label{fig;characteristic.maps}
\end{figure}
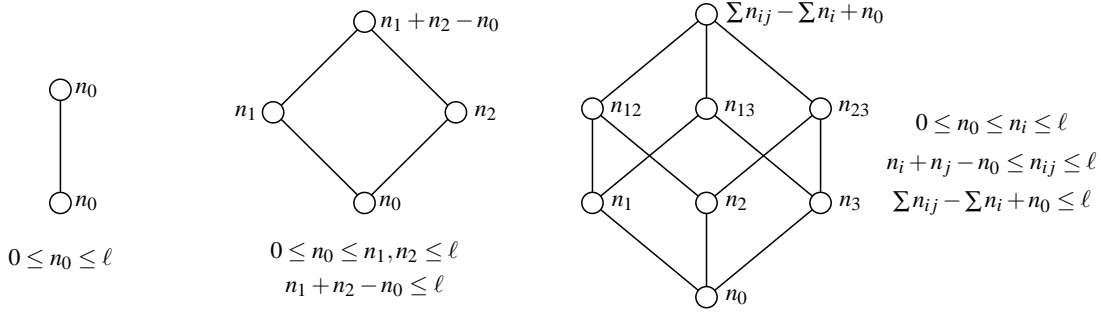

\subsection{The classical monoids II}
\label{section:renner:classical2}

We now return to the monoids $M(\SS_{\pm X},\cS)$ and $M(\SS_{\pm X}^e,\cS)$
from \S\ref{section:renner:classical1}. For the rest of the paper, all
mention of (P1)-(P4) refers to the adapted versions of these as in
Remark 1 at the end of \S\ref{section:presentation}.
As observed in \S\ref{section:reflection.monoids},
the map $J\mapsto E_{\geq J}$ is a poset isomorphism 
$E\cong\cS$ which is equivariant with respect
to the $\SS_{\pm X}$ and $\SS_{\pm X}^e$ actions. Thus, in running
through (P1)-(P4) we can work
with the admissible $J\in E$ rather than the corresponding intervals
$E_{\geq J}\in\cS$. This makes the notation a little less cumbersome.

\paragraph{(1). The monoid $M(\SS_{\pm X},\cS)$.}
The atoms in $E$ are the $a(I):=I\cup(-X\setminus -I),\,I\subseteq X=\{1,\ldots,\ell\}$ of
\S\ref{section:idempotents:polytopes}.
There are thus $2^\ell$ atoms here  versus the $\ell$ atoms in the $\sl_n$ case. 
Now to the sets $O_k$ for $k\geq 1$. 
Let $a(I)$ be an atom of $E$ with $I=\{i_1,\ldots,i_k\}$. Then 
\begin{equation}
  \label{eq:11}
a(I)\cdot (i_1,-i_1)\cdots(i_k,-i_k)=-X=a(\varnothing),  
\end{equation}
so there is a single $\SS_{\pm X}$-orbit on the atoms, and we take
$O_1=\{a\}$ with $a:=a(\varnothing)$. 

For $O_2$ we can use the set $\mathit{Char}_2$ of the previous
section, although it turns out that with the $\SS_{\pm X}$ action we do
can do a little more. 
Let $a(I),a(K)$ be a pair of atoms with $|I|\leq |K|$ and $I\cap
K=\{i_1,\ldots,i_k\}$. Then
$$
(a(I),a(K))\cdot(i_1,-i_1)\cdots(i_k,-i_k)=(a(I_1),a(K_1))
$$ 
with 
$I_1=I\setminus(I\cap K)$ and $K_1=K\setminus(I\cap K)$ disjoint. The
pair $I_1,K_1$ can then be moved by the
$\SS_X$-action as far as possible to the left of
$\{1,\ldots,\ell\}$ while remaining disjoint. Thus, for $O_2$ we take the pairs $\{a(I),a(K)\}$ with
$I=\{1,\ldots,j_1\}$, $K=\{j_1+1,\ldots,j_2\}$ for all
$0\leq j_1<j_2\leq \ell$.

For $k=3$ we can play a similar game, but this doesn't work for
$k>3$. Instead, for 
$k>2$ we restrict the $\SS_{\pm X}$-action on $E$ to the subgroup
$\SS_X\subset\SS_{\pm X}$
and consider orbit representatives on the $k$-tuples
as in remark 3 at the end of \S\ref{section:presentation}. 
Thus the $O_k\,,(k>2)$ will be 
sets of representatives with possible redundancies.
If $f\in\mathit{Char}_k$ is a characteristic map, then by the
construction preceding Lemma \ref{renner:result100} we have a unique
tuple $(I_1,\ldots,I_k)_f$ with characteristic map $f$. For $O_k$ we
take the set of $\{a(I_1),\ldots,a(I_k)\}$ where $(I_1,\ldots,I_k)_f$
arises via $f\in\mathit{Char}_k$.

Write $s_i:=(i,i+1)(-i,-i-1),\,(1\leq i\leq \ell-1)$ and $s_0:=(1,-1)$ and let
$$
\omega_i:=s_{i-1}\cdots s_1s_0s_1\cdots s_{i-1}
$$
for $i>1$ and
$\omega_1=s_0$. If $a(I)$ is an atom with 
$I=\{i_1,\ldots,i_k\}$, let 
\begin{equation}
  \label{eq:12}
  \aa(I):=\omega_{i_1}\ldots\omega_{i_k} a\,\omega_{i_k}\ldots\omega_{i_1}.
\end{equation}

Let $J\in E$ be admissible with 
$\pm X\setminus\pm J=\{\pm i_1,\ldots,\pm i_k\}$. Then, recalling that
$J^+=J\cap X$, we 
take as fixed word for $J$
\begin{equation}
  \label{eq:12a}
\aa(\,\widehat{i}_1,\ldots,i_k,J^+)\cdots\aa(i_1,\ldots,\widehat{i}_k,J^+)
\end{equation}
when $k>1$ (and where $\aa(\,\widehat{i}_1,\ldots,i_k,J^+)$ means 
$\aa(\,\{\widehat{i}_1,\ldots,i_k\}\cup J^+)$), or $\aa(J^+)\aa(i_1,J^+)$ when $k=1$.

Finally then to (P4) and $\AA=\{H_t\}$ where $H_t=\{J\in E\,|\,Jt=J\}$. Every $t\in T$
in $\SS_{\pm X}$ is conjugate to $s_0$ or $s_1$ (using the Coxeter group structure from the 
end of \S\ref{section:renner:classical1}) so there are two $\SS_{\pm X}$ orbits on $\AA$
with representatives 
$H_0:=H_{s_0}$ and $H_1:=H_{s_1}$, 
where $H_0$ consists of those $J\in E$ with $\pm 1\not\in J$ and 
$H_1$ those $J$ with either $\pm 1,\pm 2\not\in J$ or $1,2\in J$ or $-1,-2\in J$. 
If $J$ is to be minimal with $H_0\supseteq E_{\geq J}$ then
$J$ has the form 
\begin{equation}
  \label{eq:13}
\begin{pspicture}(0,0)(14,1)
\psline[linewidth=1pt]{-}(3,1)(11,1)
\psline[linewidth=1pt]{-}(3,0.5)(11,0.5)
\psline[linewidth=1pt]{-}(3,0)(11,0)
\psline[linewidth=1pt]{-}(3,0)(3,1)
\psline[linewidth=1pt]{-}(4,0)(4,1)
\psline[linewidth=1pt]{-}(11,0)(11,1)
\psframe[fillstyle=solid,fillcolor=black](4,.5)(6,1)
\psframe[fillstyle=solid,fillcolor=black](6,.5)(9,0)
\psframe[fillstyle=solid,fillcolor=black](9,.5)(11,1)
\rput(3.5,.75){$1$}\rput(3.5,.25){$-1$}
\end{pspicture}  
\end{equation}
which is $\aa(J^+)\aa(1,J^+)$.
Similarly, if $J$ is to be minimal with $H_1\supseteq E_{\geq J}$ then
$J$ has the form 
\begin{equation}
  \label{eq:14}
  \begin{pspicture}(0,0)(14,1)
\psline[linewidth=1pt]{-}(3,1)(11,1)
\psline[linewidth=1pt]{-}(3,0.5)(11,0.5)
\psline[linewidth=1pt]{-}(3,0)(11,0)
\psline[linewidth=1pt]{-}(3,0)(3,1)
\psline[linewidth=1pt]{-}(4,0)(4,1)
\psline[linewidth=1pt]{-}(5,0)(5,1)
\psline[linewidth=1pt]{-}(11,0)(11,1)
\psframe[fillstyle=solid,fillcolor=black](5,.5)(6,1)
\psframe[fillstyle=solid,fillcolor=black](6,.5)(9,0)
\psframe[fillstyle=solid,fillcolor=black](9,.5)(11,1)
\rput(3.5,.75){$1$}\rput(3.5,.25){$-1$}
\rput(4.5,.75){$2$}\rput(4.5,.25){$-2$}
\end{pspicture}
\end{equation}
which is $\aa(1,J^+)\aa(2,J^+)$.

The set \emph{Iso\/} thus consists of the pairs  
$$
(\aa(1,I)\aa(2,I),s_1)
$$
for all $I\subseteq X\setminus\{1,2\}$ and 
$$
(\aa(I)\aa(1,I),s_0)
$$ 
for all
$I\subseteq X\setminus\{1\}$. Rather than write out the resulting presentation
for this monoid here, we save it for Example \ref{eg:symplectic} below.

\paragraph{(2). The monoid $M(\SS_{\pm X}^e,\cS)$.}
The difference here is that we pass to the subgroup
$\SS_{\pm X}^e$ of $\SS_{\pm X}$ and its action on $E$.
The atoms are the $a(I):=I\cup(-X\setminus -I),\,I\subseteq X$ as before.
If $a(I)$ is one such with $I=\{i_1,\ldots,i_k\}$, then for $k$ even 
\begin{equation}
  \label{eq:20}
a(I)\cdot (i_1,-i_2)(-i_1,i_2)\cdots(i_{k-1},-i_k)(-i_{k-1},i_k)=-X=a(\varnothing),  
\end{equation}
and for $k$ odd
\begin{equation}
  \label{eq:21}
a(I)\cdot
(i_1,-i_2)(-i_1,i_2)\cdots(i_{k-2},-i_{k-1})(-i_{k-2},i_{k-1})(i_{k-1},i_{k})(-i_{k-1},-i_{k})
\cdots(1,2)(-1,-2)
\end{equation}
gives $a(1)$. 
Although (\ref{eq:11}) still holds in the even case, we change here to the
version (\ref{eq:20}) because of our choice of generators for
$\SS_{\pm X}^e$ below. Thus, $O_1=\{a_1:=a(\varnothing),a_2=a(1)\}$. 

Let $a(I),a(K)$ be a pair of atoms with $I\cap
K=\{i_1,\ldots,i_k\}$. Then a similar argument as in the
$\SS_{\pm X}$ case gives $O_2$ the pairs $\{a(I),a(K)\}$ with
$I=\{1,\ldots,j_1\}$, $K=\{j_1+1,\ldots,j_2\}$ (when $k$ is even) 
or $I=\{1,\ldots,j_1\}$, $K=\{j_1,\ldots,j_2\}$
(when $k$ is odd), with $0\leq j_1<j_2\leq \ell$ in both cases.

The $O_k\,,(k>2)$ are exactly as in the $\SS_{\pm X}$ case,
since $\SS_X\subset\SS_{\pm X}^e$.
Thus $O_k$ is the set of 
$\{a(I_1),\ldots,a(I_k)\}$ where $(I_1,\ldots,I_k)_f$
arises via $f\in\mathit{Char}_k$.

Write $s_i:=(i,i+1)(-i,-i-1),\,(1\leq i\leq \ell-1)$ and $s_0:=(1,-2)(-1,2)$ and let
$$
\omega_{ij}:=s_{i-1}\cdots s_1s_{j-1}\cdots s_2s_0s_2\cdots s_{j-1}s_1\cdots s_{i-1}
$$
for $1<i<j\leq\ell-1$, or 
$\omega_{1j}:=s_{j-1}\cdots s_2s_0s_2\cdots s_{j-1},\,(j>2)$ 
or $\omega_{12}:=s_0$.
If $a(I)$ is an atom with $I=\{i_1,\ldots,i_k\}$ and $k$ even, let 
\begin{equation}
  \label{eq:22}
  \aa(I):=\omega_{i_1i_2}\ldots\omega_{i_{k-1}i_k} a_1\,\omega_{i_{k-1}i_k}\ldots\omega_{i_1i_2},
\end{equation}
or if $k$ is odd
\begin{equation}
  \label{eq:23}
  \aa(I):=\omega_{i_1i_2}\ldots\omega_{i_{k-2}i_{k-1}}s_{i_{k-1}}\cdots s_1
a_2\,s_1\cdots s_{i_{k-1}}\omega_{i_{k-2}i_{k-1}}\ldots\omega_{i_1i_2}.
\end{equation}

If $J\in E$ is admissible then it is represented by the word (\ref{eq:12a}) and
the comments following it.
The treatment of (P4) is also virtually identical to the previous case:
every $t\in T$ is conjugate in $\SS_{\pm X}^e$ to $s_1$, so
there is a single $\SS_{\pm X}^e$-orbit on $\AA$ with representative
$H_1:=H_{s_1}$ consisting of the $J\in E$ with either $\pm 1,\pm 2\not\in J$ or
$1,2\in J$ or $-1,-2\in J$. 
The set \emph{Iso\/} thus consists of the pairs  
$(\aa(1,I)\aa(2,I),s_1)$
for all $I\subseteq X\setminus\{1,2\}$.
Again, we save the presentation of this monoid for Example \ref{eg:evenorthogonal} below.

\begin{example}[the symplectic monoids $\msp_n$]
\label{eg:symplectic}
Let $n=2\ell$ and  
$$
G_0=\sp_n=\{g\in\gl_n\,|\,g^TJg=J\}
\text{ for } 
J=\left[\begin{array}{cc}
0&J_0\\-J_0&0  
\end{array}\right],
$$
where $J_0=\sum_{i=1}^{\ell} E_{i,\ell-i+1}$ is $\ell\times\ell$.
Note that as in \cite{Li_Renner03}, this is the version of the
symplectic group given by Humphreys \cite{Humphreys75} rather than the version
used by Solomon in \cite{Solomon95}.
Let $T_0=\sp_n\cap\t_n$, the matrices of the form
$$
\diag(t_1,\ldots,t_\ell,t_\ell^{-1},\ldots,t_1^{-1})
$$
with the $t_i\in k^\times$. Let $G=k^\times\sp_n$ with maximal torus 
$T=k^\times T_0$, and let
the symplectic monoid
$\msp_n=\ov{k^\times\sp_n}\subset\m_n$.

For $i=0,\ldots,\ell$ let $v_i\in\XXX(T)$ be given by 
$$
v_i\,t_0\cdot\diag(t_1,\ldots,t_\ell,t_\ell^{-1},\ldots,t_1^{-1})
=t_i
$$
so that $\XXX(T)$ is the free $\Z$-module on $\{v_0,\ldots,v_\ell\}$. The roots
$\Phi(G_0,T_0)=\Phi(G,T)$ have type $C_\ell$:
$$
\{\pm v_i\pm v_j\, (1\leq i< j\leq\ell)\}\cup\{\pm 2v_i\,(1\leq i\leq\ell)\}
$$
lying in an $\ell$-dimensional subspace of $\XXX=\XXX(T)\otimes\R$. The group
$G$ has rank $\ell+1$ and semisimple rank $\ell$.
We use the simple system $\Delta=\{2v_1,v_{i+1}-v_i\,(1\leq i\leq
\ell-1)\}$ as described in \S\ref{section:reflection.monoids}.

We now describe an isomorphism between the Renner monoid $R$ of
$\msp_{2\ell}$ and the monoid $M(\SS_{\pm\ell},\cS)$ of partial
isomorphisms described in (1) above. 
The units in $M(\SS_{\pm\ell},\cS)$ are $\SS_{\pm\ell}$ and the units
in the Renner monoid $R$ are the Weyl group $W(C_\ell)$ so we have the isomorphism
$\SS_{\pm\ell}\cong W(C_\ell)$ given at the end of \S\ref{section:renner:classical1}.
Let $s_0,\ldots,s_{\ell-1}$ denote either the signed permutations of
$\SS_{\pm\ell}$ introduced in (1) above or the 
simple reflections in $W(C_\ell)$. 

The idempotents in $M(\SS_{\pm\ell},\cS)$ are the partial
identities $\id_{E_{\geq J}}$ on the $E_{\geq J}\in\cS$, and the idempotents in $R$ are $E(\ov{T})$, the 
matrices
$\diag(t_1,\ldots,t_\ell,t_\ell^{-1},\ldots,t_1^{-1})$
with $t_i\in\{0,1\}$.
We write $E$ for the
idempotents in $M(\SS_{\pm\ell},\cS)$
as well as for the poset of admissible subsets.
Let $\eta:\pm X\rightarrow\{1,\ldots,n=2\ell\}$ be given by
$$
\eta(i):=
\left\{\begin{array}{ll}
i,&i>0\\
2\ell+1+i,&i<0.\\
\end{array}\right.
$$
Define $\zeta:E\rightarrow E(\ov{T})$ by
$\id_{E_{\geq J}}\mapsto e(J):=\sum_{j\in\eta J}E_{jj}$ for $J\subset\pm X$ admissible,
and $\id_{\pm X}\mapsto I_n$.
Then $\zeta:E\rightarrow E(\ov{T})$ is an isomorphism that
is equivariant with respect to the $\SS_{\pm\ell}$-action on $E$ 
and the $W(C_\ell)$-action on $E(\ov{T})$ (see
\cite{Solomon95}*{Example 5.5}).

Finally, if $e=\id_{E_{\geq J}}$ is an idempotent in $M(\SS_{\pm\ell},\cS)$
and $G=\SS_{\pm\ell}$, then the idempotent stabilizer $G_e$ consists
of those $\pi\in\SS_{\pm\ell}$ that fix the admissible set $J$
pointwise. Similarly, we have $W(C_\ell)_{e\zeta}$ consisting
of those $\pi\theta\in W(C_\ell)$ with $e(J)\pi\theta=e(J)$. This is
also equivalent to $\pi$ fixing $J$ pointwise. We thus have our
isomorphism $M(\SS_{\pm\ell},\cS)\cong R$ by Proposition 
\ref{section:reflectionmonoids:result200} (we could also have used 
Proposition \ref{section:rennermonoids:result1000} but the above is more direct).

\paragraph{A presentation for the Renner monoid of $\msp_{2\ell}$:} 
It remains to take the (P1)-(P4) data for $M(\SS_{\pm\ell},\cS)$
listed in (1) above and apply Theorem \ref{presentations:result1000}.
We have generators 
$s_0,\ldots,s_{\ell-1},a$
with \emph{(Units)\/} relations $(s_is_j)^{m_{ij}}=1$
where the $m_{ij}$ are given by
$$
\begin{pspicture}(0,0)(14,1)
\rput(3,-.5){
\rput(4,1){\BoxedEPSF{fig2.eps scaled 400}}
\rput(1.65,.6){$s_0$}\rput(3,.6){$s_1$}
\rput(4.9,.6){$s_{\ell-2}$}\rput(6.35,.6){$s_{\ell-1}$}
\rput(2.4,1.25){$4$}
}
\end{pspicture}
$$
in the usual way.
The \emph{(Idem1)\/} relation is $a^2=a$, and the \emph{(Idem2)}
relations are 
$$
\alpha(1,\ldots,j_1)\alpha(j_1+1,\ldots,j_2)
=\alpha(j_1+1,\ldots,j_2)\alpha(1,\ldots,j_1),
$$
for all $0\leq j_1<j_2\leq \ell$, with $\aa(I)$ given by (\ref{eq:12}).
The \emph{(Idem3)} relations are 
$$
\aa(I_1)\ldots\aa(I_{k-1})=\aa(I_1)\ldots\aa(I_{k-1})\aa(I)
$$
for $(I_1,\ldots,I_{k-1},I)_f$ arising from $f\in\mathit{Char}_k$,
and
with $(a(I_1),\ldots,$ $a(I_{k-1}))\in\mathit{Ind}_{k-1}$, $(k\geq 2)$ and
all $a(K)\supseteq\bigcap a(I_i)$, where $\mathit{Ind}_{k-1}$ is given by
Proposition \ref{presentations:result375}. The \emph{(RefIdem)\/}
relations consist of three families: 
$$
s_0\,\aa(I)=\aa(I)s_0,
\text{ and }
s_i\,\aa(I)=\aa(I)s_i,
\text{ and }
s_i\,\aa(I)=\aa(I)^{s_i}s_i.
$$
The first is
for all $I\subseteq X$ with $1\not\in I$ (if $1\in I$ then the relations
$s_0\,\aa(I)=\aa(I)^{s_0}s_0$ are vacuous); the
second for $1\leq i\leq \ell-1$ and $i,i+1\in I$ or $i,i+1\not\in I$;
the third when exactly one of $i,i+1$ lies in $I$;
finally,
$\aa(I)^{s_i}=s_i\omega_{i+1}\omega_{i_2}\cdots\omega_{i_k}
a\omega_{i_k}\cdots\omega_{i_2}\omega_{i+1}$
when $I=\{i,i_2,\ldots,i_k\}$, and when $i+1\in I$ is similar. 
Finally, the \emph{(Iso)\/} relations are
$$
\aa(1,I)\aa(2,I)s_1=\aa(1,I)\aa(2,I),
\text{ and }
\aa(I)\aa(1,I)s_0=\aa(I)\aa(1,I),
$$
the first for all $I\subseteq X\setminus\{1,2\}$ and
the second for all $I\subseteq X\setminus\{1\}$.
\end{example}

\begin{example}[the odd dimensional special orthogonal monoids
  $\mso_n$]
\label{eg:oddorthogonal}
This is very similar to the previous case so we will just run
through the answers. Let $n=2\ell+1$ and 
$$
G_0=\so_n=\{g\in\gl_n\,|\,g^TJg=J\}
\text{ for } 
J=\left[\begin{array}{ccc}
0&0&J_0\\
0&1&0\\
-J_0&0&0  
\end{array}\right],
$$
with $J_0$ as in Example \ref{eg:symplectic}. We have taken the definition
of $\so_n$ given in  \cite{Li_Renner03} rather than \cite{Humphreys75} to make
the similiarity with $\sp_n$ more apparent. 
We have $T_0=\so_n\cap\t_n$, the matrices of the form
$\diag(t_1,\ldots,t_\ell,\pm 1,t_\ell^{-1},\ldots,t_1^{-1})$
with the $t_i\in k^\times$; $G=k^\times\so_n$ with
$T$ as before and the orthogonal monoid $\mso_n=\ov{k^\times\so_n}\subset\m_n$.

The roots have type $B_\ell$, so are the same as $C_\ell$ except with
$\pm v_i$ instead of $\pm 2v_i$. Nevertheless, the Weyl group
$W(B_\ell)$ is isomorphic to $W(C_\ell)$ and we take the simple system 
$\Delta=\{v_1,v_{i+1}-v_i\,(1\leq i\leq \ell-1)\}$. 

If $R$ is the Renner monoid of $\mso_n$ then the isomorphism
$M(\SS_{\pm\ell},\cS)\cong R$ is completely analogous to before: the
only changes are that in the isomorphism 
$\theta:\SS_{\pm\ell}\rightarrow W(B_\ell)$ we have 
$(1,-1)\mapsto s_0:=s_{v_1}$ and in the isomorphism 
$\zeta:E\rightarrow E(\ov{T})$ we have 
$\eta:\pm X\rightarrow\{1,\ldots,n=2\ell+1\}$ given by
$$
\eta(i):=
\left\{\begin{array}{ll}
i,&i>0\\
2\ell+2+i,&i<0\\
\end{array}\right.
$$
and $e(J):=E_{\ell+1,\ell+1}+\sum_{j\in\eta J}E_{jj}$ for $J\subset\pm
X$ admissible.

\paragraph{A presentation for the Renner monoid of $\mso_{2\ell+1}$:} 
This is identical to the presentation in the $\msp_{2\ell}$ case of Example 
\ref{eg:symplectic}.
\end{example}

\begin{example}[the even dimensional special orthogonal monoids
  $\mso_n$]
\label{eg:evenorthogonal}
Let $n=2\ell$ and 
$$
G_0=\so_n=\{g\in\gl_n\,|\,g^TJg=J\}
\text{ for } 
J=\left[\begin{array}{cc}
0&J_0\\
J_0&0 
\end{array}\right],
$$
with $J_0$ as in Example \ref{eg:symplectic}
and $\mso_n$ as in Example \ref{eg:oddorthogonal}.

The roots have type $D_\ell$: $\{\pm v_i\pm v_j\, (1\leq i<
j\leq\ell)\}$ with $\Delta=\{v_1+v_2,v_{i+1}-v_i\,(1\leq i\leq
\ell-1)\}$. If $R$ is the Renner monoid of $\mso_{2\ell}$, 
the isomorphism $M(\SS_{\pm\ell}^e,\cS)\cong R$ is built from the
isomorphism $\theta:\SS_{\pm\ell}^e\rightarrow W(D_\ell)$ given 
at the end of \S\ref{section:renner:classical1}
together with 
$\zeta:E\rightarrow E(\ov{T})$ exactly as for $\msp_n$. 

\paragraph{A presentation for the Renner monoid of\/ $\mso_{2\ell}$:} 
We take the (P1)-(P4) data for $M(\SS_{\pm\ell}^e,\cS)$
listed in (2) above and apply Theorem \ref{presentations:result1000}.
We have generators 
$s_0,\ldots,s_{\ell-1},a_1,a_2$
with \emph{(Units)\/} relations $(s_is_j)^{m_{ij}}=1$
where the $m_{ij}$ are given by
$$
\begin{pspicture}(0,0)(14,2)
\rput(3.5,0){
\rput(4,1){\BoxedEPSF{fig3a.eps scaled 400}}
\rput(1.75,.45){$s_0$}\rput(1.75,1.55){$s_{1}$}\rput(2.95,.6){$s_{2}$}
\rput(4.85,.6){$s_{\ell-2}$}\rput(6.25,.6){$s_{\ell-1}$}
}
\end{pspicture}
$$
The \emph{(Idem1)\/} relations are
$a_1^2=a_1,a_2^2=a_2$, 
and the \emph{(Idem2)}
relations are 
$$
\alpha(1,\ldots,j_1)\alpha(j_1+\ve,\ldots,j_2)
=\alpha(j_1+\ve,\ldots,j_2)\alpha(1,\ldots,j_1),
$$
for all $0\leq j_1<j_2\leq \ell$, $\ve=0,1$ and with $\aa(I)$ given
by (\ref{eq:22})-(\ref{eq:23}).
The \emph{(Idem3)} relations are exactly as in the $\msp_n$ case.
The \emph{(RefIdem)\/} relations are the same as for $\msp_n$ for
$s_i\,(1\leq i\leq\ell-1)$; the relations involving $s_0$ are 
slightly different. We get:
$$
s_0\alpha(I)=\alpha(I)s_0,
\text{ and }
s_0\alpha(I)=\alpha(I)^{s_0}s_0.
$$
with the first for all $I$ with $1,2\not\in I$
and the second where at least one (or both) of $1,2$ are in $I$
It is straightforward to give an expression for $\alpha(I)^{s_0}$.
Finally, the \emph{(Iso)\/} relations are
$$
\alpha(1,I)\alpha(2,I)s_1=\alpha(1,I)\alpha(2,I),
$$
for all $I\subseteq X\setminus\{1,2\}$.
\end{example}

\subsection{An example of Solomon}
\label{section:renner:permutohedron}

For the beautiful interplay between group theory
and combinatorics that results, we look at a family of examples 
considered by Solomon in \cite{Solomon95}*{Example 5.7}.
We follow the pattern of the last section, defining first an
algebraic monoid $M$ followed by an abstract monoid of partial isomorphisms
which turns out to be isomorphic to the Renner monoid of $M$. 

Let $G_0=\sl_n$ and $V_0$ the natural
module for $G_0$. Let $\bigwedge^p V_0$ be the $p$-th exterior power
and let
$$
V=\bigotimes_{p=1}^{n-1}\bigwedge^p V_0,\text{ with }
\dim V:=m=\prod_{p=1}^{n-1}\binom{n}{p}.
$$
If $\rho:G_0\rightarrow GL(V)$ is the corresponding representation
then let $M=\ov{k^\times\rho(G_0)}\subset\m_{m}$. Let $R$ be
the Renner monoid of $M$.

Now to a monoid of partial isomorphisms. Take an $n$-dimensional
Euclidean space with basis $\{u_1,\ldots,u_n\}$ and $\SS_n$
acting by $u_i\pi=u_{i\pi}$ for 
$\pi\in\SS_n$. The $(n-1)$-simplex $\Delta^{n-1}$ is the convex hull
of the $u_i$, and as the $\SS_n$-action is linear, it restricts
to an action on $\Delta^{n-1}$. This is just the action of the group
of reflections and rotations of $\Delta^{n-1}$. In particular, if $O$
is an admissible partial orientation of $\Delta^{n-1}$ as in Example
\ref{eg:permutohedron} of \S\ref{section:idempotents:polytopes}, then
it is clear that the image $O\pi$ is also admissible. Consider the
induced $\SS_n$-action on the set $E_0$ of admissible partial
orientations and extend it to the poset $E$ of Example
\ref{eg:permutohedron} by defining $\1\pi=\1$ for all $\pi\in\SS_n$. 
This action is clearly by poset isomorphisms. 

Thus the collection of intervals $E_{\geq O}=\{O'\in E\,|\,O\leq O'\}$ forms a
system $\cS$ of subsets of $E$ for $\SS_n$ with $M(\SS_n,\cS)$ the
corresponding monoid of partial isomorphisms. 

\paragraph{The isomorphism $M(\SS_n,\cS)\cong R$.}
By Proposition \ref{section:rennermonoids:result1000} it suffices to
establish an isomorphism from $M(\SS_n,\cS)$ to $M(W,\cS_P)$ where $W$
is the Weyl group of $G$ (or $G_0$) and $P$ is the polytope described
in \S\ref{section:renner:general}. 

The Weyl group is $W(A_{n-1})$ and we take 
$\theta:\SS_n\rightarrow W(A_{n-1})$ the standard isomorphism given by 
$(i,i+1)\mapsto s_i:=s_{v_{i+1}-v_i}$. 

We now describe $P$, following \cite{Solomon95}*{Example 5.7}. 
It turns out to be convenient to describe another abstract polytope
$P'$ first, and then relate this back to the $P$ we are interested
in. 
Let $X=\{1,\ldots,n\}$ and $\tau=\{J_1,\ldots,J_{n-1}\}$ be a collection
of subsets of $X$ with $|J_i|=i$. Thus, $\tau$ contains exactly one
non-empty proper set of each possible cardinality. 
Let $\Sigma$ be the set of all such $\tau$. Given $\tau\in\Sigma$, let
$a_j$ be the number of $J_i$ in which $j$ occurs, and let
$v_\tau=(a_1,\ldots,a_n)^T\in\R^n$. 

\begin{proposition}
\label{section:reflectionmonoids:result1000}
The convex hull $P'$ of the $v_\tau$, for $\tau\in\Sigma$, is the
$(n-1)$-permutohedron having the
parameters $m_1,\ldots,m_n=0,\ldots,n-1$.
\end{proposition}

\begin{proof}
We start with some elementary observations:
\begin{description}
\item[(i).] if $\pi\in\SS_n$ then
$\tau\pi:=\{J_1\pi,\ldots,J_{n-1}\pi\}\in\Sigma$ with $v_{\tau\pi}=(a_{1\pi^{-1}},\ldots,a_{n\pi^{-1}})$;
\item[(ii).] we have $\ov{\tau}:=\{X\setminus J_1,\ldots,X\setminus  J_{n-1}\}\in\Sigma$ with 
$v_{\ov{\tau}}=(\ov{a}_1,\ldots,\ov{a}_n)$ for $\ov{a}_j=n-a_j-1$;
\item[(iii).] let $\tau\in\Sigma$ and suppose that for some $i$ we
  have $j<j'$ with $j\in  J_i$ and 
$j'\not\in J_i$. Let $\tau'$ be such that $J_i'=\{j'\}\cup(J_i\setminus\{j\})$ and
all other $J'_j$ the same as in $\tau$. Then $\tau'\in\Sigma$. We write
$\tau\vdash\tau'$ to denote this move.
\end{description}
If $\tau_0$ is such that $J_i=\{1,\ldots,i\}$ then $v_{\ov{\tau}_0}=(0,1,\ldots,n-1)$, so that $\Sigma$
contains all the permutations of this vector, and 
the $(n-1)$-permutohedron is contained in $P'$.
The reverse inclusion is established by showing that the $v_\tau$ are contained in the 
$(n-1)$-permutohedron. If $v_\tau=(a_1,\ldots,a_n)^T$ then by 
\cite{Gaiha:Gupta77}*{Theorem 2} it suffices to show 
for all $Y\subseteq X$ with $|Y|=k$ that
$\sum_{i\in Y}a_i\geq 0+1+\cdots+(k-1)$.
Equivalently, by applying the involution $\tau\mapsto\ov{\tau}$, we show that 
${\textstyle\sum_{i\in Y}a_i\leq (n-1)+(n-2)+\cdots+(n-k)}$.
By permuting and relabeling we can assume that $Y=\{1,\ldots,k\}$; 
thus it remains to show for all $v_\tau=(a_1,\ldots,a_n)$ and all $k$ that
\begin{equation}
  \label{eq:16}
  a_1+\cdots+a_k\leq (n-1)+\cdots+(n-k). 
\end{equation}
Consider first the $\tau_0$ given above
with $v_{\tau_0}=(n-1,\ldots,1,0)$. Then this clearly satisfies (\ref{eq:16}). 
If $\tau\vdash\tau'$ and $\tau$ satisfies (\ref{eq:16}) then so does $\tau'$.
For $\tau\in\Sigma$ compare the subsets $J_i=\{1,\ldots,i\}$ of $\tau_0$ and
$J_i'$ of $\tau$. Then there is a 1-1 correspondence between the $1\leq j\leq i$
that are not in $J'_i$ and the $i+1\leq j\leq n$ that are in $J'_i$. Working
through the $i$, we get a sequence of moves
$\tau_0\vdash\cdots\vdash\tau$, and hence that 
$\tau$ satisfies (\ref{eq:16}) as required. 
\qed
\end{proof}

The polytope $P'$ is not quite the $P$ described in \S\ref{section:renner:general}.
To get it back, we need to compute the columns of the matrix $A$ whose rows
$\mathbf{a}_i=(a_{i1},\ldots,a_{im})$
are given by (\ref{eq:24}).
Recall the simple roots $v_{p+1}-v_p$ from Example 
\ref{eg:general} and let
$$
(v_{p+1}-v_p)^\vee(t)=\diag(1,\ldots,t,t^{-1},\ldots,1)\text{ for }1\leq p\leq n-1
$$
be the corresponding coroots with the $t$ in the $p$-th position. If 
$v_\tau=(a_1,\ldots,a_n)^T$ arises from $\tau=\{J_1,\ldots,J_{n-1}\}$
with $J_1=\{i\},J_2=\{j,k\},\ldots,$
$J_{n-1}=\{1,\ldots,\widehat{q},\ldots,n\}$,
then $V$ has basis the $v$ of the form
$$
  v=v_i\otimes(v_j\wedge v_k)\otimes\cdots
\otimes(v_1\wedge\cdots\wedge\widehat{v}_q\wedge \cdots\wedge v_n),
$$
as $\tau$ ranges over $\Sigma$ and where $\{v_1,\ldots,v_n\}$ is a basis for $V_0$. Then
$$
\rho(v_{p+1}-v_p)^\vee(t)v=t^{a_p-a_{p+1}}v,
$$
and so the columns of $A$ are 
the $(a_1-a_2,\ldots,a_{n-1}-a_n)^T$. In particular the map
$
(x_1,\ldots,x_n)^T\mapsto (x_1-x_2,\ldots,x_{n-1}-x_n)^T
$
sends the permutohedron $P'$ of Proposition
\ref{section:reflectionmonoids:result1000} to the polytope $P$
described in \S\ref{section:renner:general}.

Let $O$ be an admissible partial orientation of $\Delta^{n-1}$ and
$O\mapsto f'_O$ be the isomorphism $E\rightarrow\FF(P')$ of
Proposition \ref{presentations:result400}, with
$m_1,\ldots,m_n=0,\ldots,n-1$. The map $\R^n\rightarrow \R^{n-1}$
given by $(x_1,\ldots,x_n)$ $\mapsto (x_1-x_2,\ldots,x_{n-1}-x_n)$
induces an isomorphism $\FF(P')\rightarrow\FF(P)$ which we write as
$f'_O\mapsto f_O$. Finally, let $\zeta$ send the partial identity on
the interval $E_{\geq O}$ of $E$ to the partial identity on the
interval $E_{\geq f_O}$ of $\FF(P)$. 

That $\zeta$ is equivariant and $\theta$ preserves idempotent
stabilizers (which actually turn out to be trivial)
we leave to the reader, although we supply the following
hint: the vertices of $P$ can be labeled (in a one to one fashion) by
the $g\in W(A_{n-1})$ and the edges can be labeled by the $s_i$ so
that there is an $s_i$-labeled edge connecting $g$ to $g'$ if and only
if $g'=gs_i$ (in the
language of \cite{Everitt}*{Chapter 3}, the $1$-skeleton of $P$ is the
universal cover, or Cayley graph, of the presentation $2$-complex of
$W$ with respect to its presentation as Coxeter group). The action of
$W(A_{n-1})$ on the vertices of $P$ can then be described as follows:
if $g=s_{i_1}\ldots s_{i_k}\in W(A_{n-1})$ and $v$ is the vertex of
$P$ labeled by the identity, then let $v'$ be the terminal vertex of a
path starting at $v$ and with edges labeled
$s_{i_1},\ldots,s_{i_k}$. For any vertex $u$, let $s_{j_1}\ldots s_{j_\ell}$ 
be the label of a path from $v$ to $u$, and let $u'$ be
the terminal vertex of a path starting at $v'$ and with label 
$s_{j_1}\ldots s_{j_\ell}$. Then $g$ maps $u$ to $u'$ (and in particular
$v$ to $v'$). In the language of \cite{Everitt}*{Chapter 4}, the $W(A_{n-1})$-action
is as the Galois group of the covering of $2$-complexes. 

Define $\varphi:M(\SS_n,\cS)\rightarrow M(W(A_{n-1}),\cS_P)$ as in 
Proposition \ref{section:reflectionmonoids:result200}.

\paragraph{Presentation data for the monoid $M(\SS_n,\cS)$.}
The atoms are the partial orientations $a_J$ from \S\ref{section:idempotents:polytopes}
for $J$ a non-empty proper subset of $X=\{1,\ldots,n\}$. The
$\SS_n$-action on the partial orientations induces an action on the
atoms given by $a_J\cdot\pi=a_{J\pi}$ for $\pi\in\SS_n$. Thus, we just
have the action of $\SS_n$ on the subsets of $X$, so for the
representatives $O_k$ we can appeal to
\S\ref{section:renner:interlude}.

The set $\mathit{Char}_1$ corresponds to the $n_0$ with $0\leq n_0\leq
n$,
and we take $O_1=\{a_1,\ldots,a_{n-1}\}$ with $a_i:=a_{\{1\ldots,i\}}$. The
absence of an $a_0$ and $a_n$ is because we have restricted to the
action on the non-empty proper subsets of $X$.
The
set $\mathit{Char}_2$ corresponds to the $n_0,n_1,n_2$ such that $0\leq
n_0\leq n_1,n_2$ with $n_1+n_2-n_0\leq n$ and $0<n_i<n$. From
\S\ref{section:renner:interlude} we get a tuple $(J,K)$ where
$$
J=\{1,\ldots,n_0\}\cup\{n_0+1,\ldots,n_1\}
\text{ and }
K=\{1,\ldots,n_0\}\cup\{n_1+1,\ldots,n_1+n_2-n_0\}
$$
are representatives for the corresponding orbit. Thus we take
$O_2$ to be the pairs $\{a_J,a_K\}$.
The set $\mathit{Char}_3$ corresponds to the $n_0,\ldots,n_3,n_{ij}$
satisfying the conditions on the far right of Figure
\ref{fig;characteristic.maps}, together with $0<n_{ij}<n$. We get a
corresponding tuple $(J_1,J_2,J_3)$ using the scheme $(\dag)$ at the
end of \S\ref{section:renner:interlude}, and we take $O_3$ to be the 
the set of $\{a_{J_1},a_{J_2},a_{J_3}\}$.

For $J$ a non-empty proper subset of $X$, fix an element
$w_J\in\SS_n$ with $Jw_J=\{1,\ldots,|J|\}$ and let 
\begin{equation}
  \label{eq:10}
  \aa_J:=\omega_J a_k \omega_J^{-1}.
\end{equation}
It turns out that for an arbitrary $O\in E$ we do not require
an expression in the atoms for $O$, except in the case $O=\1$, the
formally adjoined unique maximal element. We take $\1:=\bigvee a_J$,
the join over all the atoms, i.e.: over all non-empty proper subsets
$J$ of $X$.

Finally we have the set $\mathit{Iso}$. The set $T$ consists of the
transpositions $(i,j)\in\SS_n$ and for $t\in T$, $H_t=\{O\in
E_0\,|\,Ot=O\}\cup\{\1\}$ with $\AA=\{H_t\,|\,t\in T\}$. There is a single $\SS_n$-orbit on $\AA$
with representative $H_1:=H_{s_1}$ where $s_i:=(i,i+1)$. 
The $O$ are the partial admissible orientations of $\Delta^{n-1}$, and
one such is fixed by $s_1$ exactly when the edge joining $v_1$ and
$v_2$ is not in $O$, and for all $i>2$, the edge joining $v_1$ and
$v_i$ lies in $O$ if and only if the edge joining $v_2$ and $v_i$ lies in $O$. 
We want $O$ minimal with the property that $H_1\supseteq E_{\geq O}$.
But if $O<\1$ then the interval $E_{\geq O}$ contains an admissible
partial orientation in which all the edges of $\Delta^{n-1}$ are
oriented (i.e.: a total order). Thus, $E_{\geq O}$ contains an $O'$ in
which the edge joining $v_1$ and $v_2$ is oriented, and so
$O's_1\not=O'$. The result is that $H_1\not\supseteq E_{\geq O}$. 

The only element of $E$ then that is minimal with $H_1\supseteq E_{\geq O}$ is
$\1$, and $\mathit{Iso}$ consists of the single pair $(\1,s_1)$. 

\begin{example}[the presentation for the Renner monoid of $M$]
We have generators 
$s_1,\ldots,s_{n-1}$ and $a_1,\ldots,a_{n-1}$
with \emph{(Units)\/} relations $(s_is_j)^{m_{ij}}=1$
where the $m_{ij}$ are given by the symbol
$$
\begin{pspicture}(0,0)(14,1)
\rput(3,-0.5){
\rput(4,1){\BoxedEPSF{fig2.eps scaled 400}}
\rput(1.65,.6){$s_1$}\rput(3,.6){$s_2$}
\rput(4.9,.6){$s_{n-2}$}\rput(6.35,.6){$s_{n-1}$}
}
\end{pspicture}
$$
The \emph{(Idem1)\/} relations are $a_i^2=a_i\,(1\leq i\leq n-1)$ and the \emph{(Idem2)\/}
relations are $\aa_J\aa_K=\aa_K\aa_L$ where 
the $\aa$'s are given by (\ref{eq:10}),
$J=\{1,\ldots,n_0\}\cup\{n_0+1,\ldots,n_1\}$, 
$K=\{1,\ldots,n_0\}\cup\{n_1+1,\ldots,n_1+n_2-n_0\}$ and 
$0\leq n_0\leq n_1,n_2$ are such that $n_1+n_2-n_0\leq n$
and $0<n_i<n$.

The presentation for the permutohedron from
\S\ref{section:idempotents:polytopes}
gives \emph{(Idem3)\/} relations
$\aa_{J_1}\aa_{J_2}=\aa_{J_1}\aa_{J_2}\aa_{J_3}$ for all 
$\{a_{J_1},a_{J_2},a_{J_3}\}\in O_3$ such that 
$J_1\not=J_1\cap J_2\not=J_2$; that is,
$n_1-n_0$, $n_{13}-n_1-n_3+n_0$ are not both zero, and 
$n_2-n_0$, $n_{23}-n_2-n_3+n_0$ are not both zero.

The \emph{(RefIdem)\/} are $s_i\aa_J=\aa_{Js_i}s_i$ for
$1\leq i\leq n-1$ and $J$ a non-empty proper subset of $X$. Finally
the \emph{(Iso)\/} are the single relation
$$
\prod\aa_J\cdot s_1=\prod\aa_J,
$$
where the product is over all proper non-empty subsets $J$ of $X$.
\end{example}




\end{document}